\documentclass[10pt]{amsart}

\usepackage[pagebackref,colorlinks=true]{hyperref}  %za cvetni linkove (obache bavi)

\usepackage{amsfonts}
\usepackage{amsmath,amsthm,amssymb,amscd,enumerate,eucal,url}

\usepackage{eucal,url,amssymb,verbatim,
enumerate,amscd,
}

\allowdisplaybreaks[1]

\numberwithin{equation}{section}

\newtheorem{thrm}{Theorem}[section]
\newtheorem{lemma}[thrm]{Lemma}

\usepackage[margin=1in]{geometry}
\newcommand{\R}{\mathbb{R}}
\newcommand{\lc}{\langle}
\newcommand{\rc}{\rangle}
\newcommand{\br}[1]{\overline{#1}}
\newcommand{\dt}{\,.\,}
\newcommand{\spc}{\ \,}

\begin{document}

\begin{abstract} 
We exploit the Cartan-K\"ahler theory to prove the
local existence of real analytic quaternionic contact structures for any
prescribed values of the respective curvature functions and their covariant
derivatives at a given point on a manifold.  We show that, in a certain
sense, the different real analytic quaternionic contact geometries in $4n+3$
dimensions depend, modulo diffeomorphisms, on $2n+2$ real analytic functions
of $2n+3$ variables.  
\end{abstract}



\title[On the existence of local quaternionic contact geometries]
{On the existence of local quaternionic contact geometries}

\date{\today}

\author{Ivan Minchev}
\address[Ivan Minchev]{University
of Sofia, Faculty of Mathematics and Informatics, blvd. James Bourchier 5, 1164 Sofia, Bulgaria}
\email{minchev@fmi.uni-sofia.bg}

\author{Jan Slov\'{a}k}
\address[Jan Slov\'{a}k]{
Department of Mathematics and Statistics, Masaryk University, Kotlarska 2, 61137 Brno,
Czech Republic}
\email{slovak@muni.cz}

\maketitle

\setcounter{tocdepth}{2} \tableofcontents

%%%%%%%%%%
\section{Introduction}

The quaternionic contact (briefly: qc) structures are a rather recently
developed concept in the differential geometry that has proven to be a very
useful tool when dealing with a certain type of analytic problems
concerning the extremals and the choice of a best constant in the $L^2$
Folland-Stein inequality on the quaternionic Heisenberg group \cite{IMV2},
\cite{IMV}, \cite{IMV3}.  Originally, the concept was introduced by O. 
Biquard \cite{Biq1}, who was partially motivated by a preceding result of C. 
LeBrun \cite{LeB} concerning the existence of a large family of complete
quaternionic K\"ahler metrics of negative scalar curvature, defined on the
unit ball $B^{4n+4}\subset\mathbb R^{4n+4}$.  By interpreting $B^{4n+4}$ as
a quaternionic hyperbolic space, $B^{4n+4}\cong Sp(n+1,1)/Sp(n)\times
Sp(1)$, LeBrun was able to construct deformations of the associated twistor
space $\mathcal Z$---a complex manifold which, in this case, is
biholomorphically equivalent to a certain open subset of the complex
projective space $\mathbb C\mathbb P^{2n+3}$---that preserve its induced
contact structure and anti-holomorphic involution, and thus can be pushed
down to produce deformations of the standard (hyperbolic) quaternionic
K\"ahler metric of $B^{4n+4}$.  The whole construction is parametrized by
an arbitrary choice of a sufficiently small holomorphic function of ${2n+3}$
complex variables and the result in \cite{LeB} is that the moduli space of
the so arising family of complete quaternionic K\"ahler metrics on $B^{4n+4}$
is infinite dimensional.  LeBrun also observed that, if multiplied by a
function that vanishes along the boundary sphere $S^{4n+3}$ to order two,
the deformed metric tensors on $B^{4n+4}$ extend smoothly across $S^{4n+3}$
but their rank drops to $4$ there.  

It was discovered later by Biquard
\cite{Biq1} that the arising structure on $S^{4n+3}$ is essentially given by
a certain very special type of a co-dimension 3 distribution which he
introduced as a qc structure on $S^{4n+3}$ and called the conformal boundary
at infinity of the corresponding quaternionic K\"ahler metric on $B^{4n+4}$. 
Biquard proved also the converse \cite{Biq1}: He showed that each real
analytic qc structure on a manifold $M$ is the conformal boundary at
infinity of a (germ) unique quaternionic K\"ahler metric defined in a small
neighborhood of $M$.  Therefore, already by the very appearance of the new
concept of a qc geometry, it was clear that there exist infinitely many
examples---namely, the global qc structures on the sphere $S^{4n+3}$
obtained by the LeBrun's deformations of $B^{4n+4}$.  

However, the number of
the explicitly known examples remains so far very restricted.  There is
essentially only one generic method for obtaining such structures
explicitly.  It is based on the existence of a certain very special type of
Riemannian manifolds, the so called 3-Sasaki like spaces.  These are
Riemannian manifolds that admit a special triple $R_1, R_2,R_3$ of Killing
vector fields, subject to some additional requirements (we refer to
\cite{IMV3} and the references therein for more detail on the topic), which
carry a natural qc structure defined by the orthogonal complement of the
triple $\{R_1, R_2,R_3\}$.  There are no explicit examples of qc structures
(not even locally) for which it is proven that they can not be generated by
the above construction.

The formal similarity with the definition of a CR (Cauchy-Riemann) manifold,
considered in the complex analysis, might suggests that one should look for
new examples of qc structures by studding hypersurfaces in the quaternionic
coordinate space $\mathbb H^{n+1}$.  This idea, however, turned out to be
rather unproductive in the quaternionic case: In \cite{IMV5} it was shown
that each qc hypersurfaces embedded in $\mathbb H^{n+1}$ is necessarily
given by a quadratic form there and that all such hypersurfaces are locally
equivalent, as qc manifolds, to the standard (3-Sasaki) sphere.

In the present paper we reformulate the problem of local existence of qc
structures as a problem of existence of integral manifolds of an appropriate
exterior differential system to which we apply methods from the
Cartan-K\"ahler theory and show its integrability.  The definition of the
respective exterior differential system is based entirely on the formulae
obtained in \cite{MS} for the associated canonical Cartan connection and its
curvature.  We compute explicitly the relevant character sequence $v_{1},
v_{2},\dots$ of the system (cf.  the discussion in Section~\ref{ex_dif_sys})
and show that it passes the so called Cartan's test, i.e., that the system
is in involution.  From there we obtain our main result in the paper---this
is Theorem~\ref{main}---that asserts the local existence of qc structures
for any prescribed values of their respective curvatures and associated
covariant derivatives at any fixed point on a manifold.  

Furthermore, since
the last non-zero character of the associated exterior differential system
is $v_{2n+3}=2n+2$, we obtain a certain description for the associated
moduli spaces.  Namely, we have that, in a certain sense (the precise
formulation requires care, cf.  \cite{Bry}), the real analytic qc structures
in $4n+3$ dimensions depend, modulo diffeomorphisms, on $2n+2$ functions of
$2n+3$ variables.  Comparing our result to the LeBrun's family of qc
structures on the sphere $S^{4n+3}$---parametrized by a single holomorphic
function of $2n+3$ complex variables (which has the same generality as two
real analytic functions of $2n+3$ real variables)---we observe that it
simply is not "big enough" in order to provide a local model for all
possible qc geometries in dimension $4n+3$.

{\bf Acknowledgments.} I.M. is supported by Contract DFNI
I02/4/12.12.2014 and Contract 80-10-33/2017 with the Sofia University
"St.Kl.Ohridski".  
%I.M.  is also supported by a SoMoPro II Fellowship which
%is co-funded by the European Commission\footnote{This article reflects only
%the author's views and the EU is not liable for any use that may be made of
%the information contained therein.} from \lq\lq{}People\rq\rq{} specific
%program (Marie Curie Actions) within the EU Seventh Framework Program on the
%basis of the grant agreement REA No.  291782.  It is further co-financed by
%the South-Moravian Region.  
J.S. is supported by the grant P201/12/G028 of
the Grant Agency of the Czech Republic.

%%%%%%%%%%
\section{Quaternionic contact structures as integral manifolds of exterior differential systems}

Our work has been inspired and heavily influenced by the series of lectures by
Robert Bryant at the Winter School Geometry and Physics in Srn\'{\i}, January
2015, essentially along the lines of \cite{Bry}. In particular our description
of the qc structures in the following paragraphs follows this
source closely. 

\subsection{Exterior differential systems}\label{ex_dif_sys} 
In general, an exterior differential system is a graded differentially
closed ideal $\mathcal I$ in the algebra of differential forms on a
manifold $N$. Integral manifolds of such a system are immersions $f:M\to N$
such that the pullback $f^*\alpha$ of any form $\alpha \in\mathcal I$ 
vanishes on $M$. Typically, the differential ideal $\mathcal I$ encounters
all differential consequences of a system of partial differential equations
and understanding the algebraic structure of $\mathcal I$ helps to
understand the structure of the solution set. We need a special
form of exterior differential systems corresponding to the 
geometric structures modelled on homogeneous spaces, the so called Cartan
geometries. This means our system will be generated by one-forms forming the
Cartan's coframing intrinsic to a geometric structure and its differential
consequences (the curvature and its derivatives).

For this paragraph, we adopt the following ranges of indices: 
$1\le a,b,c,d,e\le n$, $1\le s\le l$, where $l$
and $n$ are some fixed positive integers.

We consider the following general problem: \emph{Given a set of real 
analytic functions $C^a_{bc}:\mathbb R^l\rightarrow\mathbb R$ 
with $C^a_{bc}=-C^a_{cb}$, find linearly independent one-forms $\omega^a$, 
defined on a domain $\Omega\subset\mathbb R^n$, 
and a mapping $u=(u^s):\Omega \rightarrow \mathbb R^l$ so that the equations
\begin{equation}\label{str_eq_d_omega}
d\omega^a=-\frac{1}{2}C^a_{bc}(u)\,\omega^b\wedge\omega^c
\end{equation} 
are satisfied everywhere on $\Omega$.}
 
The  problem is diffeomorphism invariant in the sense that 
if $(\omega^a, u)$ is any solution of \eqref{str_eq_d_omega} 
defined on  $\Omega\subset \mathbb R^n$ and 
$\Phi: \Omega^{'}\rightarrow \Omega$ is a diffeomorphism, 
then $(\Phi^\ast(\omega^a),\Phi^\ast(u))$ is a solution 
of \eqref{str_eq_d_omega} on $\Omega^{'}$. 
We regard any such two solutions as equivalent and 
we are interested in the following question: 
\emph{How many non-equivalent solutions does a given problem of this 
type admit?}    

Next, we reformulate this into a question on solutions to an exterior differential
system.
Let $N= GL(n,\mathbb R)\times\mathbb R^n\times\mathbb R^l$ and denote by
$p=(p^a_b):N\rightarrow GL(n,\mathbb R)$, $x=(x^a):N\rightarrow \mathbb R^n$
and $u=(u^s):N\rightarrow\mathbb R^l$ the respective projections.  Setting
$\omega^a\overset{def}{\ =\ }p^a_b\,dx^b$,
we consider the differential ideal $\mathcal I$ on $N$ generated by the 
set of two-forms
\begin{equation*}
\Upsilon^a\overset{def}{=}d\omega^a+\frac{1}{2}C^{a}_{bc}(u)\,\omega^b\wedge\omega^c.
\end{equation*}

Then, the solutions of \eqref{str_eq_d_omega} are precisely the
$n$-dimensional integral
manifolds of $\mathcal I$ on which the restriction of the $n$-form
$\omega^1\wedge\dots\wedge\omega^n$ is nowhere vanishing. %; the triple
%$(N,\mathcal I, \omega^1\wedge\dots\wedge\omega^n)$ is a an example of an
%exterior differential system.  
The reformulation of the problem
\eqref{str_eq_d_omega} in this setting allows for an easy access of tools
from the Cartan-K\"ahler theory.  We shall see, we may restrict our attention to a
certain set of sufficient conditions for the integrability of the system,
known as the Cartan's Third Theorem, and refer the reader to \cite{Bry} or
\cite{BCGGG} and the references therein for a more detailed and 
general discussion on the topic.

Differentiating \eqref{str_eq_d_omega} gives
\begin{equation}\label{str_eq_d2_omega}
\begin{aligned}
0&=d^2\omega^a=-\frac12 d(C^a_{bc}\omega^b\wedge\omega^c)
\\ &=
-\frac{1}{2}\frac{\partial C^a_{bc}(u)}{\partial u^s}\,du^s\wedge\omega^b\wedge\omega^c+
\frac{1}{3}\bigl(C^a_{be}(u)\,C^e_{cd}(u)+C^a_{ce}(u)\,C^e_{db}(u)+C^a_{de}(u)\,C^e_{bc}(u)
\bigr)\,\omega^b\wedge\omega^c\wedge\omega^d.
\end{aligned}
\end{equation}
If $C^a_{bc}$ were curvature functions of a Cartan connection, then 
these differential consequences are governed by the well known
Bianchi identities, and they are then quadratic.

\subsection{Assumptions and conclusions}\label{difsys_conclusions}
In order to employ the Cartan-K\"ahler theory we need to replace the
quadratic terms by some linear objects.  
Thus we posit the following two
assumptions:

\smallskip  
\noindent{\bf Assumption I: }\emph{Let us assume that there exist a real analytic
mapping $F=(F^s_a):\mathbb R^l\rightarrow \mathbb R^{ln}$ for which 
%the second term on the RHS of \eqref{str_eq_d2_omega} can be expressed as
\begin{equation}\label{Assumption_I}
d\bigl(C^a_{bc}\omega^b\wedge \omega^c\bigr) = \frac{\partial
C^a_{bc}(u)}{\partial u^s}\bigl(du^s + F^s_d \omega^d
\bigr)\wedge\omega^b\wedge \omega^c
.\end{equation}}

Of course, this assumption is equivalent to the requirement
\begin{equation*}
\frac{1}{3}\bigl(C^a_{be}(u)\,C^e_{cd}(u)+C^a_{ce}(u)\,C^e_{db}(u)+C^a_{de}(u)\,C^e_{bc}(u)
\bigr)\,\omega^b\wedge\omega^c\wedge\omega^d\ 
 =\ -\frac{1}{2}\frac{\partial C^a_{bc}(u)}{\partial u^s}\,F^s_d(u)\,\omega^b\wedge\omega^c\wedge\omega^d.
\end{equation*}
and then, on the integral manifolds of $\mathcal I$, 
\eqref{str_eq_d2_omega} takes the form 
\begin{equation}\label{0_eq_d2omega}
0=d^2\omega^a=-\frac{1}{2}\frac{\partial C^a_{bc}(u)}{\partial u^s}\,\Big(du^s+F^s_d(u)\,\omega^d\Big)\wedge\omega^b\wedge\omega^c.
\end{equation}

\smallskip
{\bf Assumption II: } \emph{Interpreting \eqref{0_eq_d2omega} as a system of
algebraic equations for the unknown one-forms $du^s$ (for a fixed $u$), we
assume that it is non-degenerate, i.e, that \eqref{0_eq_d2omega} yields
$du^s\in \text{span}\{\omega^a\}$.  
As a consequence, at any $u$, the set of all
solutions $du^s$ is parametrized by a certain vector space (since the
system \eqref{0_eq_d2omega} is linear).  We will assume that the dimension
of this vector space is a constant $D$ (independent of $u$).}

Let us take the latter two assumptions as granted in the rest of this
paragraph.
Since $\mathcal I$ is a differential ideal, it is algebraically generated by
the forms $\Upsilon^a$ and $d\Upsilon^a$.  By \eqref{Assumption_I}, we have
\begin{equation}
2d\Upsilon^a =d\Big(C^{a}_{bc}(u)\,\omega^b\wedge\omega^c\Big)=
\frac{\partial C^a_{bc}(u)}{\partial
u^s}\,\Big(du^s+F^s_d(u)\omega^d\Big)\wedge\omega^b\wedge\omega^c+2C^a_{bc}\,\Upsilon^b\wedge\omega^c 
\end{equation}
and therefore, $\mathcal I$ is algebraically generated by $\Upsilon^a$ and
the three-forms 
\begin{equation*} \Xi^a\overset{def}{\ =\ }\frac{\partial
C^a_{bc}(u)}{\partial
u^s}\,\Big(du^s+F^s_d(u)\omega^d\Big)\wedge\omega^b\wedge\omega^c. 
\end{equation*}

If we take  $\Omega^a$ to be some other basis of one-forms  for the vector space span$\{\omega^a\}$, we can express the forms $\Xi^a$  as
\begin{equation}
\Xi^a=\Pi^a_{bc}\wedge\Omega^b\wedge\Omega^c,
\end{equation}
where $\Pi^a_{bc}$ are linear combinations of the linearly independent
one-forms 
$$
\{du^s+F^s_d(u)\omega^d\ :\ s=1,\dots,n\}
.$$

Consider the  sequence $v_1(u),v_2(u),\dots,v_n(u)$ of non-negative integers
defined, for any fixed $u$, as $v_1(u)=0$,  
\begin{equation*}
%\begin{split}
v_d(u)=\text{rank}\Big\{\Pi^a_{bc}(u)\ :\ a=%&
1,\dots,n,\ 1\le b<c\le d\Big\}\\
%&
-\ \text{rank}\Big\{\Pi^a_{bc}\ :\ a=1,\dots,n,\ 1\le b<c\le d-1\Big\};
%\end{split}
\end{equation*}
for $1< d\le n-1$, and 
\begin{equation*}
v_n(u)=l-\ \text{rank}\Big\{\Pi^a_{bc}\ :\ a=1\dots n,\ 1\le b<c\le n-1\Big\}.
\end{equation*}

If, for every $u\in \mathbb R^l$, one can find a basis $\Omega^a$ of
span$\{\omega^a\}$ for which the Cartan's Test
\begin{equation} \label{Cartan_test_general}
v_1(u)+2v_2(u)+\dots+nv_n(u)=D,
\end{equation}
is satisfied (remind $D$ is the constant dimension from the above Assumption
II), then the system \eqref{str_eq_d_omega} is said to be \emph{in
involution} (this method of computation for the Cartan's sequence of an
ideal is based on \cite{BCGGG}, Proposition 1.15).  It is an important
result of the theory of exterior differential systems (essentially due to
Cartan, cf.  \cite{Bry}) that if the system is in involution, then for any
$u_0$, there exists a solution $(\omega^a, u)$ of \eqref{str_eq_d_omega}
defined on a neighborhood $\Omega$ of $0\in \mathbb R^n$ for which
$u(0)=u_0$ and 
$$
{du^s}|_{0}=F^s_d(u_0){\omega^d}|_{0}
.$$ 
Moreover, in certain sense (see again \cite{Bry} for a more precise
formulation), the \emph{different solutions $(\omega^a, u)$ of
\eqref{str_eq_d_omega}, modulo diffeomorphisms, depend on $v_k(u)$ functions
of $k$ variables, where $v_k(u)$ is the last non-vanishing integer in the
Cartan's sequence $v_1(u),\dots, v_n(u)$.}

The geometric significance of the above is quite clear: Assume that we are
interested in a geometric structure of a certain type that can be
characterized by a unique Cartan connection.  Then, the structure equations
of the corresponding Cartan connection are some equations of type
\eqref{str_eq_d_omega} involving the curvature of the connection.  The
solutions of the so arising exterior differential system are precisely the
different local geometries of the fixed type that we are considering.

\subsection{Quaternionic contact manifolds}

Let $M$ be a $(4n+3)$-dimensional manifold and $H$ be a smooth distribution on $M$ of codimension three. 
The pair $(M,H)$ is said to be a quaternionic contact (abbr. qc) structure if around each point of $M$ there exist 1-forms $\eta_1,\eta_2,\eta_3$ with common kernel $H,$ a positive definite inner product $g$ on $H,$ and endomorphisms $ I_1, I_2, I_3$ of $H,$ satisfying
\begin{gather}\label{def-hat-eta}
(I_1)^2=( I_2)^2=( I_3)^2=-\text{id}_H,\qquad I_1\,I_2=- I_2\, I_1=I_3,\\\nonumber
d\eta_s(X,Y)=2 g(I_sX,Y) \qquad \ \text{ for all }\ X,Y\in H.
\end{gather}

As  shown in \cite{Biq1}, if $\dim(M)>7$, one can always find, locally, a triple $\xi_1,\xi_2,\xi_3$  of  vector fields on $M$ satisfying for all $X\in H$,
\begin{equation}\label{Reeb}
\eta_s(\xi_t)=\delta^s_t,\qquad\eta_s(\xi_t,X)=-d\eta_t(\xi_s,X)
\end{equation}
($\delta^s_t$ being the Kronecker delta). $\xi_1,\xi_2,\xi_3$ are called Reeb vector fields corresponding to $\eta_1,\eta_2,\eta_3$. In the seven dimensional case the existence of Reeb vector fields is an additional integrability condition on the qc structure (cf. \cite{D}) which we will assume to be satisfied.  

It is well known that the qc structures represent a very interesting
instance of the so called parabolic geometries, i.e. Cartan geometries
modelled on $G/P$ with $G$ semisimple and $P\subset G$ parabolic. The above
definition is a description of these geometries with the additional
assumption that their harmonic torsions vanish. 

As the authors showed in \cite{MS}, the canonical Cartan connection with the
properly normalized curvature can be computed explicitly, including
closed formulae for all its curvature components and their covariant
derivatives. This provides the complete background for viewing the
structures as integral manifolds of an appropriate exterior differential
system (cf. paragraph \ref{qc_as_EDS}), running the Cartan
test, checking the involution of the system, and concluding the generality of
the structures in question (the section \ref{sec3} below). 

For the convenience of the readers we are going to explain the results from
\cite{MS} in detail now. This requires to introduce some notation first. 

\subsection{Conventions for complex tensors and indices}\label{prelim}
In the sequel, we use without comment the convention of summation
over repeating indices; the small Greek indices $\alpha,\beta,\gamma,\dots$
will have the range $1,\dots,2n$, whereas the indices $s,t,k,l,m$ will be
running from $1$ to $3$.

Consider the Euclidean vector space $\R^{4n}$ with its standard inner
product $\lc,\rc$ and a quaternionic structure induced by the identification
$\R^{4n}\cong\mathbb H^n$ with the quaternion coordinate space $\mathbb
H^n$.  The latter means that we endow $\R^{4n}$ with a fixed triple
$J_1,J_2,J_3$ of complex structures which are Hermitian with respect to
$\lc,\rc$ and satisfy $J_1\,J_2=-J_2\,J_1=J_3.$ The complex vector space
$\mathbb C^{4n}$, being the complexification of $\R^{4n}$, splits as a direct
sum of $+i$ and $-i$ eigenspaces, $\mathbb C^{4n}=\mathcal W\oplus\br
{\mathcal W}$, with respect to the complex structure $J_1.$ The complex
2-form $\pi$,
\begin{equation*}
\pi(u,v)\overset{def}{=}\lc J_2u,v\rc+i\lc J_3u,v\rc,\qquad u,v\in \mathbb C^{4n},
\end{equation*}
has type $(2,0)$ with respect to $J_1$, i.e., it satisfies $\pi(J_1u,v)=\pi(u,J_1v)=i\pi(u,v)$. Let us fix an $\lc,\rc$-orthonormal basis (once and for all)  
\begin{equation}\label{fixed-basis}
\{\mathfrak e_\alpha\in \mathcal W, \mathfrak e_{\bar\alpha}\in \br{\mathcal W}\},\qquad \mathfrak e_{\bar\alpha}=\br{ \mathfrak e_\alpha},
\end{equation}
with dual basis $\{\mathfrak e^\alpha, \mathfrak e^{\bar\alpha}\}$ 
so that $\pi=\mathfrak e^1\wedge \mathfrak e^{n+1}+\mathfrak e^2\wedge \mathfrak e^{n+2}+\dots + \mathfrak e^{n}\wedge \mathfrak e^{2n}$. Then, we have
\begin{equation}
\lc,\rc = g_{\alpha\bar\beta}\, \mathfrak e^\alpha\otimes \mathfrak e^{\bar\beta} + g_{\bar\alpha\beta}\, \mathfrak e^{\bar\alpha}\otimes \mathfrak e^{\beta},\qquad \pi=\pi_{\alpha\beta}\, \mathfrak e^{\alpha}\wedge \mathfrak e^{\beta} 
\end{equation}
with
\begin{equation}\label{constants}
 g_{\alpha\bar\beta}=g_{\bar\beta\alpha}=\begin{cases}1, & \mbox{if } \alpha=\beta\\0, & \mbox{if } \alpha\ne\beta 
 \end{cases},\qquad \pi_{\alpha\beta}=-\pi_{\beta\alpha}=\begin{cases} 1,& \mbox{if } \alpha+n=\beta\\-1,& \mbox{if } \alpha=\beta+n\\0, & \mbox{otherwise}.
 \end{cases}
\end{equation}

Any array of complex numbers indexed by lower and upper Greek letters (with and without bars) corresponds to a tensor, e.g., $\{A_{\alpha\dt\dt}^{\spc\beta\bar\gamma}\}$ corresponds to the tensor 
\begin{equation*}
A_{\alpha\dt\dt}^{\spc\beta\bar\gamma}\,\mathfrak e^\alpha\otimes\mathfrak e_{\beta}\otimes\mathfrak e_{\bar\gamma}.
\end{equation*}     
Clearly, the vertical as well as the horizontal position of an index carries information about the tensor. For two-tensors, we take $B^\alpha_\beta$ to mean $B^{\spc\alpha}_{\beta\dt},$ i.e., the lower index is assumed to be first.  We use $g_{\alpha\bar\beta}$ and $g^{\alpha\bar\beta}=g^{\bar\beta\alpha}=g_{\alpha\bar\beta}$ to lower and raise indices in the usual way, e.g., 
$$A^{\spc\beta}_{\alpha\dt\gamma}=g_{\bar\sigma\gamma}\, A_{\alpha\dt\dt}^{\spc\beta\bar\sigma},\qquad A^{\bar\alpha\beta\bar\gamma}=g^{\bar\alpha\sigma}A_{\sigma\dt\dt}^{\spc\beta\bar\gamma}.$$

We  use  the following convention:  Whenever an array $\{A_{\alpha\dt\dt}^{\spc\beta\bar\gamma}\}$ appears,  the array $\{A_{\bar\alpha\dt\dt}^{\spc\bar\beta\gamma}\}$ will be assumed to be defined, by default, by the complex conjugation
 $$A_{\bar\alpha\dt\dt}^{\spc\bar\beta\gamma}= \overline{A_{\alpha\dt\dt}^{\spc\beta\bar\gamma}}.$$ This means that we interpret $\{A_{\alpha\dt\dt}^{\spc\beta\bar\gamma}\}$ as a representation of a real tensor, defined on $\R^{4n}$, 
 with respect to the fixed complex basis \eqref{fixed-basis}; the corresponding real tensor in this case is
\begin{equation*}
A_{\alpha\dt\dt}^{\spc\beta\bar\gamma}\,\mathfrak e^\alpha\otimes\mathfrak e_{\beta}\otimes\mathfrak e_{\bar\gamma} + A_{\bar\alpha\dt\dt}^{\spc\bar\beta\gamma}\,\mathfrak e^{\bar\alpha}\otimes\mathfrak e_{\bar\beta}\otimes\mathfrak e_{\gamma}.
\end{equation*}     
 
 Notice that we have $\pi^{\alpha}_{\bar\sigma}\,\pi^{\bar\sigma}_{\beta}=-\ \delta^{\alpha}_{\beta}$ ($\delta^{\alpha}_{\beta}$ is the Kronecker delta). We  introduce a complex antilinear endomorphism $\mathfrak j$ of the tensor algebra of $\R^{4n}$, which takes a tensor with components $T_{\alpha_1\dots\alpha_k\bar\beta_1\dots\bar\beta_l\dots}$ to a tensor of the same type, with components $(\mathfrak jT)_{\alpha_1\dots\alpha_k\bar\beta_1\dots\bar\beta_l\dots}$, by the formula
 \begin{equation}
 (\mathfrak jT)_{\alpha_1\dots\alpha_k\bar\beta_1\dots\bar\beta_l\dots}=\sum_{\bar\sigma_1\dots\bar\sigma_k\tau_1\dots\tau_l\dots}\pi^{\bar\sigma_1}_{\alpha_1}\dots\pi^{\bar\sigma_k}_{\alpha_k}\,\pi^{\tau_1}_{\bar\beta_1}\dots\pi^{\tau_l}_{\bar\beta_l}\dots T_{\bar\sigma_1\dots\bar\sigma_k\tau_1\dots\tau_l\dots}.
 \end{equation}

By definition, the group $Sp(n)$ consists of all endomorphisms of $\R^{4n}$ that preserve the inner product $\lc,\rc$ and commute with the complex structures $J_1,J_2$ and $J_3$.   
With the above notation, we can alternatively describe $Sp(n)$ as the set of all two-tensors $\{U^\alpha_\beta\}$ satisfying
\begin{equation}\label{def-Spn}
 g_{\sigma\bar\tau}U^{\sigma}_{\alpha}U^{\bar\tau}_{\bar\beta}=g_{\alpha\bar\beta},\qquad \pi_{\sigma\tau}U^{\sigma}_{\alpha}U^{\tau}_{\beta}=\pi_{\alpha\beta}.
 \end{equation}
For its Lie algebra, $sp(n)$, we have the following description:
\begin{lemma}\label{sp(n)} For a tensor $\{X_{\alpha\bar\beta}\}$, the following conditions are equivalent:
\par (1) $\{X_{\alpha\bar\beta}\} \in sp(n)$.
\par (2) $X_{\alpha\bar\beta}=-X_{\bar\beta\alpha},$ $(\mathfrak jX)_{\alpha\bar\beta}=X_{\alpha\bar\beta}.$
 \par (3) $X_{\beta}^{\alpha}=\pi^{\alpha\sigma}Y_{\sigma\beta}$ for some tensor $\{Y_{\alpha\beta}\}$ satisfying $Y_{\alpha\beta}=Y_{\beta\alpha}$ and $(\mathfrak jY)_{\alpha\beta}=Y_{\alpha\beta}$.
\end{lemma}
 \begin{proof}
The equivalence between (1) and (2) follows by differentiating \eqref{def-Spn} at the identity.   
To obtain (3), we define the tensor $\{Y_{\sigma\beta}\}$ by $Y_{\sigma\beta}=-\pi_{\sigma\tau}X^{\tau}_{\beta}=-\pi_{\sigma}^{\bar\tau}X_{\beta\bar\tau}.$
\end{proof}

\subsection{The canonical Cartan connection and its structure equations}\label{Can_Cartan}
It is well known that to each qc manifold $(M,H)$ one can associate a
unique, up to a diffeomorphism, regular, normal Cartan geometry, i.e., a
certain principle bundle $P_1\rightarrow M$ endowed with a Cartan
connection that satisfies some natural normalization conditions.  In
\cite{MS} we have provided an explicit construction for both the bundle and
the Cartan connection in terms of geometric data generated entirely by the
qc structure of $M$.  Here we will briefly recall this construction since it
is important for the rest of the paper.  The method we are using is
essentially the original Cartan's method of equivalence that was
applied with a great success by Chern and Moser in \cite{ChM} for solving
the respective equivalence problem in the CR case.  It is based entirely on
classical exterior calculus and does not require any preliminary knowledge
concerning the theory of parabolic geometries or the related Lie algebra
cohomology.

By definition, if $(M,H)$ is a qc manifold, around each point of $M$, we can find $\eta_s,I_s$ and $g$ satisfying \eqref{def-hat-eta}. Moreover, if $\tilde\eta_1,\tilde\eta_2,\tilde\eta_3$ are any (other) 1-forms satisfying  \eqref{def-hat-eta} for some  symmetric and positive definite $\tilde g\in H^*\otimes H^*$ and endomorphisms $\tilde I_s \in End(H)$ in place of $g$ and $I_s$ respectively, then it is known (see for example the appendix of \cite{IMV5}) that there exists a positive real-valued function $\mu$ and an $SO(3)$-valued function $\Psi=(a_{st})_{3\times 3}$ so that
\begin{equation*}
\tilde \eta_s=\mu\, a_{ts}\, \eta_t,\qquad \tilde g=\mu \,  g,\qquad \tilde I_s=a_{ts}\,  I_t.
\end{equation*}
Therefore, there exists a natural principle bundle $\pi_o: P_o\rightarrow M$ with structure group $CSO(3)=\R^{+}\times SO(3)$ whose local sections are precisely the triples of
1-forms $(\eta_1,\eta_2,\eta_3)$ satisfying \eqref{def-hat-eta}. Clearly, on $P_o$ we obtain a global triple of canonical one-forms which we will denote again by $(\eta_1,\eta_2,\eta_3)$. The equations \eqref{def-hat-eta} yield (\cite{MS}, Lemma~3.1) the following expressions for the exterior derivatives of the canonical one-forms (using the conventions from Section~\ref{prelim})
\begin{equation}\label{first_str_eq}
\begin{cases}
d\eta_1=-\varphi_0\wedge\eta_1-\varphi_2\wedge\eta_3+\varphi_3\wedge\eta_2+2i g_{\alpha\bar\beta}\,\theta^{\alpha}\wedge\theta^{\bar\beta}\\
d\eta_2=-\varphi_0\wedge\eta_2-\varphi_3\wedge\eta_1+\varphi_1\wedge\eta_3+\pi_{\alpha\beta}\,\theta^{\alpha}\wedge\theta^{\beta}+\pi_{\bar\alpha\bar\beta}\,\theta^{\bar\alpha}\wedge\theta^{\bar\beta}\\
d\eta_3=-\varphi_0\wedge\eta_3-\varphi_1\wedge\eta_2 + \varphi_2\wedge\eta_1-i\pi_{\alpha\beta}\,\theta^{\alpha}\wedge\theta^{\beta}+i\pi_{\bar\alpha\bar\beta}\,\theta^{\bar\alpha}\wedge\theta^{\bar\beta},
\end{cases}
\end{equation}
where $\varphi_0,\varphi_1,\varphi_2,\varphi_3$ are some (local, non-unique) real one-forms on $P_o$, $\theta^\alpha$ are some (local, non-unique) complex and semibasic one-forms on $P_o$  (by semibasic we mean that the contraction of the forms with any vector field tangent to the fibers of $\pi_o$ vanishes), $g_{\alpha\bar\beta}=g_{\bar\beta\alpha}$ and $\pi_{\alpha\beta}=-\pi_{\beta\alpha}$ are the same (fixed) constants as in Section~\ref{prelim}. 

One can show (cf. \cite{MS}, Lemma 3.2) that, if $\tilde\varphi_0,\tilde\varphi_1, \tilde\varphi_2, \tilde\varphi_3,\tilde\theta^\alpha$ are any other one-forms (with the same properties as $\varphi_0,\varphi_1,\varphi_2,\varphi_3,\theta^\alpha$) that satisfy \eqref{first_str_eq}, then
\begin{equation}\label{str_eq_P_o}
\begin{cases}
\tilde\theta^{\alpha}=U^{\alpha}_{\beta}\theta^\beta + ir^\alpha\eta_1+\pi^{\alpha}_{\bar\sigma}r^{\bar\sigma}(\eta_2+i\eta_3)\\
\tilde\varphi_0=\varphi_0+2U_{\beta\bar\sigma}r^{\bar\sigma}\theta^{\beta}+2U_{\bar\beta\sigma}r^{\sigma}\theta^{\bar\beta}+\lambda_1\eta_1+\lambda_2\eta_2+\lambda_3\eta_3\\
\tilde\varphi_1=\varphi_1-2iU_{\beta\bar\sigma}r^{\bar\sigma}\theta^{\beta}+2iU_{\bar\beta\sigma}r^{\sigma}\theta^{\bar\beta}+2r_{\sigma}r^{\sigma}\eta_1-\lambda_3\eta_2+\lambda_2\eta_3,\\
\tilde\varphi_2=\varphi_2-2\pi_{\sigma\tau}U^{\sigma}_{\beta}r^{\tau}\theta^{\beta}-2\pi_{\bar\sigma\bar\tau}U^{\bar\sigma}_{\bar\beta}r^{\bar\tau}\theta^{\bar\beta}+\lambda_3\eta_1+2r_{\sigma}r^{\sigma}\eta_2-\lambda_1\eta_3,\\
\tilde\varphi_3=\varphi_3+2i\pi_{\sigma\tau}U^{\sigma}_{\beta}r^{\tau}\theta^{\beta}-2i\pi_{\bar\sigma\bar\tau}U^{\bar\sigma}_{\bar\beta}r^{\bar\tau}\theta^{\bar\beta}-\lambda_2\eta_1+\lambda_1\eta_2+2r_{\sigma}r^{\sigma}\eta_3,
\end{cases}
\end{equation}
where $U^{\alpha}_{\beta}, r^\alpha, \lambda_s$ are some appropriate functions; $\lambda_1,\lambda_2,\lambda_3$ are real, and $\{U^{\alpha}_{\beta}\}$ satisfy \eqref{def-Spn}, i.e., $\{U^{\alpha}_{\beta}\}\in Sp(n)\subset End(\R^{4n})$. Clearly, the functions $U^{\alpha}_{\beta}, r^\alpha$ and $\lambda_s$ give a parametrization of a certain Lie Group $G_1$ diffeomorphic to $Sp(n)\times \mathbb R^{4n+3}$. There exists a canonical principle bundle $\pi_1: P_1\rightarrow P_o$ whose local sections are precisely the local one-forms $\varphi_0,\varphi_1,\varphi_2,\varphi_3,\theta^\alpha$ on $P_o$ satisfying \eqref{first_str_eq}. 

We use $\varphi_0,\varphi_1,\varphi_2,\varphi_3,\theta^\alpha$ to denote also the induced canonical (global) one-forms on the principal bundle $P_1$. Then, according to \cite{MS}, Theorem~3.3, on $P_1$, there exists a unique set of complex one-forms $\Gamma_{\alpha\beta}, \phi^\alpha$ and real one-forms $\psi_1,\psi_2,\psi_3$  so that 
\begin{equation}\label{Gamma-symmetries}
\Gamma_{\alpha\beta}=\Gamma_{\beta\alpha},\qquad (\mathfrak j\Gamma)_{\alpha\beta}=\Gamma_{\alpha\beta}.
\end{equation}
and the equations
\begin{equation}\label{str-eq-con}
\begin{cases}
d\theta^\alpha=-i\phi^\alpha\wedge\eta_1-\pi^\alpha_{\bar\sigma}\phi^{\bar\sigma}\wedge(\eta_2+i\eta_3)-\pi^{\alpha\sigma}\Gamma_{\sigma\beta}\wedge\theta^\beta-\frac{1}{2}(\varphi_0+i\varphi_1)\wedge\theta^\alpha-\frac{1}{2}\pi^\alpha_{\bar\beta}(\varphi_2+i\varphi_3)\wedge\theta^{\bar\beta}\\
d\varphi_0=-\psi_1\wedge\eta_1-\psi_2\wedge\eta_2-\psi_3\wedge\eta_3-2\phi_\beta\wedge\theta^\beta-2\phi_{\bar\beta}\wedge\theta^{\bar\beta}\\
d\varphi_1=-\varphi_2\wedge\varphi_3-\psi_2\wedge\eta_3+\psi_3\wedge\eta_2+2i\phi_\beta\wedge\theta^\beta-2i\phi_{\bar\beta}\wedge\theta^{\bar\beta}\\
d\varphi_2=-\varphi_3\wedge\varphi_1-\psi_3\wedge\eta_1+\psi_1\wedge\eta_3-2\pi_{\sigma_\beta}\phi^\sigma\wedge\theta^\beta-2\pi_{\bar\sigma\bar\beta}\phi^{\bar\sigma}\wedge\theta^{\bar\beta}\\
d\varphi_3=-\varphi_1\wedge\varphi_2-\psi_1\wedge\eta_2+\psi_2\wedge\eta_1+2i\pi_{\sigma_\beta}\phi^\sigma\wedge\theta^\beta-2i\pi_{\bar\sigma\bar\beta}\phi^{\bar\sigma}\wedge\theta^{\bar\beta},
\end{cases}
\end{equation}
are satisfied.
The so obtained one-forms $\{\eta_s\}$, $\{\theta^\alpha\}$, $\{ \varphi_0\}$, $\{\varphi_s\}$, $\{\Gamma_{\alpha\beta}\}$, $\{\phi^\alpha\}$, $\{\psi_s\}$ represent the components of the canonical Cartan connection (cf. \cite{MS}, Section~5) corresponding to a fixed splitting of the relevant Lie algebra 
\begin{equation*}
sp(n+1,1)=\mathfrak g_{-2}\oplus \mathfrak g_{-1}\oplus \underbrace{\mathbb R\oplus sp(1)\oplus sp(n)}_{\mathfrak g_0}\oplus \mathfrak g_{1} \oplus \mathfrak g_{2}.
\end{equation*} 

The curvature of the Cartan connection may be represented (cf. \cite{MS}, Proposition 4.1) by a set of globally defined complex-valued functions 
 \begin{equation}\label{CurvatureFunctions}
 \mathcal S_{\alpha\beta\gamma\delta},\  \mathcal V_{\alpha\beta\gamma},\ \mathcal  L_{\alpha\beta},\ \mathcal  M_{\alpha\beta}, \ \mathcal C_\alpha, \  \mathcal H_\alpha,\  \mathcal P, \  \mathcal Q,\  \mathcal R
 \end{equation}
 satisfying:
 
 {\bf (I)} Each of the arrays  $\{\mathcal S_{\alpha\beta\gamma\delta}\},\{ \mathcal V_{\alpha\beta\gamma}\},\{ \mathcal L_{\alpha\beta}\}, \{ \mathcal M_{\alpha\beta}\}$ is totally symmetric in its indices.
 
 {\bf (II)} We have
 \begin{equation}\label{properties-curvature}
 \begin{cases}
 (\mathfrak j{\mathcal S})_{\alpha\beta\gamma\delta}=\mathcal S_{\alpha\beta\gamma\delta}\\
 (\mathfrak j{\mathcal L})_{\alpha\beta}=\mathcal L_{\alpha\beta}\\
 \overline {\mathcal R}=\mathcal R.
 \end{cases}
 \end{equation}
 
{\bf  (III)} The exterior derivatives of the connection one-forms $\Gamma_{\alpha\beta}$, $\phi_\alpha$ and $\psi_s$  are given by  
 \begin{equation}\label{dGamma_ab}
 \begin{split}
 d\Gamma_{\alpha\beta}\ =\ &-\pi^{\sigma\tau}\Gamma_{\alpha\sigma}\wedge\Gamma_{\tau\beta} + 2\pi^{\bar\sigma}_{\alpha}(\phi_\beta\wedge\theta_{\bar\sigma}-\phi_{\bar\sigma}\wedge\theta_\beta)+2\pi^{\bar\sigma}_{\beta}(\phi_\alpha\wedge\theta_{\bar\sigma}-\phi_{\bar\sigma}\wedge\theta_\alpha)\\
 &+\pi^{\sigma}_{\bar\delta}\,\mathcal S_{\alpha\beta\gamma\sigma}\,\theta^{\gamma}\wedge\theta^{\bar\delta}+\Big(\mathcal V_{\alpha\beta\gamma}\,\theta^\gamma
 +\pi^{\bar\sigma}_{\alpha}\,\pi^{\bar\tau}_{\beta}\,\mathcal V_{\bar\sigma\bar\tau\bar\gamma}\,\theta^{\bar\gamma}\Big)\wedge\eta_1\\
 &-i\pi^{\sigma}_{\bar\gamma}\,\mathcal V_{\alpha\beta\sigma}\,\theta^{\bar\gamma}\wedge(\eta_2+i\eta_3)+i(\mathfrak j\mathcal V)_{\alpha\beta\gamma}\,\theta^{\gamma}\wedge(\eta_2-i\eta_3)\\
 &-i\mathcal L_{\alpha\beta}\,(\eta_2+i\eta_3)\wedge(\eta_2-i\eta_3)+\mathcal M_{\alpha\beta}\,\eta_1\wedge(\eta_2+i\eta_3)+(\mathfrak j M)_{\alpha\beta}\,\eta_1\wedge(\eta_2-i\eta_3),
  \end{split}
 \end{equation}
 
 \begin{equation}\label{dphi_a}
 \begin{split}
 d\phi_{\alpha}\ =\ \  \ &\frac{1}{2}(\varphi_0+i\varphi_1)\wedge\phi_\alpha 
 +\frac{1}{2}\pi_{\alpha\gamma}(\varphi_2-i\varphi_3)\wedge\phi^{\gamma}
 -\pi^{\bar\sigma}_{\alpha}\, \Gamma_{\bar\sigma\bar\gamma}\wedge\phi^{\bar\gamma}
 -\frac{i}{2}\,\psi_1\wedge\theta_\alpha\\
  -\ &\frac{1}{2}\,\pi_{\alpha\gamma}(\psi_2-i\psi_3)\wedge\theta^{\gamma}
 -i\pi_{\bar\delta}^{\sigma}\,\mathcal V_{\alpha\gamma\sigma}\,\theta^{\gamma}\wedge\theta^{\bar\delta}+\mathcal M_{\alpha\gamma}\,\theta^\gamma\wedge\eta_1+\pi^{\bar\sigma}_{\alpha}\,\mathcal L_{\bar\sigma\bar\gamma}\,\theta^{\bar\gamma}\wedge\eta_1\\
 +\ &i\mathcal L_{\alpha\gamma}\,\theta^{\gamma}\wedge(\eta_2-i\eta_3)-i\pi^{\sigma}_{\bar\gamma}\,\mathcal M_{\alpha\sigma}\,\theta^{\bar\gamma}\wedge(\eta_2+i\eta_3)
 -\mathcal C_{\alpha}(\eta_2+i\eta_3)\wedge(\eta_2-i\eta_3)\\
 +\ &\mathcal H_\alpha\,\eta_1\wedge(\eta_2+i\eta_3)+i\pi_{\alpha\sigma}\,\mathcal C^\sigma\,\eta_1\wedge(\eta_2-i\eta_3),
  \end{split}
 \end{equation}
 
 \begin{equation}\label{dpsi_1}
 \begin{split}
 d\psi_1\ =\ \  \ &\varphi_0\wedge\psi_1-\varphi_2\wedge\psi_3+\varphi_3\wedge\psi_2-4i \phi_\gamma\wedge\phi^\gamma+4\pi^{\sigma}_{\bar\delta}\,\mathcal L_{\gamma\sigma}\,\theta^\gamma\wedge\theta^{\bar\delta}+4\mathcal C_\gamma\,\theta^\gamma\wedge\eta_1\\
 +\ &4\mathcal C_{\bar\gamma}\,\theta^{\bar\gamma}\wedge\eta_1-4i\pi_{\bar\gamma\bar\sigma}\, \mathcal C^{\bar\sigma}\,\theta^{\bar\gamma}\wedge(\eta_2+i\eta_3)
 +4i\pi_{\gamma\sigma}\, \mathcal C^{\sigma}\,\theta^{\gamma}\wedge(\eta_2-i\eta_3)\\
 +\ &\mathcal P\,\eta_1\wedge(\eta_2+i\eta_3)+\overline{\mathcal P}\,\eta_1\wedge(\eta_2-i\eta_3)+i\mathcal R\,(\eta_2+i\eta_3)\wedge(\eta_2-i\eta_3),
 \end{split}
 \end{equation}

\begin{equation}\label{dpsi_23}
 \begin{split}
 d\psi_2+i\,d\psi_3\ =\ \  \ &(\varphi_0-i\varphi_1)\wedge(\psi_2+i\psi_3)+i(\varphi_2+i\varphi_3)\wedge\psi_1+4\pi_{\gamma\delta}\phi^\gamma\wedge\phi^\delta+4i\pi^{\bar\sigma}_{\gamma}\,\mathcal M_{\bar\sigma\bar\delta}\,\theta^\gamma\wedge\theta^{\bar\delta}\\
 +&\ 4i \pi^{\bar\sigma}_{\gamma}\,\mathcal C_{\bar\sigma}\,\theta^\gamma\wedge\eta_1
 -4\mathcal H_{\bar\gamma}\,\theta^{\bar\gamma}\wedge\eta_1-4i\mathcal C_{\bar\gamma}\,\theta^{\bar\gamma}\wedge(\eta_2+i\eta_3)
 -4i\pi_{\gamma}^{\bar\sigma}\, \mathcal H_{\bar\sigma}\,\theta^{\gamma}\wedge(\eta_2-i\eta_3)\\
 -\ &i\mathcal R\,\eta_1\wedge(\eta_2+i\eta_3)+\overline{\mathcal Q}\,\eta_1\wedge(\eta_2-i\eta_3)-\overline{\mathcal P}\,(\eta_2+i\eta_3)\wedge(\eta_2-i\eta_3).
 \end{split}
 \end{equation}

\subsection{The qc structures as integral manifolds of an exterior
differential system}\label{qc_as_EDS} As we have seen above, each qc structure $(M,H)$
determines a principle bundle $P_1$ over $M$ with a coframing
\begin{equation}\label{coframe}
\eta_s,\ \theta^\alpha,\ \varphi_0,\ \varphi_s,\ \Gamma_{\alpha\beta},\ \phi^\alpha,\ \psi_s
\end{equation} 
satisfying  
\eqref{Gamma-symmetries}, \eqref{first_str_eq}, \eqref{str-eq-con}, together with a set of functions 
\begin{equation}\label{functions_str}
 \mathcal S_{\alpha\beta\gamma\delta},\  \mathcal V_{\alpha\beta\gamma},\ \mathcal  L_{\alpha\beta},\ \mathcal  M_{\alpha\beta}, \ \mathcal C_\alpha, \  \mathcal H_\alpha,\  \mathcal P, \  \mathcal Q,\  \mathcal R
 \end{equation}
 with the respective properties (I), (II) and (III) of Section~\ref{Can_Cartan}. As it can be easily shown, the converse is also true, i.e., each manifold $P_1$ endowed with a coframing  \eqref{coframe} and function \eqref{functions_str}, satisfying all the respective properties, can be viewed, locally (in a unique way), as the canonical principle bundle of a (unique) qc structure. Therefore, finding local qc structures is equivalent to finding linearly independent one-forms  \eqref{coframe} and functions \eqref{functions_str} on an open domain in $\mathbb R^{\dim(P_1)}$ satisfying the above properties.
This is, clearly, a problem of type \eqref{str_eq_d_omega} and thus it reduces---as explained in Section~\ref{ex_dif_sys}---to a typical problem from the theory of exterior differential systems that can be handled using the Cartan's Third Theorem. 

 For the respective exterior differential system, the validity of Assumption I, Section~\ref{ex_dif_sys} follows immediately from \cite{MS}, Propositions~4.2 which says that the exterior differentiation of \eqref{dGamma_ab}, \eqref{dphi_a}, \eqref{dpsi_1} and \eqref{dpsi_23} produces equations that can be put into the form: 

\begin{multline}\label{d2Gamma_ab}
 \Big(d^2\Gamma_{\alpha\beta}\  = \Big)\ \pi^{\sigma}_{\bar\delta}{\mathcal S}^\ast_{\alpha\beta\gamma\sigma}\wedge\theta^\gamma\wedge\theta^{\bar\delta}
+\ {\mathcal V}^\ast_{\alpha\beta\gamma}\wedge\theta^\gamma\wedge\eta_1 
 +\ \pi^{\bar\mu}_\alpha\,\pi^{\bar\nu}_\beta\,{\mathcal V}^\ast_{\bar\mu\bar\nu\bar\gamma}\wedge\theta^{\bar\gamma}\wedge\eta_1 \\
 -\ i\pi^{\sigma}_{\bar\gamma}\,{\mathcal V}^\ast_{\alpha\beta\sigma}\wedge\theta^{\bar\gamma}\wedge\big(\eta_2+i\eta_3\big) 
+\ i\pi^{\bar\mu}_\alpha\,\pi^{\bar\nu}_\beta\,\pi^{\bar\xi}_\gamma\,{\mathcal V}^\ast_{\bar\mu\bar\nu\bar\xi}\wedge\theta^{\gamma}\wedge\big(\eta_2-i\eta_3\big) \\
-i\mathcal L_{\alpha\beta}^\ast\wedge\big(\eta_2+i\eta_3\big)\wedge\big(\eta_2-i\eta_3\big) 
+\mathcal M^\ast_{\alpha\beta}\wedge\eta_1\wedge\big(\eta_2+i\eta_3\big) \\
+\pi^{\bar\mu}_\alpha\,\pi^{\bar\nu}_\beta {\mathcal M}^\ast_{\bar\mu\bar\nu}\wedge\eta_1\wedge\big(\eta_2-i\eta_3\big) \ =\ 0;
 \end{multline} 
 \begin{multline}\label{d2phi_a}
 \Big(d^2\phi_\alpha\  =\Big) \  
-\ i\pi^{\nu}_{\bar\gamma}\,{\mathcal V}^\ast_{\alpha\beta\nu}\wedge\theta^{\beta}\wedge\theta^{\bar\gamma} 
+\pi^{\bar\mu}_\alpha\,\mathcal L_{\bar\mu\bar\beta}^\ast\wedge\theta^{\bar\beta}\wedge\eta_1
 +\mathcal M^\ast_{\alpha\beta}\wedge\theta^{\beta}\wedge\eta_1\\
-i\pi^{\nu}_{\bar\beta}\,\mathcal M^\ast_{\alpha\nu}\wedge\theta^{\bar\beta}\wedge\big(\eta_2+i\eta_3\big) 
+i\mathcal L^\ast_{\alpha\beta}\wedge\theta^{\beta}\wedge\big(\eta_2-i\eta_3\big) 
-\mathcal C^\ast_\alpha
 \wedge\big(\eta_2+i\eta_3\big)\wedge\big(\eta_2-i\eta_3\big) \\
+i\pi_{\alpha}^{\bar\mu}\,\mathcal C^\ast_{\bar\mu}\wedge\eta_1\wedge\big(\eta_2-i\eta_3\big) 
 +\mathcal H^\ast_\alpha
 \wedge\eta_1\wedge\big(\eta_2+i\eta_3\big) \ =\ 0;
 \end{multline}
 \begin{multline}\label{d2psi_1}
 \Big(d^2\psi_1\  =\Big) \ 4\pi^\mu_{\bar\gamma}\, \mathcal L^\ast_{\beta\mu}\wedge\theta^{\beta}\wedge\theta^{\bar\gamma}+4\mathcal C^\ast_\beta\wedge\theta^\beta\wedge\eta_1 
 +4\mathcal C^\ast_{\bar\gamma}\wedge\theta^{\bar\gamma}\wedge\eta_1
 +4i\pi_{\beta}^{\bar\mu}\,\mathcal C^\ast_{\bar\mu}\wedge\theta^{\beta}\wedge\big(\eta_2-i\eta_3\big)\\ 
-4i\pi^{\mu}_{\bar\gamma}\,\mathcal C^\ast_{\mu}\wedge\theta^{\bar\gamma}\wedge\big(\eta_2+i\eta_3\big) 
+ \mathcal P^\ast\wedge\eta_1\wedge\big(\eta_2+i\eta_3\big) + \overline{\mathcal P^\ast}\wedge\eta_1\wedge\big(\eta_2-i\eta_3\big)\\
+i\mathcal R^\ast\wedge\big(\eta_2+i\eta_3\big) \wedge\big(\eta_2-i\eta_3\big) 
\ =\ 0;
 \end{multline}
  \begin{multline}\label{d2psi_2}
 \Big(d^2\big(\psi_2+i\psi_3\big)\  =\Big) \ 4i\pi^{\bar\mu}_{\beta}\, \mathcal M^\ast_{\bar\mu\bar\gamma}\wedge\theta^{\beta}\wedge\theta^{\bar\gamma}
 +4i\pi^{\bar\mu}_{\beta}\,\mathcal C^\ast_{\bar\mu}\wedge\theta^{\beta}\wedge\eta_1
  -4\mathcal H^\ast_{\bar\gamma}\wedge\theta^{\bar\gamma}\wedge\eta_1
  -4\mathcal C^\ast_{\bar\gamma}\wedge\theta^{\bar\gamma}\wedge\big(\eta_2+i\eta_3\big)\\
 -4i\pi^{\bar\mu}_{\beta}\,\mathcal H^\ast_{\bar\mu}\wedge\theta^{\beta}\wedge\big(\eta_2-i\eta_3\big)
 -i\mathcal R^\ast\wedge\eta_1\wedge\big(\eta_2+i\eta_3\big) + \overline{\mathcal Q^\ast}\wedge\eta_1\wedge\big(\eta_2-i\eta_3\big)
 \\ 
-\overline{\mathcal P^\ast}\wedge\big(\eta_2+i\eta_3\big) \wedge\big(\eta_2-i\eta_3\big) 
\ =\ 0.
 \end{multline}
Where
\begin{equation}\label{curvature_one_forms}
 \mathcal S^\ast_{\alpha\beta\gamma\delta},\  \mathcal V^\ast_{\alpha\beta\gamma},\ \mathcal  L^\ast_{\alpha\beta},\ \mathcal  M^\ast_{\alpha\beta}, \ \mathcal C^\ast_\alpha, \  \mathcal H^\ast_\alpha,\  \mathcal P^\ast, \  \mathcal Q^\ast,\  \mathcal R^\ast
 \end{equation}
are certain (new) one-forms on $P_1$ each one of which begins with the differential of the corresponding curvature component followed by certain corrections terms. More precisely, we have

\begin{multline}\label{S*}
\mathcal S^\ast_{\alpha\beta\gamma\delta}\overset{def}{\ =\ }d{\mathcal S}_{\alpha\beta\gamma\delta}-\pi^{\tau\nu}\Gamma_{\nu\alpha}\,\mathcal S_{\tau\beta\gamma\delta}
 -\pi^{\tau\nu}\Gamma_{\nu\beta}\,\mathcal S_{\alpha\tau\gamma\delta}-\pi^{\tau\nu}\Gamma_{\nu\gamma}\,\mathcal S_{\alpha\beta\tau\delta}
 -\pi^{\tau\nu}\Gamma_{\nu\delta}\,\mathcal S_{\alpha\beta\gamma\tau}\\
 -\varphi_0\, \mathcal S_{\alpha\beta\gamma\delta}
 -2i\big(\pi_{\alpha\tau}\,\mathcal V_{\delta\beta\gamma}+\pi_{\beta\tau}\,\mathcal V_{\alpha\gamma\delta}+\pi_{\gamma\tau}\,\mathcal V_{\alpha\beta\delta}
 +\pi_{\delta\tau}\,\mathcal V_{\alpha\beta\gamma}\big)\theta^{\tau}
 -2i\Big(g_{\alpha\bar\tau}\,(\mathfrak j\mathcal V)_{\delta\beta\gamma}\\
 +g_{\beta\bar\tau}\,(\mathfrak j\mathcal V)_{\alpha\delta\gamma}+g_{\gamma\bar\tau}\,(\mathfrak j\mathcal V)_{\alpha\beta\delta}
 +g_{\delta\bar\tau}\,(\mathfrak j\mathcal V)_{\alpha\beta\gamma}\Big)\theta^{\bar\tau}
\end{multline} 
\begin{multline}\label{V*}
 \mathcal V^\ast_{\alpha\beta\gamma}\overset{def}{\ =\ }d{\mathcal V}_{\alpha\beta\gamma}-\pi^{\tau\nu}\Gamma_{\nu\alpha}\,\mathcal V_{\tau\beta\gamma}
 -\pi^{\tau\nu}\Gamma_{\nu\beta}\,\mathcal V_{\alpha\tau\gamma}-\pi^{\tau\nu}\Gamma_{\nu\gamma}\,\mathcal V_{\alpha\beta\tau}
 +i\pi^{\sigma}_{\bar\tau}\,\phi^{\bar\tau}\,\mathcal S_{\alpha\beta\gamma\sigma}\\-\frac{1}{2}\big(3\varphi_0+i\varphi_1\big)\, \mathcal V_{\alpha\beta\gamma}
 +\frac{1}{2}\big(\varphi_2-i\varphi_3\big)\, (\mathfrak j\mathcal V)_{\alpha\beta\gamma}
 +2\big(\pi_{\alpha\tau}\,\mathcal M_{\beta\gamma}+\pi_{\beta\tau}\,\mathcal M_{\alpha\gamma}+\pi_{\gamma\tau}\,\mathcal M_{\alpha\beta}\big)\theta^{\tau}\\
 +2\big(g_{\alpha\bar\tau}\, \mathcal L_{\beta\gamma}
 +g_{\beta\bar\tau}\,\mathcal L_{\alpha\gamma}+g_{\gamma\bar\tau}\,\mathcal L_{\alpha\beta}\big)\theta^{\bar\tau}
 \end{multline}
\begin{multline}\label{L*}
 \mathcal L^\ast_{\alpha\beta}\overset{def}{\ =\ }d\mathcal L_{\alpha\beta}-\pi^{\tau\sigma}\Gamma_{\sigma\alpha}\,\mathcal L_{\tau\beta}-\pi^{\tau\sigma}\Gamma_{\sigma\beta}\,\mathcal L_{\alpha\tau}
-2\varphi_0\,\mathcal L_{\alpha\beta}
-\frac{1}{2}\big(\varphi_2+i\varphi_3\big)\mathcal M_{\alpha\beta}
 -\frac{1}{2}\big(\varphi_2-i\varphi_3\big)(\mathfrak j\mathcal M)_{\alpha\beta}\\
  -\phi^\sigma\,\mathcal V_{\alpha\beta\sigma}-\pi^{\bar\mu}_\alpha\,\pi^{\bar\nu}_\beta\,\phi^{\bar\sigma}\,\mathcal V_{\bar\mu\bar\nu\bar\sigma}
  -2i\big(\pi_{\alpha\tau}\mathcal C_\beta+\pi_{\beta\tau}\mathcal C_\alpha\big)\theta^\tau
 -2i\big(g_{\alpha\bar\tau}\,\pi^{\bar\sigma}_{\beta}\,\mathcal C_{\bar\sigma}+g_{\beta\bar\tau}\,\pi^{\bar\sigma}_{\alpha}\,\mathcal C_{\bar\sigma}\big)\theta^{\bar\tau}
 \end{multline} 
\begin{multline}\label{M*}
 \mathcal M^\ast_{\alpha\beta}\overset{def}{\ =\ }d\mathcal M_{\alpha\beta}-\pi^{\tau\sigma}\Gamma_{\sigma\alpha}\,\mathcal M_{\tau\beta}-\pi^{\tau\sigma}\Gamma_{\sigma\beta}\,\mathcal M_{\alpha\tau}
 -\big(2\varphi_0+i\varphi_1\big)\,\mathcal M_{\alpha\beta}+\big(\varphi_2-i\varphi_3\big)\mathcal L_{\alpha\beta}
 +2\pi^{\sigma}_{\bar\tau}\,\phi^{\bar\tau}\,\mathcal V_{\alpha\beta\sigma}\\
 +2\big(\pi_{\alpha\tau}\mathcal H_\beta+\pi_{\beta\tau}\mathcal H_\alpha\big)\theta^\tau
 -2i\big(g_{\alpha\bar\tau}\,\mathcal C_{\beta}+g_{\beta\bar\tau}\,\mathcal C_{\alpha}\big)\theta^{\bar\tau}
 \end{multline} 
\begin{multline}\label{C*}
\mathcal C^\ast_{\alpha}\overset{def}{\ =\ }d\mathcal C_\alpha-\pi^{\tau\sigma}\Gamma_{\sigma\alpha}\,\mathcal C_\tau 
 -\frac{1}{2}\big(5\varphi_0+i\varphi_1\big)\,\mathcal C_\alpha+\pi^{\bar\sigma}_{\alpha}\big(\varphi_2-i\varphi_3\big)\mathcal C_{\bar\sigma}
+ 2i\pi^{\sigma}_{\bar\tau}\,\phi^{\bar\tau}\,\mathcal L_{\alpha\sigma} - i\phi^{\tau}\,\mathcal M_{\alpha\tau} \\
-\frac{i}{2} \big(\varphi_2+i\varphi_3\big)\mathcal H_{\alpha}+\frac{1}{2}\pi_{\alpha\tau}\,\mathcal P\,\theta^\tau-\frac{1}{2}g_{\alpha\bar\tau}\,\mathcal R\,\theta^{\bar\tau}
 \end{multline}
 \begin{multline}\label{H*}
\mathcal H^\ast_{\alpha}\overset{def}{\ =\ }d\mathcal H_\alpha-\pi^{\tau\sigma}\Gamma_{\sigma\alpha}\,\mathcal H_\tau 
 -\frac{1}{2}\big(5\varphi_0+3i\varphi_1\big)\,\mathcal H_\alpha-\frac{3i}{2}\big(\varphi_2-i\varphi_3\big)\mathcal C_{\alpha}
+3\pi^{\sigma}_{\bar\tau}\,\phi^{\bar\tau}\,\mathcal M_{\alpha\sigma}\\
-\frac{1}{2}\pi_{\alpha\tau}\,\mathcal Q\,\theta^\tau-\frac{i}{2}g_{\alpha\bar\tau}\,\mathcal P\,\theta^{\bar\tau}
 \end{multline}
\begin{equation}\label{R*}
\mathcal R^\ast\overset{def}{\ =\ }d\mathcal R  
 -3\varphi_0\,\mathcal R +\big(\varphi_2+i\varphi_3\big)\mathcal P+\big(\varphi_2-i\varphi_3\big)\overline{\mathcal P}+8\phi^\tau\,\mathcal C_\tau+8\phi^{\bar\tau}\,\mathcal C_{\bar\tau}
 \end{equation}
 \begin{equation}\label{P*}
\mathcal P^\ast\overset{def}{\ =\ }d\mathcal P  
 -\big(3\varphi_0+i\varphi_1\big)\mathcal P +\frac{i}{2}\big(\varphi_2+i\varphi_3\big){\mathcal Q}-\frac{3}{2}\big(\varphi_2-i\varphi_3\big){\mathcal R}-4i\phi^\tau\mathcal H_{\tau}+12\pi_{\bar\tau\bar\sigma}\phi^{\bar\tau}\,\mathcal C_{\bar\sigma}
 \end{equation}
  \begin{equation}\label{Q*}
\mathcal Q^\ast\overset{def}{\ =\ }d\mathcal Q  
 -\big(3\varphi_0+2i\varphi_1\big)\mathcal Q +2i\big(\varphi_2-i\varphi_3\big){\mathcal P}-16\pi_{\bar\tau\bar\sigma}\phi^{\bar\tau}\mathcal H^{\bar\sigma}
 \end{equation}

In order to show that the Assumption II, Section~\ref{ex_dif_sys} holds true for the differential system under consideration, we observe that the Bianchi identities \eqref{d2Gamma_ab}, \eqref{d2phi_a}, \eqref{d2psi_1} and \eqref{d2psi_2} imply that the one-forms \eqref{curvature_one_forms} belong to the linear span of $\eta_1$,  $\eta_2$, $\eta_3$, $\theta^\alpha$, $\theta^{\bar\alpha}$. Furthermore, if we are considering  the above Bianchi identities as a system of algebraic equations for the unknown one-forms  \eqref{curvature_one_forms}, then---since this system is clearly linear---the solutions may be parametrized by elements of a certain vector space. In \cite{MS}, Proposition~4.3 we have given an explicit description for this vector space. Namely, we have shown that, on $P_1$, there exist unique, globally defined, complex valued functions
 \begin{equation}\label{solution_space_coord}
 \begin{split}
 \mathcal A_{\alpha\beta\gamma\delta\epsilon},\ & \mathcal B_{\alpha\beta\gamma\delta},\  \mathcal C_{\alpha\beta\gamma\delta} ,\ \mathcal  D_{\alpha\beta\gamma},\ \mathcal  E_{\alpha\beta\gamma},
 \ \mathcal  F_{\alpha\beta\gamma},\ \mathcal  G_{\alpha\beta},\ \mathcal  X_{\alpha\beta},\ \mathcal  Y_{\alpha\beta},\ \mathcal  Z_{\alpha\beta}, \\
  &\ (\mathcal N_1)_\alpha,\ (\mathcal N_2)_\alpha,\ (\mathcal N_3)_\alpha,\ (\mathcal N_4)_\alpha,\ (\mathcal N_5)_\alpha,\ \mathcal U_s,\  \mathcal W_s
\end{split}
\end{equation}
 so that:
 
 {\bf (I)} Each of the arrays  $\{\mathcal A_{\alpha\beta\gamma\delta\epsilon}\}$, $\{ \mathcal B_{\alpha\beta\gamma\delta}\}$, $\{ \mathcal C_{\alpha\beta\gamma\delta}\}$, $\{ \mathcal D_{\alpha\beta\gamma}\}$, $\{ \mathcal E_{\alpha\beta\gamma}\}$, $\{ \mathcal F_{\alpha\beta\gamma}\}$, $\{ \mathcal G_{\alpha\beta\gamma}\}$, $\{ \mathcal X_{\alpha\beta}\}$, $\{ \mathcal Y_{\alpha\beta}\}$, $\{ \mathcal Z_{\alpha\beta}\}$ is totally symmetric in its indices.
 
{\bf  (II)} We have  
 \begin{equation}\label{FormulaeForCovDer}
 \begin{split}
\mathcal S^\ast_{\alpha\beta\gamma\delta}&\ =\  
\mathcal A_{\alpha\beta\gamma\delta\epsilon}\,\theta^{\epsilon}-\pi^{\sigma}_{\bar\epsilon}(\mathfrak j\mathcal A)_{\alpha\beta\gamma\delta\sigma}\,\theta^{\bar\epsilon}
 +\Big(\mathcal B_{\alpha\beta\gamma\delta}+(\mathfrak j \mathcal B)_{\alpha\beta\gamma\delta}\Big)\eta_1 + i\mathcal C_{\alpha\beta\gamma\delta}\big(\eta_2+i\eta_3\big)\\
&\qquad\qquad\qquad\qquad\qquad\qquad\qquad\qquad\qquad\qquad\qquad\qquad\qquad - i(\mathfrak j\mathcal C)_{\alpha\beta\gamma\delta}\big(\eta_2-i\eta_3\big)\\
\mathcal V^\ast_{\alpha\beta\gamma}&\ =\ \mathcal C_{\alpha\beta\gamma\epsilon}\,\theta^{\epsilon}+\pi^{\sigma}_{\bar\epsilon}\,\mathcal B_{\alpha\beta\gamma\sigma}\,\theta^{\bar\epsilon}
 +\mathcal D_{\alpha\beta\gamma}\eta_1 + \mathcal E_{\alpha\beta\gamma}\big(\eta_2+i\eta_3\big)+\mathcal F_{\alpha\beta\gamma}\big(\eta_2-i\eta_3\big)\\
 \mathcal L^\ast_{\alpha\beta}&\ =\ -
(\mathfrak j\mathcal F)_{\alpha\beta\epsilon}\,\theta^{\epsilon}\!-\!\pi^{\sigma}_{\bar\epsilon}\, \mathcal F_{\alpha\beta\sigma}\,\theta^{\bar\epsilon}
 +i\Big((\mathfrak j\mathcal Z)_{\alpha\beta}-\mathcal Z_{\alpha\beta}\Big)\eta_1 +i \mathcal G_{\alpha\beta}\big(\eta_2+i\eta_3\big)-i(\mathfrak j\mathcal G)_{\alpha\beta}\big(\eta_2-i\eta_3\big)\\
  \mathcal M^\ast_{\alpha\beta}&\ =\ -\mathcal E_{\alpha\beta\epsilon}\,\theta^{\epsilon}+\pi^{\sigma}_{\bar\epsilon}\Big((\mathfrak j\mathcal F)_{\alpha\beta\sigma}-i\mathcal D_{\alpha\beta\sigma}\Big)\,\theta^{\bar\epsilon}
 +\mathcal X_{\alpha\beta}\eta_1 + \mathcal Y_{\alpha\beta}\big(\eta_2+i\eta_3\big)+\mathcal Z_{\alpha\beta}\big(\eta_2-i\eta_3\big)\\
  \mathcal C^\ast_{\alpha}&\ =\ \mathcal G_{\alpha\epsilon}\,\theta^{\epsilon}-i\pi^{\sigma}_{\bar\epsilon}\mathcal Z_{\alpha\sigma}\,\theta^{\bar\epsilon}
 +(\mathcal N_1)_{\alpha}\eta_1 + (\mathcal N_2)_{\alpha}\big(\eta_2+i\eta_3\big)+(\mathcal N_3)_{\alpha}\big(\eta_2-i\eta_3\big)\\
  \mathcal H^\ast_{\alpha}&\ =\ -\mathcal Y_{\alpha\epsilon}\,\theta^{\epsilon}+i\pi^{\sigma}_{\bar\epsilon}\big(\mathcal G_{\alpha\sigma}-\mathcal X_{\alpha\sigma}\big)\theta^{\bar\epsilon}
 +(\mathcal N_4)_{\alpha}\eta_1 + (\mathcal N_5)_{\alpha}\big(\eta_2+i\eta_3\big)\\
 &\qquad\qquad\qquad\qquad\qquad\qquad\qquad\qquad\qquad\qquad\qquad+\Big((\mathcal N_1)_{\alpha}+i\pi^{\bar\sigma}_{\alpha}(\mathcal N_3)_{\bar\sigma}\Big)\big(\eta_2-i\eta_3\big)\\
  \mathcal R^\ast&\ =\  4\pi^{\bar\sigma}_{\epsilon}(\mathcal N_3)_{\bar\sigma}\,\theta^{\epsilon}+4\pi^{\sigma}_{\bar\epsilon}(\mathcal N_3)_{\sigma}\,\theta^{\bar\epsilon}
 +i\big(\mathcal U_3-\overline{\mathcal U}_3\big)\eta_1 -i\big(\mathcal U_1+\mathcal W_3\big)\big(\eta_2+i\eta_3\big)\\
 &\qquad\qquad\qquad\qquad\qquad\qquad\qquad\qquad\qquad\qquad\qquad\qquad+i\big(\overline{\mathcal U}_1+\overline{\mathcal W}_3\big)\big(\eta_2-i\eta_3\big)\\
  \mathcal P^\ast&\ =\ -4(\mathcal N_2)_{\epsilon}\,\theta^{\epsilon}-4\Big((\mathcal N_3)_{\bar\epsilon}+i\pi^{\sigma}_{\bar\epsilon}(\mathcal N_1)_{\sigma}\Big)\,\theta^{\bar\epsilon}
 +\mathcal U_1\eta_1 + \mathcal U_2\big(\eta_2+i\eta_3\big)+\mathcal U_3\big(\eta_2-i\eta_3\big)\\
  \mathcal Q^\ast&\ =\ 4(\mathcal N_5)_{\epsilon}\,\theta^{\epsilon}+4i\pi^{\sigma}_{\bar\epsilon}\Big((\mathcal N_2)_{\sigma}+(\mathcal N_4)_{\sigma}\Big)\,\theta^{\bar\epsilon}
 +\mathcal W_1\eta_1 + \mathcal W_2\big(\eta_2+i\eta_3\big)+\mathcal W_3\big(\eta_2-i\eta_3\big).
 \end{split}
 \end{equation}

\subsection{The Cartan test}
For the (real) dimension $D$ of the vector space determined by \eqref{solution_space_coord}, we calculate
\begin{multline}\label{compute_D}
D=
2{{2n+4}\choose{5}}+4{{2n+3}\choose{4}}+6{{2n+2}\choose{3}}+8{{2n+1}\choose{2}}+20n+12\\
=\frac{2}{15}(2n+5)(2n+3)(n+3)(n+2)(n+1).
\end{multline} 

Following the scheme of Section~\ref{ex_dif_sys}, the problem of finding all possible coframings \eqref{coframe} and functions \eqref{functions_str} satisfying the respective relations (i.e., the problem of finding all local qc structures) may be seen, equivalently, as the problem of solving a certain associated exterior differential system, which we describe next: Let us denote the (real) dimension of $P_1$ by $d_1$. We have
\begin{equation}\label{def_d_1}
d_1={{2n+5}\choose{2}}=(2n+5)(n+2)
\end{equation}
 
 The functions \eqref{functions_str}, with their respective properties (I) and (II) assumed, determine a vector space for which these functions represent the coordinate components of vectors. For the dimension $d_2$ of this vector space, we easily compute
 \begin{equation}\label{def_d_2}
 d_2={{2n+3}\choose{4}}+2{{2n+2}\choose{3}}+3{{2n+1}\choose{2}}+8n+5=\frac{1}{6}(2n+5)(2n+3)(n+2)(n+1).
\end{equation}
 
 Then, the associated exterior differential system that we are considering is defined by a differential ideal $\mathcal I$ on the product manifold 
\begin{equation}\label{ProductN}
N= GL(d_1,\mathbb R)\times\mathbb R^{d_1}\times\mathbb R^{d_2}.
\end{equation} 
We can interpret,  in a natural way (cf. Section~\ref{ex_dif_sys}), \eqref{coframe} and \eqref{functions_str} as one-forms and functions on $N$ respectively. Then, the ideal $\mathcal I$ is algebraically generated by the two-forms given by the structure equations \eqref{first_str_eq}, \eqref{str-eq-con},  \eqref{dGamma_ab}, \eqref{dphi_a}, \eqref{dpsi_1}, \eqref{dpsi_23}, and by the three-forms determined by the Bianchi identities \eqref{d2Gamma_ab}, \eqref{d2phi_a}, \eqref{d2psi_1} and  \eqref{d2psi_2} (these are the only non-trivial equations that we obtain by exterior differentiation of the structure equations). Since only the latter are relevant for the computation of the character sequence of the ideal (cf. Section~\ref{ex_dif_sys}), we will denote them by $\Delta_{\alpha\beta}$, $\Delta_{\alpha}$ and $\Psi_s$ respectively, i.e., we have: 
\begin{multline}\label{Form_d2Gamma_ab}
 {\Delta}_{\alpha\beta}\  {=} \ \pi^{\sigma}_{\bar\delta}{\mathcal S}^\ast_{\alpha\beta\gamma\sigma}\wedge\theta^\gamma\wedge\theta^{\bar\delta}
+\ {\mathcal V}^\ast_{\alpha\beta\gamma}\wedge\theta^\gamma\wedge\eta_1 
 +\ \pi^{\bar\mu}_\alpha\,\pi^{\bar\nu}_\beta\,{\mathcal V}^\ast_{\bar\mu\bar\nu\bar\gamma}\wedge\theta^{\bar\gamma}\wedge\eta_1 \\
 -\ i\pi^{\sigma}_{\bar\gamma}\,{\mathcal V}^\ast_{\alpha\beta\sigma}\wedge\theta^{\bar\gamma}\wedge\big(\eta_2+i\eta_3\big) 
+\ i\pi^{\bar\mu}_\alpha\,\pi^{\bar\nu}_\beta\,\pi^{\bar\xi}_\gamma\,{\mathcal V}^\ast_{\bar\mu\bar\nu\bar\xi}\wedge\theta^{\gamma}\wedge\big(\eta_2-i\eta_3\big) \\
-i\mathcal L_{\alpha\beta}^\ast\wedge\big(\eta_2+i\eta_3\big)\wedge\big(\eta_2-i\eta_3\big) 
+\mathcal M^\ast_{\alpha\beta}\wedge\eta_1\wedge\big(\eta_2+i\eta_3\big) \\
+\pi^{\bar\mu}_\alpha\,\pi^{\bar\nu}_\beta {\mathcal M}^\ast_{\bar\mu\bar\nu}\wedge\eta_1\wedge\big(\eta_2-i\eta_3\big);
 \end{multline} 
 \begin{multline}\label{Form_d2phi_a}
 {\Delta}_{\alpha}\  {=} \ 
-\ i\pi^{\nu}_{\bar\gamma}\,{\mathcal V}^\ast_{\alpha\beta\nu}\wedge\theta^{\beta}\wedge\theta^{\bar\gamma} 
+\pi^{\bar\mu}_\alpha\,\mathcal L_{\bar\mu\bar\beta}^\ast\wedge\theta^{\bar\beta}\wedge\eta_1
 +\mathcal M^\ast_{\alpha\beta}\wedge\theta^{\beta}\wedge\eta_1\\
-i\pi^{\nu}_{\bar\beta}\,\mathcal M^\ast_{\alpha\nu}\wedge\theta^{\bar\beta}\wedge\big(\eta_2+i\eta_3\big) 
+i\mathcal L^\ast_{\alpha\beta}\wedge\theta^{\beta}\wedge\big(\eta_2-i\eta_3\big) 
-\mathcal C^\ast_\alpha
 \wedge\big(\eta_2+i\eta_3\big)\wedge\big(\eta_2-i\eta_3\big) \\
+i\pi_{\alpha}^{\bar\mu}\,\mathcal C^\ast_{\bar\mu}\wedge\eta_1\wedge\big(\eta_2-i\eta_3\big) 
 +\mathcal H^\ast_\alpha
 \wedge\eta_1\wedge\big(\eta_2+i\eta_3\big);
 \end{multline}
 \begin{multline}\label{Form_d2psi_1}
 \Psi_1\  {=} \ 4\pi^\mu_{\bar\gamma}\, \mathcal L^\ast_{\beta\mu}\wedge\theta^{\beta}\wedge\theta^{\bar\gamma}+4\mathcal C^\ast_\beta\wedge\theta^\beta\wedge\eta_1 
 +4\mathcal C^\ast_{\bar\gamma}\wedge\theta^{\bar\gamma}\wedge\eta_1
 +4i\pi_{\beta}^{\bar\mu}\,\mathcal C^\ast_{\bar\mu}\wedge\theta^{\beta}\wedge\big(\eta_2-i\eta_3\big)\\ 
-4i\pi^{\mu}_{\bar\gamma}\,\mathcal C^\ast_{\mu}\wedge\theta^{\bar\gamma}\wedge\big(\eta_2+i\eta_3\big) 
+ \mathcal P^\ast\wedge\eta_1\wedge\big(\eta_2+i\eta_3\big) + \overline{\mathcal P^\ast}\wedge\eta_1\wedge\big(\eta_2-i\eta_3\big)\\
+i\mathcal R^\ast\wedge\big(\eta_2+i\eta_3\big) \wedge\big(\eta_2-i\eta_3\big);
 \end{multline}
  \begin{multline}\label{Form_d2psi_2}
 \Psi_2+i\Psi_3\  {=} \ 4i\pi^{\bar\mu}_{\beta}\, \mathcal M^\ast_{\bar\mu\bar\gamma}\wedge\theta^{\beta}\wedge\theta^{\bar\gamma}
 +4i\pi^{\bar\mu}_{\beta}\,\mathcal C^\ast_{\bar\mu}\wedge\theta^{\beta}\wedge\eta_1
  -4\mathcal H^\ast_{\bar\gamma}\wedge\theta^{\bar\gamma}\wedge\eta_1
  -4\mathcal C^\ast_{\bar\gamma}\wedge\theta^{\bar\gamma}\wedge\big(\eta_2+i\eta_3\big)\\
 -4i\pi^{\bar\mu}_{\beta}\,\mathcal H^\ast_{\bar\mu}\wedge\theta^{\beta}\wedge\big(\eta_2-i\eta_3\big)
 -i\mathcal R^\ast\wedge\eta_1\wedge\big(\eta_2+i\eta_3\big) + \overline{\mathcal Q^\ast}\wedge\eta_1\wedge\big(\eta_2-i\eta_3\big)
 \\ 
-\overline{\mathcal P^\ast}\wedge\big(\eta_2+i\eta_3\big) \wedge\big(\eta_2-i\eta_3\big).
 \end{multline}

In order to show that our exterior differential system $\mathcal I$ is in involution---which would allow us to apply the Cartan's Third Theorem to it---we need to compute the character sequence $v_1,v_2,v_3,\dots,v_{d_1}$ of the system and show that the Cartan's test
\begin{equation} 
D=v_1+2v_2+3v_3+\dots+d_1v_{d_1}
\end{equation}
is satisfied. We will do this in the next section. 
 
 \section{Involutivity of the associated exterior differential
system}\label{sec3}
 
 \subsection{Setting out a few more conventions}\label{extra_conv} According to our current conventions the Greek indices $\alpha,\beta, \gamma$ are running from $1$ to $2n$. Here, however, we will need also indices that have the range $1,\dots,n$ for which we will use again the small Greek letters but already printed in black, e.g, $\pmb\alpha, \pmb\beta, \pmb\gamma, \dots$.  Primed bold indices will be used to indicate a shift by $n$, e.g., $\pmb\alpha'\overset{def}{=}\pmb\alpha+n$, and thus they will always have the range  $(n+1),\dots,2n$. If a number in brackets is used as an index (e.g., [15]), it means that we take a index in the range $1,\dots, n$ that is congruent to the original number in the brackets modulo $n$ (so if n=6, then [15] as an index
corresponds to 3). With this conventions, the constants $\pi_{\alpha\beta}$ from Section~\ref{prelim} are determined by 
 \begin{equation}
 \pi^{\pmb\alpha}_{\pmb{\bar\beta}}=0,\qquad \pi^{\pmb{\alpha'}}_{\pmb{\bar\beta}}=\delta^{\pmb\alpha}_{\pmb\beta}=-\pi^{\pmb{\alpha}}_{\pmb{\bar\beta'}}\qquad\text{($\delta^{\pmb\alpha}_{\pmb\beta}$ being the Kronecker delta).}
 \end{equation}
 
 Furthermore, the properties of the functions ${\mathcal S}_{\alpha\beta\gamma\delta}$ and ${\mathcal L}_{\alpha\beta}$ given by \eqref{properties-curvature} may be, equivalently, written as
\begin{equation}\label{properties_of_S}
{\mathcal S}_{\pmb{\alpha'\beta'\gamma'\delta'}}=\overline{{\mathcal S}_{\pmb{\alpha\beta\gamma\delta}}},\qquad {\mathcal S}_{\pmb{\alpha\beta'\gamma'\delta'}}=-\overline{{\mathcal S}_{\pmb{\alpha'\beta\gamma\delta}}},\qquad {\mathcal S}_{\pmb{\alpha\beta\gamma'\delta'}}=\overline{{\mathcal S}_{\pmb{\alpha'\beta'\gamma\delta}}}, 
\end{equation} 
\begin{equation}\label{properties_of_L}
{\mathcal L}_{\pmb{\alpha'\beta'}}=\overline{{\mathcal L}_{\pmb{\alpha\beta}}},\qquad {\mathcal L}_{\pmb{\alpha\beta'}}=-\overline{{\mathcal L}_{\pmb{\alpha'\beta}}}. 
\end{equation} 

Similarly, since by \eqref{Form_d2Gamma_ab} it can be easily verified that $(\mathfrak j{\Delta})_{\alpha\beta}=\Delta_{\alpha\beta}$, we have also the equations
\begin{equation}\label{properties_delta_ab}
{\Delta}_{\pmb{\alpha'\beta'}}=\overline{{\Delta}_{\pmb{\alpha\beta}}},\qquad {\Delta}_{\pmb{\alpha\beta'}}=-\overline{{\Delta}_{\pmb{\alpha'\beta}}}. 
\end{equation} 
 
\subsection{Introducing an appropriate coordinate system}  Let us fix an integral element $E\subset T_oN$  at the origin $o\in N$ of the associated exterior differential system, defined by the equations 
\begin{equation}\label{DefE}
 \mathcal S^\ast_{\alpha\beta\gamma\delta}\ =\  \mathcal V^\ast_{\alpha\beta\gamma}\ =\ \mathcal  L^\ast_{\alpha\beta}\ =\ \mathcal  M^\ast_{\alpha\beta} \ =\ \mathcal C^\ast_\alpha \ =\  \mathcal H^\ast_\alpha\ =\  \mathcal P^\ast \  =\ \mathcal Q^\ast\  =\ \mathcal R^\ast\ =\ 0
\end{equation}
and the structure equations \eqref{first_str_eq}, \eqref{str-eq-con},  \eqref{dGamma_ab}, \eqref{dphi_a}, \eqref{dpsi_1} and \eqref{dpsi_23}.
In order to compute the sequence of Cartan characters of the ideal $\mathcal I$ (cf. Section~\ref{ex_dif_sys}) we  need to introduce a real basis for the vector space  span$\{\eta_s,\theta^\alpha, \theta^{\bar\alpha}\}$. Let us take $\xi^\alpha,\zeta^\alpha$ to be the real one-forms defined by $\theta^\alpha=\xi^\alpha+i\zeta^\alpha$ and consider the basis $\{\eta_s,\xi^\alpha,\zeta^\alpha\}$. In general,  in the terminology of \cite{Bry} and \cite{BCGGG}, the choice of a bases here corresponds to a choice of an integral flag
\begin{equation*}
\{0\}=E_1\subset E_2\subset \dots \subset E_{d_1}=E
\end{equation*}
which we construct by dualizing the corresponding coframe of $E$.
Part of the difficulty in showing the Cartan's test and computing the corresponding Cartan characters of an ideal lies in the appropriate choice of the integral flag. Unfortunately, the natural choice of real coordinates that we have suggested above does not produce a Cartan-ordinary flag (i.e., a flag for which the Cartan's test is satisfied). Therefore, we will need a slightly more complicated construction here. 

Let  $\mu^{\pmb\alpha}$ and $\nu^{\pmb\alpha}$ be (real) one-forms on $N$ determined by the equations
\begin{equation*}
\begin{cases}
\xi^{n+1}=\mu^{1}+\zeta^{n}+\eta_3,\\
\xi^{\pmb\alpha'}=\mu^{\pmb\alpha}+\zeta^{[\pmb\alpha-1]},\ \text{if}\ \pmb\alpha\ne 1, \\
\zeta^{\pmb\beta'}=\nu^{\pmb\beta}+\mu^{[\pmb\beta-2]}, \ \text{for all}\ 1\le\pmb\beta\le n.
\end{cases}
\end{equation*}
Then, we choose a new basis of one-forms  $\{\epsilon^1,\dots,\epsilon^{4n+3}\}$ for span$\big\{\eta_s,\theta^\alpha, \theta^{\bar\alpha}\big\}$ by setting
\begin{equation}\label{new_basis}
\begin{gathered}
\epsilon^{\pmb\alpha}=\xi^{\pmb\alpha},\qquad \epsilon^{\pmb\alpha+n}=\zeta^{\pmb\alpha},\\ 
\epsilon^{2n+1}=\eta_1,\qquad \epsilon^{2n+2}=\eta_2,\qquad \epsilon^{2n+3}=\eta_3,\qquad \epsilon^{\pmb\alpha+2n+3}=\mu^{\pmb\alpha},\qquad \epsilon^{\pmb\alpha+3n+3}=\nu^{\pmb\alpha}.
\end{gathered}
\end{equation}

Notice that because of \eqref{properties_delta_ab}, we can restrict our attention only to the three-forms $\Delta_{\pmb{\alpha\beta}}$, $\Delta_{\pmb{\alpha\beta'}}$, $\Delta_{\pmb{\alpha}}$, $\Delta_{\pmb{\alpha'}}$ and $\Psi_s$. Substituting \eqref{new_basis} into \eqref{Form_d2Gamma_ab}, \eqref{Form_d2phi_a}, \eqref{Form_d2psi_1} and \eqref{Form_d2psi_2} gives:
\begin{multline}\label{Form_d2Gamma_ab_real}
 {\Delta}_{\pmb\alpha\pmb\beta}\  {=}\ \frac{1}{2}\Big({\mathcal S}^\ast_{\pmb{\alpha\beta\gamma\delta'}}-{\mathcal S}^\ast_{\pmb{\alpha\beta\delta\gamma'}}\Big)\wedge\xi^{\pmb\gamma}\wedge\xi^{\pmb\delta}\ -\   \Big(i{\mathcal S}^\ast_{\pmb{\alpha\beta\gamma\delta'}}+i{\mathcal S}^\ast_{\pmb{\alpha\beta\delta\gamma'}}+{\mathcal S}^\ast_{\pmb{\alpha\beta\gamma}[{\pmb{\delta}}+1]}+{\mathcal S}^\ast_{\pmb{\alpha\beta\gamma'}[{\pmb{\delta}+1}]\pmb'}\Big)\wedge\xi^{\pmb\gamma}\wedge\zeta^{\pmb\delta}\\
\ +\   \frac{1}{2}\Big({\mathcal S}^\ast_{\pmb{\alpha\beta\gamma\delta'}}-{\mathcal S}^\ast_{\pmb{\alpha\beta\delta\gamma'}}+i{\mathcal S}^\ast_{\pmb{\alpha\beta\gamma'}[{\pmb{\delta}+1}]\pmb'}-i{\mathcal S}^\ast_{\pmb{\alpha\beta\delta'}[{\pmb{\gamma}+1}]\pmb'}-i{\mathcal S}^\ast_{\pmb{\alpha\beta\gamma}[{\pmb{\delta}}+1]}+i{\mathcal S}^\ast_{\pmb{\alpha\beta\delta}[{\pmb{\gamma}}+1]}\\
+{\mathcal S}^\ast_{\pmb{\alpha\beta}[{\pmb{\gamma}}+1][{\pmb{\delta}+1}]\pmb'}-{\mathcal S}^\ast_{\pmb{\alpha\beta}[{\pmb{\delta}}+1][{\pmb{\gamma}+1}]\pmb'}\Big)\wedge\zeta^{\pmb\gamma}\wedge\zeta^{\pmb\delta}
\ +\ \Big({\mathcal V}^\ast_{\pmb{\alpha\beta\gamma}}+\overline{{\mathcal V}^\ast_{\pmb{\alpha'\beta'\gamma}}}\Big)\wedge\xi^{\pmb\gamma}\wedge\eta_1\\
+\Big(i{\mathcal V}^\ast_{\pmb{\alpha\beta\gamma}}-i\overline{{\mathcal V}^\ast_{\pmb{\alpha'\beta'\gamma}}}+{\mathcal V}^\ast_{\pmb{\alpha\beta}[{\pmb{\gamma}+1}]\pmb'}+\overline{{\mathcal V}^\ast_{\pmb{\alpha'\beta'}[{\pmb{\gamma}+1}]\pmb'}}\Big)\wedge\zeta^{\pmb\gamma}\wedge\eta_1\ 
-\ i\Big({\mathcal V}^\ast_{\pmb{\alpha\beta\gamma'}}-\overline{{\mathcal V}^\ast_{\pmb{\alpha'\beta'\gamma'}}}\Big)\wedge\xi^{\pmb\gamma}\wedge\eta_2\\
+\Big(-{\mathcal V}^\ast_{\pmb{\alpha\beta\gamma'}}-\overline{{\mathcal V}^\ast_{\pmb{\alpha'\beta'\gamma'}}}+i{\mathcal V}^\ast_{\pmb{\alpha\beta}[{\pmb{\gamma}+1}]}-i\overline{{\mathcal V}^\ast_{\pmb{\alpha'\beta'}[{\pmb{\gamma}+1}]}}\Big)\wedge\zeta^{\pmb\gamma}\wedge\eta_2\ 
+\ \Big({\mathcal M}^\ast_{\pmb{\alpha\beta}}+\overline{{\mathcal M}^\ast_{\pmb{\alpha'\beta'}}}\Big)\wedge\eta_1\wedge\eta_2\\
+\ \Big({\mathcal V}^\ast_{\pmb{\alpha\beta\gamma'}}+\overline{{\mathcal V}^\ast_{\pmb{\alpha'\beta'\gamma'}}}
-{\mathcal S}^\ast_{\pmb{\alpha\beta\gamma}1}-{\mathcal S}^\ast_{\pmb{\alpha\beta\gamma'}1\pmb'}+i{\mathcal S}^\ast_{\pmb{\alpha\beta\gamma}[3]}-i{\mathcal S}^\ast_{\pmb{\alpha\beta\gamma'}[3]\pmb'}\Big)\wedge\xi^{\pmb\gamma}\wedge\eta_3\\
+\Big(-i{\mathcal V}^\ast_{\pmb{\alpha\beta\gamma'}}+i\overline{{\mathcal V}^\ast_{\pmb{\alpha'\beta'\gamma'}}}-{\mathcal V}^\ast_{\pmb{\alpha\beta}[{\pmb{\gamma}+1}]}-\overline{{\mathcal V}^\ast_{\pmb{\alpha'\beta'}[{\pmb{\gamma}+1}]}}
+i{\mathcal S}^\ast_{\pmb{\alpha\beta\gamma'}1\pmb'}-i{\mathcal S}^\ast_{\pmb{\alpha\beta\gamma}1}
+{\mathcal S}^\ast_{\pmb{\alpha\beta}[{\pmb{\gamma}}+1]1\pmb'}-{\mathcal S}^\ast_{\pmb{\alpha\beta}[{\pmb{\gamma}+1}]\pmb'1}\\
-{\mathcal S}^\ast_{\pmb{\alpha\beta\gamma}[3]}
-{\mathcal S}^\ast_{\pmb{\alpha\beta\gamma'}[3]\pmb'}+i{\mathcal S}^\ast_{\pmb{\alpha\beta}[{\pmb{\gamma}+1}][3]\pmb'}
+i{\mathcal S}^\ast_{\pmb{\alpha\beta}[3][{\pmb{\gamma}+1}]\pmb'}\Big)\wedge\zeta^{\pmb\gamma}\wedge\eta_3\\
+\Big(i{\mathcal M}^\ast_{\pmb{\alpha\beta}}-i\overline{{\mathcal M}^\ast_{\pmb{\alpha'\beta'}}}
-{\mathcal V}^\ast_{\pmb{\alpha\beta} 1\pmb'}-\overline{{\mathcal V}^\ast_{\pmb{\alpha'\beta'}1\pmb'}}
-i{\mathcal V}^\ast_{\pmb{\alpha\beta}[3]\pmb'}+i\overline{{\mathcal V}^\ast_{\pmb{\alpha'\beta'}[3]\pmb'}}\Big)\wedge\eta_1\wedge\eta_3\\
+\Big(-2{\mathcal L}^\ast_{\pmb{\alpha\beta}}
-i{\mathcal V}^\ast_{\pmb{\alpha\beta} 1}+i\overline{{\mathcal V}^\ast_{\pmb{\alpha'\beta'}1}}
-{\mathcal V}^\ast_{\pmb{\alpha\beta}[3]}-\overline{{\mathcal V}^\ast_{\pmb{\alpha'\beta'}[3]}}\Big)\wedge\eta_2\wedge\eta_3\ +\ \dots\ ;
 \end{multline}

\begin{multline}\label{Form_d2Gamma_ab_prim_real}
 {\Delta}_{\pmb{\alpha\beta'}}\  {=}\ \frac{1}{2}\Big({\mathcal S}^\ast_{\pmb{\alpha\beta'\gamma\delta'}}-{\mathcal S}^\ast_{\pmb{\alpha\beta'\delta\gamma'}}\Big)\wedge\xi^{\pmb\gamma}\wedge\xi^{\pmb\delta}\ 
 -\   \Big(i{\mathcal S}^\ast_{\pmb{\alpha\beta'\gamma\delta'}}+i{\mathcal S}^\ast_{\pmb{\alpha\beta'\delta\gamma'}}+{\mathcal S}^\ast_{\pmb{\alpha\beta'\gamma}[{\pmb{\delta}}+1]}-\overline{{\mathcal S}^\ast_{\pmb{\alpha'\beta\gamma}[{\pmb{\delta}+1}]}}\Big)\wedge\xi^{\pmb\gamma}\wedge\zeta^{\pmb\delta}\\
\ +\   \frac{1}{2}\Big({\mathcal S}^\ast_{\pmb{\alpha\beta'\gamma\delta'}}-{\mathcal S}^\ast_{\pmb{\alpha\beta'\delta\gamma'}}
-i\overline{{\mathcal S}^\ast_{\pmb{\alpha'\beta\gamma}[{\pmb{\delta}+1}]}}+i\overline{{\mathcal S}^\ast_{\pmb{\alpha'\beta\delta}[{\pmb{\gamma}+1}]}}-i{\mathcal S}^\ast_{\pmb{\alpha\beta'\gamma}[{\pmb{\delta}}+1]}+i{\mathcal S}^\ast_{\pmb{\alpha\beta'\delta}[{\pmb{\gamma}}+1]}\\
+{\mathcal S}^\ast_{\pmb{\alpha\beta'}[{\pmb{\gamma}}+1][{\pmb{\delta}+1}]\pmb'}-{\mathcal S}^\ast_{\pmb{\alpha\beta'}[{\pmb{\delta}}+1][{\pmb{\gamma}+1}]\pmb'}\Big)\wedge\zeta^{\pmb\gamma}\wedge\zeta^{\pmb\delta}
\ +\ \Big({\mathcal V}^\ast_{\pmb{\alpha\beta'\gamma}}-\overline{{\mathcal V}^\ast_{\pmb{\alpha'\beta\gamma}}}\Big)\wedge\xi^{\pmb\gamma}\wedge\eta_1\\
+\Big(i{\mathcal V}^\ast_{\pmb{\alpha\beta'\gamma}}+i\overline{{\mathcal V}^\ast_{\pmb{\alpha'\beta\gamma}}}+{\mathcal V}^\ast_{\pmb{\alpha\beta'}[{\pmb{\gamma}+1}]\pmb'}-\overline{{\mathcal V}^\ast_{\pmb{\alpha'\beta}[{\pmb{\gamma}+1}]\pmb'}}\Big)\wedge\zeta^{\pmb\gamma}\wedge\eta_1\ 
-\ i\Big({\mathcal V}^\ast_{\pmb{\alpha\beta'\gamma'}}+\overline{{\mathcal V}^\ast_{\pmb{\alpha'\beta\gamma'}}}\Big)\wedge\xi^{\pmb\gamma}\wedge\eta_2\\
+\Big(-{\mathcal V}^\ast_{\pmb{\alpha\beta'\gamma'}}+\overline{{\mathcal V}^\ast_{\pmb{\alpha'\beta\gamma'}}}
+i{\mathcal V}^\ast_{\pmb{\alpha\beta'}[{\pmb{\gamma}+1}]}+i\overline{{\mathcal V}^\ast_{\pmb{\alpha'\beta}[{\pmb{\gamma}+1}]}}\Big)\wedge\zeta^{\pmb\gamma}\wedge\eta_2\ 
+\ \Big({\mathcal M}^\ast_{\pmb{\alpha\beta'}}-\overline{{\mathcal M}^\ast_{\pmb{\alpha'\beta}}}\Big)\wedge\eta_1\wedge\eta_2\\
+\Big({\mathcal V}^\ast_{\pmb{\alpha\beta'\gamma'}}-\overline{{\mathcal V}^\ast_{\pmb{\alpha'\beta\gamma'}}}
-{\mathcal S}^\ast_{\pmb{\alpha\beta'\gamma}1}+\overline{{\mathcal S}^\ast_{\pmb{\alpha'\beta\gamma}1}}
+i{\mathcal S}^\ast_{\pmb{\alpha\beta'\gamma}[3]}+i\overline{{\mathcal S}^\ast_{\pmb{\alpha'\beta\gamma}[3]}}\Big)\wedge\xi^{\pmb\gamma}\wedge\eta_3\\
+\Big(-i{\mathcal V}^\ast_{\pmb{\alpha\beta'\gamma'}}-i\overline{{\mathcal V}^\ast_{\pmb{\alpha'\beta\gamma'}}}
-{\mathcal V}^\ast_{\pmb{\alpha\beta'}[{\pmb{\gamma}+1}]}+\overline{{\mathcal V}^\ast_{\pmb{\alpha'\beta}[{\pmb{\gamma}+1}]}}
-i\overline{{\mathcal S}^\ast_{\pmb{\alpha'\beta\gamma}1}}-i{\mathcal S}^\ast_{\pmb{\alpha\beta'\gamma}1}
+{\mathcal S}^\ast_{\pmb{\alpha\beta'}[{\pmb{\gamma}}+1]1\pmb'}-{\mathcal S}^\ast_{\pmb{\alpha\beta'}[{\pmb{\gamma}+1}]\pmb'1}\\
-{\mathcal S}^\ast_{\pmb{\alpha\beta'\gamma}[3]}
+\overline{{\mathcal S}^\ast_{\pmb{\alpha'\beta\gamma}[3]}}+i{\mathcal S}^\ast_{\pmb{\alpha\beta'}[{\pmb{\gamma}+1}][3]\pmb'}
+i{\mathcal S}^\ast_{\pmb{\alpha\beta'}[3][{\pmb{\gamma}+1}]\pmb'}\Big)\wedge\zeta^{\pmb\gamma}\wedge\eta_3\\
+\Big(i{\mathcal M}^\ast_{\pmb{\alpha\beta'}}+i\overline{{\mathcal M}^\ast_{\pmb{\alpha'\beta}}}
-{\mathcal V}^\ast_{\pmb{\alpha\beta'} 1\pmb'}+\overline{{\mathcal V}^\ast_{\pmb{\alpha'\beta}1\pmb'}}
-i{\mathcal V}^\ast_{\pmb{\alpha\beta'}[3]\pmb'}-i\overline{{\mathcal V}^\ast_{\pmb{\alpha'\beta}[3]\pmb'}}\Big)\wedge\eta_1\wedge\eta_3\\
+\Big(-2{\mathcal L}^\ast_{\pmb{\alpha\beta'}}
-i{\mathcal V}^\ast_{\pmb{\alpha\beta'} 1}-i\overline{{\mathcal V}^\ast_{\pmb{\alpha'\beta}1}}
-{\mathcal V}^\ast_{\pmb{\alpha\beta'}[3]}+\overline{{\mathcal V}^\ast_{\pmb{\alpha'\beta}[3]}}\Big)\wedge\eta_2\wedge\eta_3\ +\ \dots\ ;
 \end{multline}

\begin{multline}\label{Form_d2phi_a_real}
 {\Delta}_{\pmb\alpha}\  {=}\ -\frac{i}{2}\Big({\mathcal V}^\ast_{\pmb{\alpha\beta\gamma'}}-{\mathcal V}^\ast_{\pmb{\alpha\beta'\gamma}}\Big)\wedge\xi^{\pmb\beta}\wedge\xi^{\pmb\gamma}\ 
 -\   \Big({\mathcal V}^\ast_{\pmb{\alpha\beta\gamma'}}+{\mathcal V}^\ast_{\pmb{\alpha\beta'\gamma}}-i{\mathcal V}^\ast_{\pmb{\alpha\beta}[{\pmb{\gamma}}+1]}
 -i{\mathcal V}^\ast_{\pmb{\alpha\beta'}[{\pmb{\gamma}+1}]\pmb'}\Big)\wedge\xi^{\pmb\beta}\wedge\zeta^{\pmb\gamma}\\
\ +\   \frac{1}{2}\Big(-i{\mathcal V}^\ast_{\pmb{\alpha\beta\gamma'}}+i{\mathcal V}^\ast_{\pmb{\alpha\beta'\gamma}}
+{\mathcal V}^\ast_{\pmb{\alpha\beta'}[{\pmb{\gamma}+1}]\pmb'}-{\mathcal V}^\ast_{\pmb{\alpha\gamma'}[{\pmb{\beta}+1}]\pmb'}
-{\mathcal V}^\ast_{\pmb{\alpha\beta}[{\pmb{\gamma}}+1]}+{\mathcal V}^\ast_{\pmb{\alpha\gamma}[{\pmb{\beta}}+1]}\\
-i{\mathcal V}^\ast_{\pmb{\alpha}[{\pmb{\beta}}+1][{\pmb{\gamma}+1}]\pmb'}+i{\mathcal V}^\ast_{\pmb{\alpha}[{\pmb{\gamma}}+1][{\pmb{\beta}+1}]\pmb'}\Big)\wedge\zeta^{\pmb\beta}\wedge\zeta^{\pmb\gamma}
\ +\ \Big(-{\mathcal L}^\ast_{\pmb{\alpha\beta'}}+{\mathcal M}^\ast_{\pmb{\alpha\beta}}\Big)\wedge\xi^{\pmb\beta}\wedge\eta_1\\
+\Big(i{\mathcal L}^\ast_{\pmb{\alpha\beta'}}+i{\mathcal M}^\ast_{\pmb{\alpha\beta}}+{\mathcal L}^\ast_{\pmb{\alpha}[{\pmb{\beta}+1}]}
+{\mathcal M}^\ast_{\pmb{\alpha}[{\pmb{\beta}+1}]\pmb'}\Big)\wedge\zeta^{\pmb\beta}\wedge\eta_1\ 
+\ \Big(i{\mathcal L}^\ast_{\pmb{\alpha\beta}}-i{\mathcal M}^\ast_{\pmb{\alpha\beta'}}\Big)\wedge\xi^{\pmb\beta}\wedge\eta_2\\
+\Big(-{\mathcal L}^\ast_{\pmb{\alpha\beta}}-{\mathcal M}^\ast_{\pmb{\alpha\beta'}}+i{\mathcal L}^\ast_{\pmb{\alpha}[{\pmb{\beta}+1}]\pmb'}
+i{\mathcal M}^\ast_{\pmb{\alpha}[{\pmb{\beta}+1}]}\Big)\wedge\zeta^{\pmb\beta}\wedge\eta_2\ 
+\ \Big(i\overline{{\mathcal C}^\ast_{\pmb{\alpha'}}}+{\mathcal H}^\ast_{\pmb{\alpha}}\Big)\wedge\eta_1\wedge\eta_2\\
+\Big({\mathcal L}^\ast_{\pmb{\alpha\beta}}+{\mathcal M}^\ast_{\pmb{\alpha\beta'}}
+i{\mathcal V}^\ast_{\pmb{\alpha\beta}1}+i{\mathcal V}^\ast_{\pmb{\alpha\beta'}1\pmb'}+{\mathcal V}^\ast_{\pmb{\alpha\beta}[3]}
-{\mathcal V}^\ast_{\pmb{\alpha\beta'}[3]\pmb'}\Big)\wedge\xi^{\pmb\beta}\wedge\eta_3\\
+\Big(i{\mathcal L}^\ast_{\pmb{\alpha\beta}}-i{\mathcal M}^\ast_{\pmb{\alpha\beta'}}+{\mathcal L}^\ast_{\pmb{\alpha}[{\pmb{\beta}+1}]\pmb'}
-{\mathcal M}^\ast_{\pmb{\alpha}[{\pmb{\beta}+1}]}
+{\mathcal V}^\ast_{\pmb{\alpha\beta'}1\pmb'}-{\mathcal V}^\ast_{\pmb{\alpha\beta}1}
-i{\mathcal V}^\ast_{\pmb{\alpha}[{\pmb{\beta}}+1]1\pmb'}+i{\mathcal V}^\ast_{\pmb{\alpha}[{\pmb{\beta}+1}]\pmb'1}\\
+i{\mathcal V}^\ast_{\pmb{\alpha\beta'}[3]\pmb'}
+i{\mathcal V}^\ast_{\pmb{\alpha\beta}[3]}+{\mathcal V}^\ast_{\pmb{\alpha}[{\pmb{\beta}+1}][3]\pmb'}
+{\mathcal V}^\ast_{\pmb{\alpha}[3][{\pmb{\beta}+1}]\pmb'}\Big)\wedge\zeta^{\pmb\beta}\wedge\eta_3\\
+\Big(\overline{{\mathcal C}^\ast_{\pmb{\alpha'}}}+i{\mathcal H}^\ast_{\pmb{\alpha}}
-{\mathcal L}^\ast_{\pmb{\alpha}1}-{\mathcal M}^\ast_{\pmb{\alpha}1\pmb'}
+i{\mathcal L}^\ast_{\pmb{\alpha}[3]}-i{\mathcal M}^\ast_{\pmb{\alpha}[3]\pmb'}\Big)\wedge\eta_1\wedge\eta_3\\
+\Big(2i{\mathcal C}^\ast_{\pmb{\alpha}}
-i{\mathcal L}^\ast_{\pmb{\alpha}1\pmb'}-i{\mathcal M}^\ast_{\pmb{\alpha}1}
+{\mathcal L}^\ast_{\pmb{\alpha}[3]\pmb'}-{\mathcal M}^\ast_{\pmb{\alpha}[3]}\Big)\wedge\eta_2\wedge\eta_3\ +\ \dots\ ;
 \end{multline}

\begin{multline}\label{Form_d2phi_a_prim_real}
 {\Delta}_{\pmb{\alpha'}}\  {=}\ \frac{i}{2}\Big({\mathcal V}^\ast_{\pmb{\alpha'\beta'\gamma}}-{\mathcal V}^\ast_{\pmb{\alpha'\gamma'\beta}}\Big)\wedge\xi^{\pmb\beta}\wedge\xi^{\pmb\gamma}\ 
 -\   \Big({\mathcal V}^\ast_{\pmb{\alpha'\beta'\gamma}}+{\mathcal V}^\ast_{\pmb{\alpha'\gamma'\beta}}-i{\mathcal V}^\ast_{\pmb{\alpha'\beta'}[{\pmb{\gamma}}+1]\pmb'}
 -i{\mathcal V}^\ast_{\pmb{\alpha'\beta}[{\pmb{\gamma}+1}]}\Big)\wedge\xi^{\pmb\beta}\wedge\zeta^{\pmb\gamma}\\
\ +\   \frac{1}{2}\Big(i{\mathcal V}^\ast_{\pmb{\alpha'\beta'\gamma}}-i{\mathcal V}^\ast_{\pmb{\alpha'\gamma'\beta}}
+{\mathcal V}^\ast_{\pmb{\alpha'\beta'}[{\pmb{\gamma}+1}]\pmb'}-{\mathcal V}^\ast_{\pmb{\alpha'\gamma'}[{\pmb{\beta}+1}]\pmb'}
-{\mathcal V}^\ast_{\pmb{\alpha'\beta}[{\pmb{\gamma}}+1]}+{\mathcal V}^\ast_{\pmb{\alpha'\gamma}[{\pmb{\beta}}+1]}\\
+i{\mathcal V}^\ast_{\pmb{\alpha'}[{\pmb{\beta}}+1]\pmb'[{\pmb{\gamma}+1}]}-i{\mathcal V}^\ast_{\pmb{\alpha'}[{\pmb{\gamma}}+1]\pmb'[{\pmb{\beta}+1}]}\Big)\wedge\zeta^{\pmb\beta}\wedge\zeta^{\pmb\gamma}
\ +\ \Big(-\overline{{\mathcal L}^\ast_{\pmb{\alpha\beta}}}+{\mathcal M}^\ast_{\pmb{\alpha'\beta}}\Big)\wedge\xi^{\pmb\beta}\wedge\eta_1\\
+\Big(i\overline{{\mathcal L}^\ast_{\pmb{\alpha\beta}}}+i{\mathcal M}^\ast_{\pmb{\alpha'\beta}}+{\mathcal L}^\ast_{\pmb{\alpha'}[{\pmb{\beta}+1}]}
+{\mathcal M}^\ast_{\pmb{\alpha'}[{\pmb{\beta}+1}]\pmb'}\Big)\wedge\zeta^{\pmb\beta}\wedge\eta_1\ 
+\ \Big(i{\mathcal L}^\ast_{\pmb{\alpha'\beta}}-i{\mathcal M}^\ast_{\pmb{\alpha'\beta'}}\Big)\wedge\xi^{\pmb\beta}\wedge\eta_2\\
+\Big(-{\mathcal L}^\ast_{\pmb{\alpha'\beta}}-{\mathcal M}^\ast_{\pmb{\alpha'\beta'}}+i\overline{{\mathcal L}^\ast_{\pmb{\alpha}[{\pmb{\beta}+1}]}}
+i{\mathcal M}^\ast_{\pmb{\alpha'}[{\pmb{\beta}+1}]}\Big)\wedge\zeta^{\pmb\beta}\wedge\eta_2\ 
+\ \Big(-i\overline{{\mathcal C}^\ast_{\pmb{\alpha}}}+{\mathcal H}^\ast_{\pmb{\alpha'}}\Big)\wedge\eta_1\wedge\eta_2\\
+\Big({\mathcal L}^\ast_{\pmb{\alpha'\beta}}+{\mathcal M}^\ast_{\pmb{\alpha'\beta'}}
+i{\mathcal V}^\ast_{\pmb{\alpha'\beta'}1\pmb'}+i{\mathcal V}^\ast_{\pmb{\alpha'\beta}1}+{\mathcal V}^\ast_{\pmb{\alpha'\beta}[3]}
-{\mathcal V}^\ast_{\pmb{\alpha'\beta'}[3]\pmb'}\Big)\wedge\xi^{\pmb\beta}\wedge\eta_3\\
+\Big(i{\mathcal L}^\ast_{\pmb{\alpha'\beta}}-i{\mathcal M}^\ast_{\pmb{\alpha'\beta'}}+\overline{{\mathcal L}^\ast_{\pmb{\alpha}[{\pmb{\beta}+1}]}}
-{\mathcal M}^\ast_{\pmb{\alpha'}[{\pmb{\beta}+1}]}
+{\mathcal V}^\ast_{\pmb{\alpha'\beta'}1\pmb'}-{\mathcal V}^\ast_{\pmb{\alpha'\beta}1}
+i{\mathcal V}^\ast_{\pmb{\alpha'}[{\pmb{\beta}}+1]\pmb'1}-i{\mathcal V}^\ast_{\pmb{\alpha'}[{\pmb{\beta}+1}]1\pmb'}\\
+i{\mathcal V}^\ast_{\pmb{\alpha'\beta}[3]}
+i{\mathcal V}^\ast_{\pmb{\alpha'\beta'}[3]\pmb'}+{\mathcal V}^\ast_{\pmb{\alpha'}[{\pmb{\beta}+1}]\pmb'[3]}
+{\mathcal V}^\ast_{\pmb{\alpha'}[3]\pmb'[{\pmb{\beta}+1}]}\Big)\wedge\zeta^{\pmb\beta}\wedge\eta_3\\
+\Big(-\overline{{\mathcal C}^\ast_{\pmb{\alpha}}}+i{\mathcal H}^\ast_{\pmb{\alpha'}}
-{\mathcal L}^\ast_{\pmb{\alpha'}1}-{\mathcal M}^\ast_{\pmb{\alpha'}1\pmb'}
+i{\mathcal L}^\ast_{\pmb{\alpha'}[3]}-i{\mathcal M}^\ast_{\pmb{\alpha'}[3]\pmb'}\Big)\wedge\eta_1\wedge\eta_3\\
+\Big(2i{\mathcal C}^\ast_{\pmb{\alpha'}}
-i\overline{{\mathcal L}^\ast_{\pmb{\alpha}1}}-i{\mathcal M}^\ast_{\pmb{\alpha'}1}
+\overline{{\mathcal L}^\ast_{\pmb{\alpha}[3]}}-{\mathcal M}^\ast_{\pmb{\alpha'}[3]}\Big)\wedge\eta_2\wedge\eta_3\ +\ \dots\ ;
 \end{multline}

\begin{multline}\label{Form_d2psi_1_real}
 {\Psi}_1\  {=}\ 2\Big({\mathcal L}^\ast_{\pmb{\alpha\beta'}}-{\mathcal L}^\ast_{\pmb{\alpha'\beta}}\Big)\wedge\xi^{\pmb\alpha}\wedge\xi^{\pmb\beta}\ 
 -\   4\Big(i{\mathcal L}^\ast_{\pmb{\alpha\beta'}}+i{\mathcal L}^\ast_{\pmb{\alpha'\beta}}+\overline{{\mathcal L}^\ast_{\pmb{\alpha}[{\pmb{\beta}}+1]}}
 +{\mathcal L}^\ast_{\pmb{\alpha}[{\pmb{\beta}+1}]}\Big)\wedge\xi^{\pmb\alpha}\wedge\zeta^{\pmb\beta}\\
\ +\   2\Big({\mathcal L}^\ast_{\pmb{\alpha\beta'}}-{\mathcal L}^\ast_{\pmb{\alpha'\beta}}
-i{\mathcal L}^\ast_{\pmb{\alpha}[{\pmb{\beta}+1}]}+i{\mathcal L}^\ast_{\pmb{\beta}[{\pmb{\alpha}+1}]}
+i\overline{{\mathcal L}^\ast_{\pmb{\alpha}[{\pmb{\beta}}+1]}}-i\overline{{\mathcal L}^\ast_{\pmb{\beta}[{\pmb{\alpha}}+1]}}\\
-{\mathcal L}^\ast_{[{\pmb{\alpha}}+1]\pmb'[{\pmb{\beta}+1}]}+{\mathcal L}^\ast_{[{\pmb{\beta}}+1]\pmb'[{\pmb{\alpha}+1}]}\Big)\wedge\zeta^{\pmb\alpha}\wedge\zeta^{\pmb\beta}
\ +\ 4\Big({\mathcal C}^\ast_{\pmb{\alpha}}+\overline{{\mathcal C}^\ast_{\pmb{\alpha}}}\Big)\wedge\xi^{\pmb\alpha}\wedge\eta_1\\
+4\Big(i{\mathcal C}^\ast_{\pmb{\alpha}}-i\overline{{\mathcal C}^\ast_{\pmb{\alpha}}}+{\mathcal C}^\ast_{[{\pmb{\alpha}+1}]\pmb'}
+\overline{{\mathcal C}^\ast_{[{\pmb{\alpha}+1}]\pmb'}}\Big)\wedge\zeta^{\pmb\alpha}\wedge\eta_1\ 
-\ 4i\Big({\mathcal C}^\ast_{\pmb{\alpha'}}-\overline{{\mathcal C}^\ast_{\pmb{\alpha'}}}\Big)\wedge\xi^{\pmb\alpha}\wedge\eta_2\\
-\ 4\Big({\mathcal C}^\ast_{\pmb{\alpha'}}+\overline{{\mathcal C}^\ast_{\pmb{\alpha'}}}-i{\mathcal C}^\ast_{[{\pmb{\alpha}+1}]}
+i\overline{{\mathcal C}^\ast_{[{\pmb{\alpha}+1}]}} \Big)\wedge\zeta^{\pmb\alpha}\wedge\eta_2\ 
+\ \Big({\mathcal P}^\ast+\overline{\mathcal P}^\ast\Big)\wedge\eta_1\wedge\eta_2\\
+\ 4\Big({\mathcal C}^\ast_{\pmb{\alpha'}}+\overline{{\mathcal C}^\ast_{\pmb{\alpha'}}}
-{\mathcal L}^\ast_{\pmb{\alpha}1}-\overline{{\mathcal L}^\ast_{\pmb{\alpha}1}}+i{\mathcal L}^\ast_{\pmb{\alpha}[3]}
-i\overline{{\mathcal L}^\ast_{\pmb{\alpha}[3]}}\Big)\wedge\xi^{\pmb\alpha}\wedge\eta_3\\
+\ 4\Big(-i{\mathcal C}^\ast_{\pmb{\alpha'}}+i\overline{{\mathcal C}^\ast_{\pmb{\alpha'}}}-{\mathcal C}^\ast_{[{\pmb{\alpha}+1}]}
-\overline{{\mathcal C}^\ast_{[{\pmb{\alpha}+1}]}}
-i{\mathcal L}^\ast_{\pmb{\alpha}1}+i\overline{{\mathcal L}^\ast_{\pmb{\alpha}1}}
-{\mathcal L}^\ast_{[{\pmb{\alpha}}+1]\pmb'1}+{\mathcal L}^\ast_{[{\pmb{\alpha}+1}]1\pmb'}\\
-{\mathcal L}^\ast_{\pmb{\alpha}[3]}
-\overline{{\mathcal L}^\ast_{\pmb{\alpha}[3]}}+i{\mathcal L}^\ast_{[{\pmb{\alpha}+1}]\pmb'[3]}
+i{\mathcal L}^\ast_{[3]\pmb'[{\pmb{\alpha}+1}]}\Big)\wedge\zeta^{\pmb\alpha}\wedge\eta_3\\
+\ \Big(i{\mathcal P}^\ast-i\overline{\mathcal P}^\ast
-4{\mathcal C}^\ast_{1\pmb'}-4\overline{{\mathcal C}^\ast_{1\pmb'}}
-4i{\mathcal C}^\ast_{[3]\pmb'}+4i\overline{{\mathcal C}^\ast_{[3]\pmb'}}\Big)\wedge\eta_1\wedge\eta_3\\
+\Big(2{\mathcal R}^\ast
-4i{\mathcal C}^\ast_{1}+4i\overline{{\mathcal C}^\ast_{1}}
-4{\mathcal C}^\ast_{[3]}-4\overline{{\mathcal C}^\ast_{[3]}}\Big)\wedge\eta_2\wedge\eta_3\ +\ \dots\ ;
 \end{multline}

\begin{multline}\label{Form_d2psi_23_real}
 {\Psi}_2-i\Psi_3\  {=}\ -2i\Big({\mathcal M}^\ast_{\pmb{\alpha'\beta}}-{\mathcal M}^\ast_{\pmb{\alpha\beta'}}\Big)\wedge\xi^{\pmb\alpha}\wedge\xi^{\pmb\beta}\ 
+\   4\Big({\mathcal M}^\ast_{\pmb{\alpha'\beta}}+{\mathcal M}^\ast_{\pmb{\alpha\beta'}}
-i{\mathcal M}^\ast_{\pmb{\alpha'}[{\pmb{\beta}}+1]\pmb'}
 -i{\mathcal M}^\ast_{\pmb{\alpha}[{\pmb{\beta}+1}]}\Big)\wedge\xi^{\pmb\alpha}\wedge\zeta^{\pmb\beta}\\
\ +\   2\Big(-i{\mathcal M}^\ast_{\pmb{\alpha'\beta}}+i{\mathcal M}^\ast_{\pmb{\alpha\beta'}}
+{\mathcal M}^\ast_{\pmb{\alpha}[{\pmb{\beta}+1}]}-{\mathcal M}^\ast_{\pmb{\beta}[{\pmb{\alpha}+1}]}
-{\mathcal M}^\ast_{\pmb{\alpha'}[{\pmb{\beta}}+1]'}+{\mathcal M}^\ast_{\pmb{\beta'}[{\pmb{\alpha}}+1]'}\\
-i{\mathcal M}^\ast_{[{\pmb{\alpha}}+1]\pmb'[{\pmb{\beta}+1}]}+i{\mathcal M}^\ast_{[{\pmb{\beta}}+1]\pmb'[{\pmb{\alpha}+1}]}\Big)\wedge\zeta^{\pmb\alpha}\wedge\zeta^{\pmb\beta}
\ -\ 4\Big(i{\mathcal C}^\ast_{\pmb{\alpha'}}+{\mathcal H}^\ast_{\pmb{\alpha}}\Big)\wedge\xi^{\pmb\alpha}\wedge\eta_1\\
+4\Big(-{\mathcal C}^\ast_{\pmb{\alpha'}}-i{\mathcal H}^\ast_{\pmb{\alpha}}+i{\mathcal C}^\ast_{[{\pmb{\alpha}+1}]}
-{\mathcal H}^\ast_{[{\pmb{\alpha}+1}]\pmb'}\Big)\wedge\zeta^{\pmb\alpha}\wedge\eta_1\ 
+\ 4i\Big(i{\mathcal C}^\ast_{\pmb{\alpha}}+{\mathcal H}^\ast_{\pmb{\alpha'}}\Big)\wedge\xi^{\pmb\alpha}\wedge\eta_2\\
-\ 4\Big(i{\mathcal C}^\ast_{\pmb{\alpha}}-{\mathcal H}^\ast_{\pmb{\alpha'}}+{\mathcal C}^\ast_{[{\pmb{\alpha}+1}]\pmb'}
+i{\mathcal H}^\ast_{[{\pmb{\alpha}+1}]} \Big)\wedge\zeta^{\pmb\alpha}\wedge\eta_2\ 
+\ \Big(i{\mathcal R}^\ast+{\mathcal Q}^\ast\Big)\wedge\eta_1\wedge\eta_2\\
+\ 4\Big(i{\mathcal C}^\ast_{\pmb{\alpha}}-{\mathcal H}^\ast_{\pmb{\alpha'}}
-i{\mathcal M}^\ast_{\pmb{\alpha'}1\pmb'}-i{\mathcal M}^\ast_{\pmb{\alpha}1}-{\mathcal M}^\ast_{\pmb{\alpha}[3]}
+{\mathcal M}^\ast_{\pmb{\alpha'}[3]\pmb'}\Big)\wedge\xi^{\pmb\alpha}\wedge\eta_3\\
+\ 4\Big(-{\mathcal C}^\ast_{\pmb{\alpha}}+i{\mathcal H}^\ast_{\pmb{\alpha'}}+i{\mathcal C}^\ast_{[{\pmb{\alpha}+1}]\pmb'}
+{\mathcal H}^\ast_{[{\pmb{\alpha}+1}]}
+{\mathcal M}^\ast_{\pmb{\alpha}1}-{\mathcal M}^\ast_{\pmb{\alpha'}1'}
-i{\mathcal M}^\ast_{[{\pmb{\alpha}}+1]\pmb'1}+i{\mathcal M}^\ast_{[{\pmb{\alpha}+1}]1\pmb'}\\
-i{\mathcal M}^\ast_{\pmb{\alpha}[3]}
-i{\mathcal M}^\ast_{\pmb{\alpha'}[3]\pmb'}-{\mathcal M}^\ast_{[{\pmb{\alpha}+1}]\pmb'[3]}
-{\mathcal M}^\ast_{[3]\pmb'[{\pmb{\alpha}+1}]}\Big)\wedge\zeta^{\pmb\alpha}\wedge\eta_3\\
+\ \Big({\mathcal R}^\ast+i{\mathcal Q}^\ast
-4i{\mathcal C}^\ast_{1}+4{\mathcal H}^\ast_{1\pmb'}
-4{\mathcal C}^\ast_{[3]}+4i{\mathcal H}^\ast_{[3]\pmb'}\Big)\wedge\eta_1\wedge\eta_3\\
+\Big(-2i{\mathcal P}^\ast
+4{\mathcal C}^\ast_{1\pmb'}+4i{\mathcal H}^\ast_{1}
+4i{\mathcal C}^\ast_{[3]\pmb'}+4{\mathcal H}^\ast_{[3]}\Big)\wedge\eta_2\wedge\eta_3\ +\ \dots\ .
 \end{multline} 
  
 In the above identities we have omitted all the terms involving wedge products with basis one-forms \eqref{new_basis} of index grater than $2n+3$ (and have replaced them by "$\dots$") since they will turn out to be irrelevant for our further considerations here.    

For each integer $1\le \lambda \le d_1$ ($d_1$ is given by \eqref{def_d_1}), we let $\mathfrak F_\lambda$ be the real subspace $\mathfrak F_\lambda\subset T_o^\ast N$ generated by the real and the imaginary parts of all one forms $\Phi$ for which the term $\Phi\wedge{\mathcal \epsilon}_a\wedge{\mathcal \epsilon}_b$ with $1\le a,b\le \lambda$ appears on the RHS of \eqref{Form_d2Gamma_ab_real}, \eqref{Form_d2Gamma_ab_prim_real}, \eqref{Form_d2phi_a_real}, \eqref{Form_d2phi_a_prim_real}, \eqref{Form_d2psi_1_real} or \eqref{Form_d2psi_23_real}. Then, the character sequence $v_1, v_2, \dots, v_{d_1}$ of the ideal $\mathcal I$ corresponding to the fixed basis \eqref{new_basis} is given by (cf. Section~\ref{ex_dif_sys}) 
\begin{equation}
\begin{cases}
v_1=0;\\
v_\lambda=\dim(\mathfrak F_\lambda/\mathfrak F_{\lambda-1})=\dim \mathfrak F_\lambda-\dim \mathfrak F_{(\lambda-1)},\ \text{if}\ 2\le \lambda\le d_1
\end{cases}
\end{equation}

\subsection{The characters $v_2,\dots,v_n$} Let us fix an integer number $\pmb\lambda$ between  $2$ and $n$. By definition,  the (real) vector space $\mathfrak F_{\pmb\lambda}$ is generated by the real and imaginary parts of the one-forms
\begin{equation*}
\begin{gathered}
{\mathcal S}^\ast_{\pmb{\alpha\beta\gamma\delta'}}-{\mathcal S}^\ast_{\pmb{\alpha\beta\delta\gamma'}},\qquad
{\mathcal S}^\ast_{\pmb{\alpha\beta'\gamma\delta'}}-{\mathcal S}^\ast_{\pmb{\alpha\beta'\delta\gamma'}},\\
{\mathcal V}^\ast_{\pmb{\alpha\gamma\delta'}}-{\mathcal V}^\ast_{\pmb{\alpha\gamma'\delta}},\qquad
{\mathcal V}^\ast_{\pmb{\alpha'\gamma'\delta}}-{\mathcal V}^\ast_{\pmb{\alpha'\delta'\gamma}},\qquad
{\mathcal L}^\ast_{\pmb{\gamma\delta'}}-{\mathcal L}^\ast_{\pmb{\gamma'\delta}},\qquad
{\mathcal M}^\ast_{\pmb{\gamma'\delta}}-{\mathcal M}^\ast_{\pmb{\gamma\delta'}},
\end{gathered}
\end{equation*}
where $1\ \le\ {\pmb\alpha}, {\pmb\beta}\ \le n$ and $1\ \le\ {\pmb\gamma}, {\pmb\delta}\ \le \pmb\lambda$.

Let us introduce the one-forms
\begin{equation}
X_{\pmb{\alpha\beta\gamma\delta}}\overset{def}{\ =\ }\frac{1}{2}\Big({\mathcal S}^\ast_{\pmb{\alpha\beta\gamma\delta'}}-{\mathcal S}^\ast_{\pmb{\alpha\beta\delta\gamma'}}\Big),\qquad
Y_{\pmb{\alpha\beta\gamma\delta}}\overset{def}{\ =\ }\frac{1}{4}\Big({\mathcal S}^\ast_{\pmb{\alpha\beta\gamma\delta'}}+{\mathcal S}^\ast_{\pmb{\alpha\beta\delta\gamma'}}+{\mathcal S}^\ast_{\pmb{\alpha\gamma\delta\beta'}}+{\mathcal S}^\ast_{\pmb{\beta\gamma\delta\alpha'}}\Big),
\end{equation}
Then, as it can be easily verified, $X_{\pmb{\alpha\beta\gamma\delta}}$ is symmetric in $\pmb{\alpha,\beta}$ and skew-symmetric in $\pmb{\gamma,\delta}$. Furthermore, it has the property
\begin{equation}\label{Bianchi_X4}
X_{\pmb{\alpha\beta\gamma\delta}}+X_{\pmb{\alpha\gamma\delta\beta}}+X_{\pmb{\alpha\delta\beta\gamma}}=0.
\end{equation}
Whereas $Y_{\pmb{\alpha\beta\gamma\delta}}$ is totally symmetric in  $\pmb{\alpha,\beta,\gamma,\delta}$. By a straightforward substitution, one can immediately verify  the identity
\begin{equation}\label{expression_S_prim}
{\mathcal S}^\ast_{\pmb{\alpha\beta\gamma\delta'}}=\frac{1}{2}\Big(X_{\pmb{\alpha\beta\gamma\delta}}+X_{\pmb{\beta\gamma\alpha\delta}}+X_{\pmb{\gamma\alpha\beta\delta}}\Big)+Y_{\pmb{\alpha\beta\gamma\delta}}
\end{equation} 

Next, we will choose a reduced set of generators for the linear space 
\begin{equation}\label{space_X4}
\text{span}\Big\{\mathbb Re(X_{\pmb{\alpha\beta\gamma\delta}}),\ \mathbb Im (X_{\pmb{\alpha\beta\gamma\delta}})\ \Big|\  1\ \le\ {\pmb\alpha}, {\pmb\beta}\ \le n,\  1\ \le\ {\pmb\gamma}, {\pmb\delta}\ \le \pmb\lambda\Big\}
\end{equation}
 considered as a subspace in $\mathfrak F_{\pmb\lambda}/\mathfrak F_{\pmb\lambda-1}$. Notice that by \eqref{Bianchi_X4}, for any $1\le \pmb\alpha,\pmb\beta,\pmb\gamma\le \pmb\lambda -1$, we have
\begin{equation*}
X_{\pmb{\alpha\beta\gamma}\pmb\lambda}+\underbrace{{X}_{\pmb{\alpha}\pmb\lambda{\pmb{\beta\gamma}}}}_{\in \mathfrak F_{(\pmb\lambda-1)}}+X_{\pmb{\alpha\gamma}\pmb\lambda{\pmb\beta}}=0
\end{equation*} 
and thus, modulo $\mathfrak F_{\pmb\lambda-1}$ we have the relation $X_{\pmb{\alpha\beta\gamma}\pmb\lambda}\equiv X_{\pmb{\alpha\gamma\beta}\pmb\lambda}$, i.e., $X_{\pmb{\alpha\beta\gamma}\pmb\lambda}$ is symmetric in $\pmb\alpha,\pmb\beta,\pmb\gamma$, considered as an element of the quotient space $\mathfrak F_{\pmb\lambda}/\mathfrak F_{\pmb\lambda-1}$ . Therefore, 
\begin{equation*}
\text{span}\Big\{\mathbb Re(X_{\pmb{\alpha\beta\gamma\lambda}}),\ \mathbb Im (X_{\pmb{\alpha\beta\gamma\lambda}})\ \Big|\  1\ \le\ {\pmb\alpha}, {\pmb\beta},\pmb\gamma\ \le \pmb\lambda-1\Big\}\subset \frac{\mathfrak F_{\pmb\lambda}}{\mathfrak F_{\pmb\lambda-1}} 
\end{equation*}
 can be generated (over the real numbers) by 
\begin{equation}\label{dim_X4_1}
2{\pmb\lambda+1 \choose{3}}
\end{equation}
elements.

If we consider the index ranges $\pmb\lambda\le \pmb\alpha \le n$, $1\le \pmb\beta,\pmb\gamma\le \pmb\lambda-1$, we have again the identity 
\begin{equation*}
X_{\pmb{\alpha\beta\gamma}\pmb\lambda}+\underbrace{{X}_{\pmb{\alpha}\pmb\lambda{\pmb{\beta\gamma}}}}_{\in \mathfrak F_{(\pmb\lambda-1)}}+X_{\pmb{\alpha\gamma}\pmb\lambda{\pmb\beta}}=0,
\end{equation*}  
and thus  $X_{\pmb{\alpha\beta\gamma}\pmb\lambda}$ is symmetric in $\pmb\beta,\pmb\gamma$  considered as an element of the quotient space $\mathfrak F_{\pmb\lambda}/\mathfrak F_{\pmb\lambda-1}$. In this case, the respective subspace
\begin{equation*}
\text{span}\Big\{\mathbb Re(X_{\pmb{\alpha\beta\gamma\lambda}}),\ \mathbb Im (X_{\pmb{\alpha\beta\gamma\lambda}})\ \Big|\  \pmb\lambda\le \pmb\alpha \le n,\ 1\ \le\ {\pmb\beta},\pmb\gamma\ \le \pmb\lambda-1\Big\}\subset \frac{\mathfrak F_{\pmb\lambda}}{\mathfrak F_{\pmb\lambda-1}} 
\end{equation*}
can by generated by
\begin{equation}\label{dim_X4_2}
2(n-\pmb\lambda+1){\pmb\lambda \choose{2}}
\end{equation}
elements.

Similarly, the subspace
\begin{equation*}
\text{span}\Big\{\mathbb Re(X_{\pmb{\alpha\beta\gamma\lambda}}), \mathbb Im (X_{\pmb{\alpha\beta\gamma\lambda}})\ \Big|\  \pmb\lambda\le \pmb\alpha,\pmb\beta \le n,\ 1\ \le\ \pmb\gamma\ \le \pmb\lambda-1\Big\}\subset \frac{\mathfrak F_{\pmb\lambda}}{\mathfrak F_{\pmb\lambda-1}} 
\end{equation*}
can by generated by
\begin{equation}\label{dim_X4_3}
2(\pmb\lambda-1){n-\pmb\lambda+2 \choose{2}}
\end{equation}
elements.

The sum of the numbers \eqref{dim_X4_1}, \eqref{dim_X4_2} and \eqref{dim_X4_3} gives an upper bound for the dimension of \eqref{space_X4}.

We proceed in a similar fashion with the linear subspace  
\begin{equation}\label{space_S2}
\text{span}\Big\{\mathbb Re\big({\mathcal S}^\ast_{\pmb{\alpha\beta'\gamma\delta'}}-{\mathcal S}^\ast_{\pmb{\alpha\beta'\delta\gamma'}}\big),
\ \mathbb Im\big({\mathcal S}^\ast_{\pmb{\alpha\beta'\gamma\delta'}}-{\mathcal S}^\ast_{\pmb{\alpha\beta'\delta\gamma'}}\big)\ \Big|\  1\ \le\ {\pmb\alpha}, {\pmb\beta}\ \le n,\  1\ \le\ {\pmb\gamma}, {\pmb\delta}\ \le \pmb\lambda\Big\} 
\end{equation}
of ${\mathfrak F_{\pmb\lambda}}/{\mathfrak F_{\pmb\lambda-1}}$. We first introduce some new real one-forms
\begin{equation}
\begin{cases}
R_{\pmb{\alpha\beta\gamma\delta}}\overset{def}{\ =\ }\frac{1}{2}\mathbb Re\Big({\mathcal S}^\ast_{\pmb{\alpha\beta'\gamma\delta'}}-{\mathcal S}^\ast_{\pmb{\alpha\beta'\delta\gamma'}}\Big),\\
T_{\pmb{\alpha\beta\gamma\delta}}\overset{def}{\ =\ }\mathbb Re\Big({\mathcal S}^\ast_{\pmb{\alpha\beta'\gamma\delta'}}+{\mathcal S}^\ast_{\pmb{\gamma\beta'\delta\alpha'}}+{\mathcal S}^\ast_{\pmb{\delta\beta'\alpha\gamma'}}\Big),\\
U_{\pmb{\alpha\beta\gamma\delta}}\overset{def}{\ =\ }\frac{1}{2}\mathbb Im\Big({\mathcal S}^\ast_{\pmb{\alpha\beta'\gamma\delta'}}-{\mathcal S}^\ast_{\pmb{\alpha\beta'\delta\gamma'}}\Big).
\end{cases}
\end{equation}
Then, the properties \eqref{properties_of_S} of ${\mathcal S}^\ast_{\pmb{\alpha\beta'\gamma\delta'}}$ imply that $R_{\pmb{\alpha\beta\gamma\delta}}$ is skew-symmetric with respect to each of the two pairs of indices $\pmb\alpha,\pmb\beta$ and $\pmb\gamma,\pmb\delta$, and satisfies the identities 
\begin{equation}\label{popertiesR}
R_{\pmb{\alpha\beta\gamma\delta}}+R_{\pmb{\gamma\alpha\beta\delta}}+R_{\pmb{\beta\gamma\alpha\delta}}=0, \qquad R_{\pmb{\alpha\beta\gamma\delta}}=R_{\pmb{\gamma\delta\alpha\beta}},
\end{equation}
i.e., it has the algebraic properties of a Riemannian curvature tensor. We have that $T_{\pmb{\alpha\beta\gamma\delta}}$ is totally symmetric, whereas $U_{\pmb{\alpha\beta\gamma\delta}}$ is symmetric in $\pmb\alpha,\pmb\beta$, skew-symmetric in $\pmb\gamma,\pmb\delta$ and satisfies
\begin{equation}
U_{\pmb{\alpha\beta\gamma\delta}}+U_{\pmb{\alpha\delta\beta\gamma}}+U_{\pmb{\alpha\gamma\delta\beta}}=0.
\end{equation}

We have also that
\begin{equation}\label{expression_S_prim2}
\begin{aligned}
\mathbb Re({\mathcal S}^\ast_{\pmb{\alpha\beta'\gamma\delta'}})&=\frac{2}{3}\Big(R_{\pmb{\alpha\beta\gamma\delta}}+R_{\pmb{\alpha\delta\gamma\beta}}\Big)+\frac{1}{3} T_{\pmb{\alpha\beta\gamma\delta}},\\
\mathbb Im({\mathcal S}^\ast_{\pmb{\alpha\beta'\gamma\delta'}})&=\frac{1}{2}\Big(U_{\pmb{\alpha\beta\gamma\delta}}+U_{\pmb{\gamma\beta\alpha\delta}}
+U_{\pmb{\alpha\delta\gamma\beta}}+U_{\pmb{\gamma\delta\alpha\beta}}\Big).
\end{aligned}
\end{equation}

Clearly, the subspace \eqref{space_S2} is generated  by $\Big\{R_{\pmb{\alpha\beta\gamma\delta}}, U_{\pmb{\alpha\beta\gamma\delta}}\ \Big|\  1\ \le\ {\pmb\alpha}, {\pmb\beta}\ \le n,\  1\ \le\ {\pmb\gamma}, {\pmb\delta}\ \le \pmb\lambda\Big\}$. We will next reduce the number of its generators by using the above symmetry properties. 
The dependence of the one-forms $U_{\pmb{\alpha\beta\gamma\delta}}$ on their indices is a subject to the exactly same algebraic relations as that of $X_{\pmb{\alpha\beta\gamma\delta}}$ . Therefore, by \eqref{dim_X4_1}, \eqref{dim_X4_2} and \eqref{dim_X4_3}, the subspace 
\begin{equation*}
\text{span}\Big\{U_{\pmb{\alpha\beta\gamma\delta}}\ \Big|\  1\ \le\ {\pmb\alpha}, {\pmb\beta}\ \le n,\  1\ \le\ {\pmb\gamma}, {\pmb\delta}\ \le \pmb\lambda\Big\}\subset \frac{\mathfrak F_{\pmb\lambda}}{\mathfrak F_{\pmb\lambda-1}} .
\end{equation*}
can be generated by 
\begin{equation}\label{dim_U4}
{\pmb\lambda+1 \choose{3}}+(n-\pmb\lambda+1){\pmb\lambda \choose{2}}+(\pmb\lambda-1){n-\pmb\lambda+2 \choose{2}}
\end{equation}
elements.

If we assume $1\le \pmb\alpha,\pmb\beta,\pmb\gamma\le \pmb\lambda -1$, then the properties \eqref{popertiesR} easily imply
that $R_{\pmb{\alpha\beta\gamma\lambda}}\equiv 0$ modulo $\mathfrak F_{\pmb\lambda-1}$. Let us consider the index ranges 
$\pmb\lambda\le \pmb\alpha \le n$, $1\le \pmb\beta,\pmb\gamma\le \pmb\lambda-1$. We have
\begin{equation*}
R_{\pmb{\alpha\beta\gamma}\pmb\lambda}+\underbrace{{R}_{\pmb{\alpha}\pmb\lambda{\pmb\beta\gamma}}}_{\in \mathfrak F_{(\pmb\lambda-1)}}+R_{\pmb{\alpha\gamma}\pmb\lambda{\pmb\beta}}=0,
\end{equation*} 
and hence, modulo  $\mathfrak F_{\pmb\lambda-1}$, $R_{\pmb{\alpha\beta\gamma}\pmb\lambda}\equiv R_{\pmb{\alpha\gamma\beta}\pmb\lambda}$. Thus 
\begin{equation*}
\text{span}\Big\{R_{\pmb{\alpha\beta\gamma\lambda}}\ \Big|\  \pmb\lambda\le \pmb\alpha \le n,\ 1\ \le\ {\pmb\beta},\pmb\gamma\ \le \pmb\lambda-1\Big\}\subset \frac{\mathfrak F_{\pmb\lambda}}{\mathfrak F_{\pmb\lambda-1}} 
\end{equation*}
can be generated by
\begin{equation}\label{dim_R4_1}
(n-\pmb\lambda+1){\pmb\lambda \choose{2}}
\end{equation}
elements. Whereas
\begin{equation*}
\text{span}\Big\{R_{\pmb{\alpha\beta\gamma\lambda}}\ \Big|\  \pmb\lambda\le \pmb\alpha,\pmb\beta \le n,\ 1\ \le\ \pmb\gamma\ \le \pmb\lambda-1\Big\}\subset \frac{\mathfrak F_{\pmb\lambda}}{\mathfrak F_{\pmb\lambda-1}} 
\end{equation*}
can by generated by
\begin{equation}\label{dim_R4_2}
(\pmb\lambda-1){n-\pmb\lambda+1 \choose{2}}
\end{equation}
elements, since $R_{\pmb{\alpha\beta\gamma}\pmb\lambda}= -R_{\pmb{\beta\alpha\gamma}\pmb\lambda}$. 
Therefore, the dimension of \eqref{space_S2} is less or equal to the sum of \eqref{dim_U4}, \eqref{dim_R4_1} and \eqref{dim_R4_2}.

Similarly, the linear span in ${\mathfrak F_{\pmb\lambda}}/{\mathfrak F_{\pmb\lambda-1}}$ of the real and imaginary parts of the one-forms 
\begin{equation*}
\Big\{{\mathcal V}^\ast_{\pmb{\alpha\beta\gamma'}}-{\mathcal V}^\ast_{\pmb{\alpha\gamma\beta'}},\ 
{\mathcal V}^\ast_{\pmb{\alpha'\beta'\gamma}}-{\mathcal V}^\ast_{\pmb{\alpha'\gamma'\beta}}\ \Big|\ 1\ \le\ {\pmb\alpha}\ \le n,\  1\ \le\ {\pmb\beta}, {\pmb\gamma}\ \le \pmb\lambda\Big\} 
\end{equation*}
 can be generated by
\begin{equation}\label{dim_V3_1}
4{\pmb\lambda \choose{2}}+4(n-\pmb\lambda+1)(\pmb\lambda-1)
\end{equation}
elements. Whereas for the linear span of the real and imaginary parts of
\begin{equation*}
\Big\{{\mathcal L}^\ast_{\pmb{\alpha\beta'}}-{\mathcal L}^\ast_{\pmb{\beta\alpha'}},\ 
{\mathcal M}^\ast_{\pmb{\alpha\beta'}}-{\mathcal M}^\ast_{\pmb{\beta\alpha'}}\ \Big|\ 1\ \le\ {\pmb\alpha}\ \le n,\  1\ \le\ {\pmb\beta}, {\pmb\gamma}\ \le \pmb\lambda\Big\} 
\end{equation*}
in ${\mathfrak F_{\pmb\lambda}}/{\mathfrak F_{\pmb\lambda-1}}$, we need only
\begin{equation}\label{dim_LM2_1}
3(\pmb\lambda-1)
\end{equation}
generators (notice that by \eqref{properties_of_L}, the imaginary part of ${\mathcal L}^\ast_{\pmb{\alpha\beta'}}-{\mathcal L}^\ast_{\pmb{\beta\alpha'}}$ vanishes).

The sum of \eqref{dim_X4_1}, \eqref{dim_X4_2}, \eqref{dim_X4_3}, \eqref{dim_U4}, \eqref{dim_R4_1}, \eqref{dim_R4_2}, \eqref{dim_V3_1} and \eqref{dim_LM2_1}
gives an upper bound for the dimension of ${\mathfrak F_{\pmb\lambda}}/{\mathfrak F_{\pmb\lambda-1}}$, i.e., we have
\begin{equation*}
\begin{gathered}
\dim\Big(\frac{\mathfrak F_{\pmb\lambda}}{\mathfrak F_{\pmb\lambda-1}}\Big) \ \pmb\le\ 3{\pmb\lambda+1 \choose{3}}\ +\ 3(n-\pmb\lambda+1){\pmb\lambda \choose{2}}\\
+\ 3(\pmb\lambda-1){n-\pmb\lambda+2 \choose{2}}
\ +\ (n-\pmb\lambda+1){\pmb\lambda \choose{2}}\ +\ (\pmb\lambda-1){n-\pmb\lambda+1 \choose{2}}\\
+\ 4{\pmb\lambda \choose{2}}\ +\ 4(n-\pmb\lambda+1)(\pmb\lambda-1)
\ +\ 3(\lambda-1)\\ 
=\frac{1}{2}(\pmb\lambda-1)(\pmb\lambda-2n-4)(\pmb\lambda-2n-5).
\end{gathered}
\end{equation*}

The vector space $\mathfrak F_{n}$ is freely generated by the real and imaginary parts of all the one-forms (modulo their respective symmetries)  
\begin{equation*}
\begin{gathered}
X_{\pmb{\alpha\beta\gamma\delta}},\qquad R_{\pmb{\alpha\beta\gamma\delta}},\qquad U_{\pmb{\alpha\beta\gamma\delta}}\\
{\mathcal V}^\ast_{\pmb{\alpha\beta\gamma'}}-{\mathcal V}^\ast_{\pmb{\alpha\gamma\beta'}},\qquad
{\mathcal V}^\ast_{\pmb{\alpha'\beta'\gamma}}-{\mathcal V}^\ast_{\pmb{\alpha'\gamma'\beta}},\qquad
{\mathcal L}^\ast_{\pmb{\alpha\beta'}}-{\mathcal L}^\ast_{\pmb{\beta\alpha'}},\qquad
{\mathcal M}^\ast_{\pmb{\alpha\beta'}}-{\mathcal M}^\ast_{\pmb{\beta\alpha'}}
\end{gathered}
\end{equation*}
and hence we can easily compute its dimension,
\begin{multline*}
\dim(\mathfrak F_n)= \underbrace{\left[{\frac{n(n-1)}{2}+1\choose 2}-{n\choose 4}\right]}_{\text{this is for $\{R_{\pmb{\alpha\beta\gamma\delta}}\}$}}\ 
+\underbrace{3{\frac{n(n+1)}{2}\choose 2}}_{\text{$\{X_{\pmb{\alpha\beta\gamma\delta}}$, $U_{\pmb{\alpha\beta\gamma\delta}}\}$}}\ 
+\ \underbrace{4\left[n{n+1\choose 2}-{n+2\choose 3}\right]}_{\text{$\{{\mathcal V}^\ast_{\pmb{\alpha\beta\gamma'}}-{\mathcal V}^\ast_{\pmb{\alpha\gamma\beta'}}$, ${\mathcal V}^\ast_{\pmb{\alpha'\beta'\gamma}}-{\mathcal V}^\ast_{\pmb{\alpha'\gamma'\beta}}\}$}}\\ 
+\ \underbrace{3{n\choose 2}}_{\text{$\{{\mathcal L}^\ast_{\pmb{\alpha\beta'}}-{\mathcal L}^\ast_{\pmb{\beta\alpha'}}$, ${\mathcal M}^\ast_{\pmb{\alpha\beta'}}-{\mathcal M}^\ast_{\pmb{\beta\alpha'}}\}$}}\ 
=\ \frac{1}{24}n(n-1)(11n^2+61n+86).
\end{multline*}

By construction $0=\mathfrak F_1\subset \mathfrak F_2\subset\dots\subset \mathfrak F_n$ and thus
\begin{equation*}
\mathfrak F_n\cong \Big(\frac{\mathfrak F_{2}}{\mathfrak F_{1}}\Big) \oplus \Big(\frac{\mathfrak F_{3}}{\mathfrak F_{2}}\Big)\oplus\dots\oplus\Big(\frac{\mathfrak F_{n}}{\mathfrak F_{n-1}}\Big).  
\end{equation*}
Therefore we can calculate (by using, for example, some of the computer algebra systems) that
\begin{multline*}
\dim(\mathfrak F_n)=\sum_{\pmb\lambda=2}^n\dim\Big(\frac{\mathfrak F_{\pmb\lambda}}{\mathfrak F_{\pmb\lambda-1}}\Big)\ \pmb\le\  
\sum_{\pmb\lambda=2}^n\Big(\frac{1}{2}(\pmb\lambda-1)(\pmb\lambda-2n-4)(\pmb\lambda-2n-5)\Big)\\
=\ \frac{1}{24}n(n-1)(11n^2+61n+86),
\end{multline*}
which implies that the above inequality must actually be an equality, i.e., we have shown
\begin{equation}
v_{\pmb\lambda}=\dim\Big(\frac{\mathfrak F_{\pmb\lambda}}{\mathfrak F_{\pmb\lambda-1}}\Big)=\frac{1}{2}(\pmb\lambda-1)(\pmb\lambda-2n-4)(\pmb\lambda-2n-5),\qquad 2\le \pmb\lambda\le n.
\end{equation}

\subsection{The characters $v_{(n+1)},\dots,v_{2n}$}
Notice that, modulo $\mathfrak F_n$, we have (cf. \eqref{expression_S_prim}, \eqref{expression_S_prim2})
\begin{equation*}
{\mathcal S}^\ast_{\pmb{\alpha\beta\gamma\delta'}}\equiv Y_{\pmb{\alpha\beta\gamma\delta}},\qquad 
{\mathcal S}^\ast_{\pmb{\alpha\beta'\gamma\delta'}}\equiv \frac{1}{3}T_{\pmb{\alpha\beta\gamma\delta}}
\end{equation*}
and that each of the arrays
\begin{equation*}
Y_{\pmb{\alpha\beta\gamma\delta}},\quad T_{\pmb{\alpha\beta\gamma\delta}},\quad{\mathcal V}^\ast_{\pmb{\alpha\beta\gamma'}}, \quad {\mathcal V}^\ast_{\pmb{\alpha\beta'\gamma'}},\quad {\mathcal L}^\ast_{\pmb{\alpha\beta'}},\quad
{\mathcal M}^\ast_{\pmb{\alpha\beta'}}
\end{equation*}
depends totally symmetrically on its indices.

Let us fix $\pmb\lambda$ to be an integer number between $1$ and $n$. By definition, the quotient space $\mathfrak F_{(n+\pmb\lambda)}/\mathfrak F_n$ is generated by the real and imaginary parts of the one-forms:
\begin{multline*}
A_{\pmb{\alpha\beta\gamma\delta}}\overset{def}{\ =\ }i{\mathcal S}^\ast_{\pmb{\alpha\beta\gamma\delta'}}+i{\mathcal S}^\ast_{\pmb{\alpha\beta\delta\gamma'}}+{\mathcal S}^\ast_{\pmb{\alpha\beta\gamma}[{\pmb{\delta}}+1]}+{\mathcal S}^\ast_{\pmb{\alpha\beta\gamma'}[{\pmb{\delta}+1}]\pmb'}\ \equiv\ 
2i Y_{\pmb{\alpha\beta\gamma\delta}}+{\mathcal S}^\ast_{\pmb{\alpha\beta\gamma}[{\pmb{\delta}}+1]}+\frac{1}{3}T_{\pmb{\alpha\beta\gamma}[{\pmb{\delta}}+1]},\\
1\le \pmb\alpha,\pmb\beta,\pmb\gamma\le n,\qquad 1\le \pmb\delta\le \pmb\lambda;
\end{multline*}
\begin{multline*}
B_{\pmb{\alpha\beta\gamma\delta}}\overset{def}{\ =\ }-i\Big({\mathcal S}^\ast_{\pmb{\alpha\beta\gamma\delta'}}-{\mathcal S}^\ast_{\pmb{\alpha\beta\delta\gamma'}}+i{\mathcal S}^\ast_{\pmb{\alpha\beta\gamma'}[{\pmb{\delta}+1}]\pmb'}-i{\mathcal S}^\ast_{\pmb{\alpha\beta\delta'}[{\pmb{\gamma}+1}]\pmb'}-i{\mathcal S}^\ast_{\pmb{\alpha\beta\gamma}[{\pmb{\delta}}+1]}+i{\mathcal S}^\ast_{\pmb{\alpha\beta\delta}[{\pmb{\gamma}}+1]}
+{\mathcal S}^\ast_{\pmb{\alpha\beta}[{\pmb{\gamma}}+1][{\pmb{\delta}+1}]\pmb'}-{\mathcal S}^\ast_{\pmb{\alpha\beta}[{\pmb{\delta}}+1][{\pmb{\gamma}+1}]\pmb'}\Big)\\ \equiv\ 
\frac{1}{3}T_{\pmb{\alpha\beta\gamma}[{\pmb\delta}+1]}-\frac{1}{3}T_{\pmb{\alpha\beta\delta}[{\pmb\gamma}+1]}
-{\mathcal S}^\ast_{\pmb{\alpha\beta\gamma}[{\pmb\delta}+1]}+{\mathcal S}^\ast_{\pmb{\alpha\beta\delta}[{\pmb\gamma}+1]},\qquad
1\le \pmb\alpha,\pmb\beta\le n,\qquad 1\le \pmb\gamma,\pmb\delta\le \pmb\lambda;
\end{multline*}
\begin{multline*}
C_{\pmb{\alpha\beta\gamma\delta}}\overset{def}{\ =\ }-\frac{i}{2}\Big(i{\mathcal S}^\ast_{\pmb{\alpha\beta'\gamma\delta'}}+i{\mathcal S}^\ast_{\pmb{\alpha\beta'\delta\gamma'}}+{\mathcal S}^\ast_{\pmb{\alpha\beta'\gamma}[{\pmb{\delta}}+1]}-\overline{{\mathcal S}^\ast_{\pmb{\alpha'\beta\gamma}[{\pmb{\delta}+1}]}}\Big)\ \equiv\ 
\frac{1}{3}T_{\pmb{\alpha\beta\gamma\delta}}+\mathbb Im \big(Y_{\pmb{\alpha\beta\gamma}[{\pmb\delta}+1]}\big),\\
1\le \pmb\alpha,\pmb\beta,\pmb\gamma\le n,\qquad 1\le \pmb\delta\le \pmb\lambda;
\end{multline*}
\begin{multline*}
D_{\pmb{\alpha\beta\gamma\delta}}\overset{def}{\ =\ }-i\Big({\mathcal S}^\ast_{\pmb{\alpha\beta'\gamma\delta'}}-{\mathcal S}^\ast_{\pmb{\alpha\beta'\delta\gamma'}}
-i\overline{{\mathcal S}^\ast_{\pmb{\alpha'\beta\gamma}[{\pmb{\delta}+1}]}}+i\overline{{\mathcal S}^\ast_{\pmb{\alpha'\beta\delta}[{\pmb{\gamma}+1}]}}-i{\mathcal S}^\ast_{\pmb{\alpha\beta'\gamma}[{\pmb{\delta}}+1]}+i{\mathcal S}^\ast_{\pmb{\alpha\beta'\delta}[{\pmb{\gamma}}+1]}
+{\mathcal S}^\ast_{\pmb{\alpha\beta'}[{\pmb{\gamma}}+1][{\pmb{\delta}+1}]\pmb'}-{\mathcal S}^\ast_{\pmb{\alpha\beta'}[{\pmb{\delta}}+1][{\pmb{\gamma}+1}]\pmb'}\Big)\\ \equiv\ 
-\mathbb Re \big(Y_{\pmb{\alpha\beta\gamma}[{\pmb\delta}+1]}\big)+\mathbb Re \big(Y_{\pmb{\alpha\beta\delta}[{\pmb\gamma}+1]}\big),\qquad
1\le \pmb\alpha,\pmb\beta\le n,\qquad 1\le \pmb\gamma,\pmb\delta\le \pmb\lambda;
\end{multline*}
\begin{multline*}
A_{\pmb{\alpha\beta\gamma}}\overset{def}{\ =\ }i\Big({\mathcal V}^\ast_{\pmb{\alpha\beta\gamma'}}+{\mathcal V}^\ast_{\pmb{\alpha\beta'\gamma}}-i{\mathcal V}^\ast_{\pmb{\alpha\beta}[{\pmb{\gamma}}+1]}
 -i{\mathcal V}^\ast_{\pmb{\alpha\beta'}[{\pmb{\gamma}+1}]\pmb'}\Big)\ \equiv\ 
2i {\mathcal V}^\ast_{\pmb{\alpha\beta\gamma'}}+{\mathcal V}^\ast_{\pmb{\alpha\beta}[{\pmb{\gamma}}+1]}+{\mathcal V}^\ast_{\pmb{\alpha\beta'}[{\pmb{\gamma}}+1]\pmb'},\\
1\le \pmb\alpha,\pmb\beta,\le n,\qquad 1\le \pmb\gamma\le \pmb\lambda;
\end{multline*}
\begin{multline*}
B_{\pmb{\alpha\beta\gamma}}\overset{def}{\ =\ }-i{\mathcal V}^\ast_{\pmb{\alpha\beta\gamma'}}+i{\mathcal V}^\ast_{\pmb{\alpha\beta'\gamma}}
+{\mathcal V}^\ast_{\pmb{\alpha\beta'}[{\pmb{\gamma}+1}]\pmb'}-{\mathcal V}^\ast_{\pmb{\alpha\gamma'}[{\pmb{\beta}+1}]\pmb'}
-{\mathcal V}^\ast_{\pmb{\alpha\beta}[{\pmb{\gamma}}+1]}+{\mathcal V}^\ast_{\pmb{\alpha\gamma}[{\pmb{\beta}}+1]}
-i{\mathcal V}^\ast_{\pmb{\alpha}[{\pmb{\beta}}+1][{\pmb{\gamma}+1}]\pmb'}+i{\mathcal V}^\ast_{\pmb{\alpha}[{\pmb{\gamma}}+1][{\pmb{\beta}+1}]\pmb'}\\ \equiv\ 
{\mathcal V}^\ast_{\pmb{\alpha\beta'}[{\pmb\gamma}+1]\pmb'}-{\mathcal V}^\ast_{\pmb{\alpha\gamma'}[{\pmb\beta}+1]\pmb'}
-{\mathcal V}^\ast_{\pmb{\alpha\beta}[{\pmb\gamma}+1]}+{\mathcal V}^\ast_{\pmb{\alpha\gamma}[{\pmb\beta}+1]},\qquad
1\le \pmb\alpha\le n,\qquad 1\le \pmb\beta,\pmb\gamma\le \pmb\lambda;
\end{multline*}
\begin{multline*}
C_{\pmb{\alpha\beta\gamma}}\overset{def}{\ =\ }{\mathcal V}^\ast_{\pmb{\alpha'\beta'\gamma}}+{\mathcal V}^\ast_{\pmb{\alpha'\gamma'\beta}}-i{\mathcal V}^\ast_{\pmb{\alpha'\beta'}[{\pmb{\gamma}}+1]\pmb'}
 -i{\mathcal V}^\ast_{\pmb{\alpha'\beta}[{\pmb{\gamma}+1}]}\ \equiv\ 
2{\mathcal V}^\ast_{\pmb{\alpha\beta'\gamma'}}-i{\mathcal V}^\ast_{\pmb{\alpha'\beta'}[{\pmb{\gamma}+1}]\pmb'}-i{\mathcal V}^\ast_{\pmb{\alpha'\beta}[{\pmb{\gamma}+1}]},\\
1\le \pmb\alpha,\pmb\beta,\le n,\qquad 1\le \pmb\gamma\le \pmb\lambda;
\end{multline*}
\begin{multline*}
D_{\pmb{\alpha\beta\gamma}}\overset{def}{\ =\ }i{\mathcal V}^\ast_{\pmb{\alpha'\beta'\gamma}}-i{\mathcal V}^\ast_{\pmb{\alpha'\gamma'\beta}}
+{\mathcal V}^\ast_{\pmb{\alpha'\beta'}[{\pmb{\gamma}+1}]\pmb'}-{\mathcal V}^\ast_{\pmb{\alpha'\gamma'}[{\pmb{\beta}+1}]\pmb'}
-{\mathcal V}^\ast_{\pmb{\alpha'\beta}[{\pmb{\gamma}}+1]}+{\mathcal V}^\ast_{\pmb{\alpha'\gamma}[{\pmb{\beta}}+1]}
+i{\mathcal V}^\ast_{\pmb{\alpha'}[{\pmb{\beta}}+1]\pmb'[{\pmb{\gamma}+1}]}-i{\mathcal V}^\ast_{\pmb{\alpha'}[{\pmb{\gamma}}+1]\pmb'[{\pmb{\beta}+1}]}\\ \equiv\ 
{\mathcal V}^\ast_{\pmb{\alpha'\beta'}[{\pmb{\gamma}+1}]\pmb'}-{\mathcal V}^\ast_{\pmb{\alpha'\gamma'}[{\pmb{\beta}+1}]\pmb'}
-{\mathcal V}^\ast_{\pmb{\alpha'\beta}[{\pmb{\gamma}}+1]}+{\mathcal V}^\ast_{\pmb{\alpha'\gamma}[{\pmb{\beta}}+1]},\qquad
1\le \pmb\alpha\le n,\qquad 1\le \pmb\beta,\pmb\gamma\le \pmb\lambda;
\end{multline*}
\begin{multline*}
A_{\pmb{\alpha\beta}}\overset{def}{\ =\ }\frac{1}{2}\Big(i{\mathcal L}^\ast_{\pmb{\alpha\beta'}}+i{\mathcal L}^\ast_{\pmb{\alpha'\beta}}+\overline{{\mathcal L}^\ast_{\pmb{\alpha}[{\pmb{\beta}}+1]}}
 +{\mathcal L}^\ast_{\pmb{\alpha}[{\pmb{\beta}+1}]}\Big)\ \equiv\ 
i {\mathcal L}^\ast_{\pmb{\alpha\beta'}}+\mathbb Re\big({\mathcal L}^\ast_{\pmb{\alpha}[{\pmb{\beta}}+1]}\big),\qquad
1\le \pmb\alpha,\le n,\quad 1\le \pmb\beta\le \pmb\lambda;
\end{multline*}
\begin{multline*}
B_{\pmb{\alpha\beta}}\overset{def}{\ =\ }\frac{1}{2}\Big({\mathcal L}^\ast_{\pmb{\alpha\beta'}}-{\mathcal L}^\ast_{\pmb{\alpha'\beta}}
-i{\mathcal L}^\ast_{\pmb{\alpha}[{\pmb{\beta}+1}]}+i{\mathcal L}^\ast_{\pmb{\beta}[{\pmb{\alpha}+1}]}
+i\overline{{\mathcal L}^\ast_{\pmb{\alpha}[{\pmb{\beta}}+1]}}-i\overline{{\mathcal L}^\ast_{\pmb{\beta}[{\pmb{\alpha}}+1]}}
-{\mathcal L}^\ast_{[{\pmb{\alpha}}+1]\pmb'[{\pmb{\beta}+1}]}+{\mathcal L}^\ast_{[{\pmb{\beta}}+1]\pmb'[{\pmb{\alpha}+1}]}\Big)\\ \equiv\ 
\mathbb Im\big({\mathcal L}^\ast_{\pmb{\alpha}[{\pmb\beta}+1]}\big)-\mathbb Im\big({\mathcal L}^\ast_{\pmb{\beta}[{\pmb\alpha}+1]}\big),\qquad
1\le \pmb\alpha, \pmb\beta\le \pmb\lambda;
\end{multline*}
\begin{multline*}
C_{\pmb{\alpha\beta}}\overset{def}{\ =\ }{\mathcal M}^\ast_{\pmb{\alpha'\beta}}+{\mathcal M}^\ast_{\pmb{\alpha\beta'}}
-i{\mathcal M}^\ast_{\pmb{\alpha'}[{\pmb{\beta}}+1]\pmb'}
 -i{\mathcal M}^\ast_{\pmb{\alpha}[{\pmb{\beta}+1}]}\ \equiv\ 
2{\mathcal M}^\ast_{\pmb{\alpha\beta'}}-i{\mathcal M}^\ast_{\pmb{\alpha'}[{\pmb{\beta}}+1]\pmb'}
 -i{\mathcal M}^\ast_{\pmb{\alpha}[{\pmb{\beta}+1}]},\\
1\le \pmb\alpha,\le n,\quad 1\le \pmb\beta\le \pmb\lambda;
\end{multline*}
\begin{multline*}
D_{\pmb{\alpha\beta}}\overset{def}{\ =\ }-i{\mathcal M}^\ast_{\pmb{\alpha'\beta}}+i{\mathcal M}^\ast_{\pmb{\alpha\beta'}}
+{\mathcal M}^\ast_{\pmb{\alpha}[{\pmb{\beta}+1}]}-{\mathcal M}^\ast_{\pmb{\beta}[{\pmb{\alpha}+1}]}
-{\mathcal M}^\ast_{\pmb{\alpha'}[{\pmb{\beta}}+1]'}+{\mathcal M}^\ast_{\pmb{\beta'}[{\pmb{\alpha}}+1]'}
-i{\mathcal M}^\ast_{[{\pmb{\alpha}}+1]\pmb'[{\pmb{\beta}+1}]}+i{\mathcal M}^\ast_{[{\pmb{\beta}}+1]\pmb'[{\pmb{\alpha}+1}]}\\ \equiv\ 
{\mathcal M}^\ast_{\pmb{\alpha}[{\pmb{\beta}+1}]}-{\mathcal M}^\ast_{\pmb{\beta}[{\pmb{\alpha}+1}]}
-{\mathcal M}^\ast_{\pmb{\alpha'}[{\pmb{\beta}}+1]'}+{\mathcal M}^\ast_{\pmb{\beta'}[{\pmb{\alpha}}+1]'},\qquad
1\le \pmb\alpha, \pmb\beta\le \pmb\lambda.
\end{multline*}

\begin{lemma}
The linear space $\mathfrak F_{(n+\pmb\lambda)}/\mathfrak F_{{(n+\pmb\lambda-1)}}$ is generated (over the real numbers) by the real and imaginary parts of the one-forms
\begin{equation}\label{As_and_Cs}
\Big\{A_{\pmb{\alpha\beta\gamma\lambda}},\ A_{\pmb{\alpha\beta\lambda}},\ A_{\pmb{\alpha\lambda}}, \ C_{\pmb{\alpha\beta\gamma\lambda}},\ 
C_{\pmb{\alpha\beta\lambda}},\ C_{\pmb{\alpha\lambda}}\ \Big|\ 1\le \pmb\alpha,\pmb\beta,\pmb\gamma\le n\Big\}\cup\Big\{B_{\pmb{\alpha\lambda}},\ D_{\pmb{\alpha\lambda}}\Big|\ 1\le \alpha\le\lambda\Big\}.
\end{equation}
\end{lemma}

Furthermore, modulo $\mathfrak F_{{(n+\pmb\lambda)}}$, for any $1\le \pmb\alpha,\pmb\beta\le n,\  1\le \pmb\gamma,\pmb\delta\le \pmb\lambda$, we have the identities
\begin{equation}\label{Properties_mod_U}
\begin{gathered}
Y_{\pmb{\alpha\beta}[{\pmb\gamma}+1]{\pmb\delta}}\equiv Y_{\pmb{\alpha\beta\gamma}[{\pmb\delta}+1]},\qquad {\mathcal S}^\ast_{\pmb{\alpha\beta}[{\pmb\gamma}+1]{\pmb\delta}}\equiv {\mathcal S}^\ast_{\pmb{\alpha\beta\gamma}[{\pmb\delta}+1]},\qquad 
T_{\pmb{\alpha\beta}[{\pmb\gamma}+1]{\pmb\delta}}\equiv T_{\pmb{\alpha\beta\gamma}[{\pmb\delta}+1]},\\
{\mathcal V}^\ast_{\pmb{\beta}[{\pmb\gamma}+1]{\pmb\delta}}\equiv {\mathcal V}^\ast_{\pmb{\beta\gamma}[{\pmb\delta}+1]},\qquad
{\mathcal V}^\ast_{\pmb{\beta'}[{\pmb\gamma}+1]{\pmb\delta}}\equiv {\mathcal V}^\ast_{\pmb{\beta'\gamma}[{\pmb\delta}+1]},\qquad
{\mathcal V}^\ast_{\pmb{\beta'}[{\pmb\gamma}+1]\pmb'{\pmb\delta}}\equiv {\mathcal V}^\ast_{\pmb{\beta'\gamma'}[{\pmb\delta}+1]}\\
{\mathcal M}^\ast_{[{\pmb\gamma}+1]{\pmb\delta}}\equiv {\mathcal M}^\ast_{\pmb{\gamma}[{\pmb\delta}+1]},\qquad
{\mathcal M}^\ast_{[{\pmb\gamma}+1]\pmb'{\pmb\delta}}\equiv {\mathcal M}^\ast_{\pmb{\gamma'}[{\pmb\delta}+1]},\qquad
{\mathcal M}^\ast_{[{\pmb\gamma}+1]\pmb'{\pmb{\delta'}}}\equiv {\mathcal M}^\ast_{\pmb{\gamma'}[{\pmb\delta}+1]\pmb'},\\
{\mathcal L}^\ast_{[{\pmb\gamma}+1]{\pmb{\delta'}}}\equiv {\mathcal L}^\ast_{\pmb{\gamma}[{\pmb\delta}+1]\pmb'},\qquad
{\mathcal L}^\ast_{[{\pmb\gamma}+1]\pmb'{\pmb\delta}}\equiv {\mathcal L}^\ast_{\pmb{\gamma'}[{\pmb\delta}+1]}.
\end{gathered}
\end{equation}
\begin{proof}
Let us denote by $\mathfrak U_{\pmb\lambda}$ the sum of $\mathfrak F_{(n+\pmb\lambda-1)}$ and the linear span of the real and imaginary parts of the one-forms \eqref{As_and_Cs}, and consider the index ranges $1\le \pmb\alpha,\pmb\beta\le n,\  1\le \pmb\gamma,\pmb\delta\le \pmb\lambda$.  Since $A_{\pmb{\alpha\beta\gamma\delta}}\equiv C_{\pmb{\alpha\beta\gamma\delta}}\equiv 0$ modulo ${\mathfrak U_{\pmb\lambda}}$, we have 
\begin{equation}
\begin{cases}
-2i Y_{\pmb{\alpha\beta\gamma\delta}}\equiv{\mathcal S}^\ast_{\pmb{\alpha\beta\gamma}[{\pmb{\delta}}+1]}+\frac{1}{3}T_{\pmb{\alpha\beta\gamma}[{\pmb{\delta}}+1]}\\
-\frac{1}{3}T_{\pmb{\alpha\beta\gamma\delta}}\equiv \mathbb Im \big(Y_{\pmb{\alpha\beta\gamma}[{\pmb\delta}+1]}\big)
\end{cases} 
\mod{\mathfrak U_{\pmb\lambda}.
\end{equation}
Using this, we obtain, modulo ${\mathfrak U_{\pmb\lambda}}$,
\begin{equation*}
-2iY_{\pmb{\alpha\beta}[{\pmb\gamma}+1]{\pmb\delta}}\equiv{\mathcal S}^\ast_{\pmb{\alpha\beta}[{\pmb\gamma}+1][{\pmb{\delta}}+1]}+\frac{1}{3}T_{\pmb{\alpha\beta}[{\pmb\gamma}+1][{\pmb{\delta}}+1]}\equiv {\mathcal S}^\ast_{\pmb{\alpha\beta}[{\pmb\delta}+1][{\pmb{\gamma}}+1]}+\frac{1}{3}T_{\pmb{\alpha\beta}[{\pmb\delta}+1][{\pmb{\gamma}}+1]}\equiv-2iY_{\pmb{\alpha\beta}[{\pmb\delta}+1]{\pmb\gamma}}
\end{equation*}
which proves the first equation in \eqref{Properties_mod_U}. Similarly,
\begin{equation*}
{\mathcal S}^\ast_{\pmb{\alpha\beta\delta}[{\pmb\gamma}+1]}+\frac{1}{3}T_{\pmb{\alpha\beta\delta}[{\pmb\gamma}+1]}\equiv -2iY_{\pmb{\alpha\beta\delta\gamma}}\equiv-2iY_{\pmb{\alpha\beta\gamma\delta}}\equiv {\mathcal S}^\ast_{\pmb{\alpha\beta\gamma}[{\pmb\delta}+1]}+\frac{1}{3}T_{\pmb{\alpha\beta\gamma}[{\pmb\delta}+1]}
\end{equation*}
and also
\begin{equation}\label{S_mod_2n}
-\frac{1}{3}T_{\pmb{\alpha\beta}[{\pmb\gamma}+1]{\pmb\delta}}\equiv\mathbb Im\big(Y^\ast_{\pmb{\alpha\beta}[{\pmb\gamma}+1][{\pmb{\delta}}+1]}\big)\equiv \mathbb Im\big(Y^\ast_{\pmb{\alpha\beta}[{\pmb\delta}+1][{\pmb{\gamma}}+1]}\big)}\equiv-\frac{1}{2}T_{\pmb{\alpha\beta}[{\pmb\delta}+1]{\pmb\gamma}},
\end{equation}
which yields the second and the third equations in the first line of \eqref{Properties_mod_U}. The proof of the rest of \eqref{Properties_mod_U} is completely analogous.  
Now, applying \eqref{Properties_mod_U}, we have that, modulo $\mathfrak U_{\pmb\lambda}$,
\begin{equation*}
B_{\pmb{\alpha\beta\gamma\lambda}} \equiv\ 
\frac{1}{3}T_{\pmb{\alpha\beta\gamma}[{\pmb\lambda}+1]}-\frac{1}{3}T_{\pmb{\alpha\beta\lambda}[{\pmb\gamma}+1]}
-{\mathcal S}^\ast_{\pmb{\alpha\beta\gamma}[{\pmb\lambda}+1]}+{\mathcal S}^\ast_{\pmb{\alpha\beta\lambda}[{\pmb\gamma}+1]}\equiv 0,
\end{equation*}
and similarly 
\begin{equation*}
D_{\pmb{\alpha\beta\gamma\lambda}}\equiv B_{\pmb{\beta\gamma\lambda}}\equiv D_{\pmb{\beta\gamma\lambda}}\equiv0.
\end{equation*}
\end{proof}

It is easy to observe that, by a repeated application of the identities in the first line of \eqref{Properties_mod_U},  each $A_{\pmb{\alpha\beta\gamma\lambda}}$ can be made equivalent, modulo $\mathfrak F_{n+\pmb\lambda-1}$, to one of the elements in the following two sets:
\begin{equation}\label{sets_for_A4}
\Big\{A_{\pmb{\alpha\beta\gamma\lambda}}\ \Big|\ \pmb\alpha,\pmb\beta,\pmb\gamma\in\{1,\pmb\lambda,\pmb\lambda+1,\dots, n\}\Big\};\ \ 
\Big\{A_{\pmb{\alpha\beta\gamma\lambda}}\ \Big|\ \pmb\alpha,\pmb\beta\in\{1,\pmb\lambda,\pmb\lambda+1,\dots, n\},\ 2\le\pmb\gamma\le \pmb\lambda-1 \Big\}.
\end{equation}

Let us consider the one-forms $C_{\pmb{\alpha\beta\gamma\lambda}}$ modulo $\mathfrak F_{n+\pmb\lambda-1}\oplus \text{span}\Big\{\mathbb Re\big(A_{\pmb{\alpha\beta\gamma\lambda}}\big),\ \mathbb Im\big(A_{\pmb{\alpha\beta\gamma\lambda}}\big),\Big\}$. If we suppose $1\le\pmb\gamma\le \pmb\lambda-1$, then
\begin{equation*}
C_{\pmb{\alpha\beta\gamma\lambda}}\equiv\frac{1}{3}T_{\pmb{\alpha\beta\gamma\lambda}}+\mathbb Im \big(Y_{\pmb{\alpha\beta\gamma}[{\pmb\lambda}+1]}\big)
\equiv\frac{1}{3}T_{\pmb{\alpha\beta}[{\pmb\gamma}+1][{\pmb\lambda-1}]}+\mathbb Im \big(Y_{\pmb{\alpha\beta}[{\pmb\gamma+1}]{\pmb\lambda}}\big)\equiv C_{\pmb{\alpha\beta}[{\pmb\gamma}+1]\lambda}\equiv 0.
\end{equation*}
Therefore, each $C_{\pmb{\alpha\beta\gamma\lambda}}$ is equivalent to one of the forms in the set 
\begin{equation}\label{set_for_C4}
\Big\{C_{\pmb{\alpha\beta\gamma\lambda}}\ \Big|\ \lambda\le\pmb\alpha,\pmb\beta,\pmb\gamma\le n\}\Big\}.
\end{equation} 
Thus, by \eqref{sets_for_A4} and \eqref{set_for_C4}, the linear subspace 
\begin{equation*}
\text{span}\Big\{\mathbb Re\big(A_{\pmb{\alpha\beta\gamma\lambda}}\big),\ \mathbb Im\big(A_{\pmb{\alpha\beta\gamma\lambda}}\big), C_{\pmb{\alpha\beta\gamma\lambda}}\ 
\Big|1\le\pmb\alpha,\pmb\beta,\pmb\gamma\le n\Big\}\subset \frac{\mathfrak F_{n+\pmb\lambda}}{\mathfrak F_{n+\pmb\lambda-1}}
\end{equation*}
can be generated by
\begin{equation}\label{dim_AC4}
2{n-\pmb\lambda+4\choose 3}+2{\pmb\lambda-2\choose 1}{n-\pmb\lambda+3\choose 2}+ {n-\pmb\lambda+3\choose 3}
\end{equation}
elements.

Similarly, by a repeated application of the identities in the second line of \eqref{Properties_mod_U},  we obtain that each of the one-forms $A_{\pmb{\alpha\beta\lambda}}$,  $C_{\pmb{\alpha\beta\lambda}}$ can be transformed, equivalently modulo $\mathfrak F_{n+\pmb\lambda-1}$, to one of the elements in the following two sets:
\begin{equation*}
\Big\{A_{\pmb{\alpha\beta\lambda}},\ C_{\pmb{\alpha\beta\lambda}}\ \Big|\ \pmb\alpha,\pmb\beta\in\{1,\pmb\lambda,\pmb\lambda+1,\dots, n\}\Big\};\ \ 
\Big\{A_{\pmb{\alpha\beta\lambda}},\ C_{\pmb{\alpha\beta\lambda}}\ \Big|\ \pmb\alpha\in\{1,\pmb\lambda,\pmb\lambda+1,\dots, n\},\ 2\le\pmb\beta\le \pmb\lambda-1 \Big\}.
\end{equation*} 
Therefore,
\begin{equation*}
\text{span}\Big\{\mathbb Re\big(A_{\pmb{\alpha\beta\lambda}}\big),\ \mathbb Im\big(A_{\pmb{\alpha\beta\lambda}}\big), 
\mathbb Re\big(C_{\pmb{\alpha\beta\lambda}}\big),\ \mathbb Im\big(C_{\pmb{\alpha\beta\lambda}}\big)\ 
\Big|1\le\pmb\alpha,\pmb\beta\le n\Big\}\subset \frac{\mathfrak F_{n+\pmb\lambda}}{\mathfrak F_{n+\pmb\lambda-1}}
\end{equation*}
can be generated by
\begin{equation}\label{dim_AC3}
4{n-\pmb\lambda+3\choose 2}+4{\pmb\lambda-2\choose 1}{n-\pmb\lambda+2\choose 1}
\end{equation}
elements.

Clearly, 
\begin{equation*}
\text{span}\Big\{A_{\pmb{\alpha\lambda}} 
\Big|1\le\pmb\alpha\le n\Big\}\oplus\text{span}\Big\{B_{\pmb{\alpha\lambda}} 
\Big|1\le\pmb\alpha\le \lambda-1\Big\} \subset \frac{\mathfrak F_{n+\pmb\lambda}}{\mathfrak F_{n+\pmb\lambda-1}}
\end{equation*}
can be generated by 
\begin{equation}\label{dim_AB2}
n+\pmb\lambda-1
\end{equation}
elements, and similarly 
\begin{equation*}
\text{span}\Big\{\mathbb Re\big(C_{\pmb{\alpha\lambda}}\big), \mathbb Im\big(C_{\pmb{\alpha\lambda}}\big)
\Big|1\le\pmb\alpha\le n\Big\}\oplus\text{span}\Big\{D_{\pmb{\alpha\lambda}} 
\Big|1\le\pmb\alpha\le \lambda-1\Big\} \subset \frac{\mathfrak F_{n+\pmb\lambda}}{\mathfrak F_{n+\pmb\lambda-1}}
\end{equation*}
generates by 
\begin{equation}\label{dim_CD2}
2(n+\pmb\lambda-1).
\end{equation}
elements.

Therefore, the dimension of $\mathfrak F_{n+\pmb\lambda}/\mathfrak F_{n+\pmb\lambda-1}$ is bounded above by the sum of \eqref{dim_AC4}, \eqref{dim_AC3},\eqref{dim_AB2} and \eqref{dim_CD2}, i.e.,
\begin{equation}\label{vnk_ineq}
v_{n+\pmb\lambda}\ =\ \dim\Big(\frac{\mathfrak F_{n+\pmb\lambda}}{\mathfrak F_{n+\pmb\lambda-1}}\Big)\ \pmb\le\ 
\frac{1}{2}(n+\pmb\lambda-1)(n-\pmb\lambda+4)(n-\pmb\lambda+5).
\end{equation}
Later on, we shell see that in \eqref{vnk_ineq} we have, actually, an equality. 

Let us observe that equations \eqref{Properties_mod_U} and \eqref{S_mod_2n} yield the identities
\begin{equation}\label{ad_prop_S}
\begin{cases}
{\mathcal S}^\ast_{\pmb{\alpha\beta\gamma\delta}}\ \equiv -2iY_{111[\pmb\alpha+\pmb\beta+\pmb\gamma+\pmb\delta-4]}
+\mathbb Im\big(Y_{111[\pmb\alpha+\pmb\beta+\pmb\gamma+\pmb\delta-2]}\big)\\
 {\mathcal S}^\ast_{\pmb{\alpha\beta\gamma\delta'}}\,\equiv Y_{111[\pmb\alpha+\pmb\beta+\pmb\gamma+\pmb\delta-3]}\\
 {\mathcal S}^\ast_{\pmb{\alpha\beta'\gamma\delta'}}\equiv -\mathbb Im\big(Y_{111[\pmb\alpha+\pmb\beta+\pmb\gamma+\pmb\delta-2]}\big)
\end{cases}
\quad
\mod{\mathfrak F_{2n}}
\end{equation}

Similarly, the vanishing of all one-forms $A_{\pmb{\alpha\beta\gamma}}$, $C_{\pmb{\alpha\beta\gamma}}$,  modulo $\mathfrak F_{2n}$, implies that
\begin{equation}\label{ad_prop_V}
\begin{cases}
{\mathcal V}^\ast_{\pmb{\alpha\beta\gamma}}\ \equiv -2i{\mathcal V}^\ast_{11[\pmb\alpha+\pmb\beta+\pmb\gamma-3]\pmb'}
-{\mathcal V}^\ast_{11\pmb'[\pmb\alpha+\pmb\beta+\pmb\gamma-2]'}\\
 {\mathcal V}^\ast_{\pmb{\alpha'\beta'\gamma'}}\equiv -{\mathcal V}^\ast_{11[\pmb\alpha+\pmb\beta+\pmb\gamma-2]\pmb'}
-2i{\mathcal V}^\ast_{11\pmb'[\pmb\alpha+\pmb\beta+\pmb\gamma-3]'}
\end{cases}
\quad
\mod{\mathfrak F_{2n}}
\end{equation}
and the vanishing of $A_{\pmb{\alpha\beta}}$, $C_{\pmb{\alpha\beta}}$ gives

\begin{equation}\label{ad_prop_LM}
\begin{cases}
{\mathcal L}^\ast_{\pmb{\alpha\beta'}}\ \equiv i\mathbb Re\big({\mathcal L}^\ast_{1[\pmb\alpha+\pmb\beta]}\big)\\
 {\mathcal M}^\ast_{\pmb{\alpha\beta}}\equiv -2i{\mathcal M}^\ast_{1[\pmb\alpha+\pmb\beta-2]\pmb'}
-{\mathcal M}^\ast_{1\pmb'[\pmb\alpha+\pmb\beta-1]'}
\end{cases}
\quad
\mod{\mathfrak F_{2n}}
\end{equation}

\subsection{The characters $v_{(2n+1)},v_{(2n+2)}$ and $v_{(2n+3)}$}\label{sec_2n+3}
The definition of $\mathfrak F_{2n+1}$ together with the identities \eqref{ad_prop_S}, \eqref{ad_prop_V}, \eqref{ad_prop_LM} and \eqref{Properties_mod_U} implies that the quotient space $\mathfrak F_{2n+1}/\mathfrak F_{2n}$ is generated by the real and imaginary parts of the one-forms:
\begin{equation}\label{3_5_Equations_I}
\begin{gathered}
{\mathcal V}^\ast_{11[\pmb\alpha-3]\pmb'}+\mathbb Im\big({\mathcal V}^\ast_{11\pmb'[\pmb\alpha-2]\pmb'}\big);\qquad
{\mathcal V}^\ast_{11[\pmb\alpha-3]\pmb'}+\mathbb Im\big({\mathcal V}^\ast_{11\pmb'[\pmb\alpha-2]\pmb'}\big)
+i\mathbb Im\big({\mathcal V}^\ast_{11[\pmb\alpha-1]\pmb'}\big);\\
\mathbb Im\big({\mathcal V}^\ast_{11[\pmb\alpha-2]\pmb'}\big);\qquad
\mathbb Re\big({\mathcal V}^\ast_{11[\pmb\alpha-2]\pmb'}\big)
+\mathbb Im\big({\mathcal V}^\ast_{11\pmb'[\pmb\alpha-1]\pmb'}\big);
\end{gathered}
\end{equation}
\begin{equation}\label{3_5_Equations_II}
\begin{gathered}
2{\mathcal M}^\ast_{1[\pmb\alpha-2]\pmb'}
-i{\mathcal M}^\ast_{1\pmb'[\pmb\alpha-1]\pmb'}+\mathbb Re\big({\mathcal L}^\ast_{1\pmb\alpha}\big);\qquad
2{\mathcal M}^\ast_{1[\pmb\alpha-2]\pmb'}
-i{\mathcal M}^\ast_{1\pmb'[\pmb\alpha-1]\pmb'}+i\mathbb Im\big({\mathcal L}^\ast_{1\pmb\alpha}\big)+{\mathcal M}^\ast_{1\pmb\alpha\pmb'};\\
-\overline{{\mathcal L}^\ast_{1[\pmb\alpha-1]}}
+{\mathcal M}^\ast_{1[\pmb\alpha-1]\pmb'};\qquad
\overline{{\mathcal L}^\ast_{1[\pmb\alpha-1]}}
+{\mathcal M}^\ast_{1[\pmb\alpha-1]\pmb'}+\mathbb Re\big({\mathcal L}^\ast_{1[\pmb\alpha+1]}\big)-i{\mathcal M}^\ast_{1\pmb'\pmb\alpha\pmb'};
\end{gathered}
\end{equation}
\begin{equation}\label{3_5_Equations_III}
\mathbb Re\big({\mathcal C}^\ast_{\pmb\alpha}\big);\qquad
\mathbb Im\big({\mathcal C}^\ast_{\pmb\alpha}\big)-\mathbb Re\big({\mathcal C}^\ast_{[\pmb\alpha+1]\pmb'}\big);\qquad
i{\mathcal C}^\ast_{\pmb\alpha\pmb'}+{\mathcal H}^\ast_{\pmb\alpha};\qquad
-i{\mathcal C}^\ast_{\pmb\alpha\pmb'}+{\mathcal H}^\ast_{\pmb\alpha}
-{\mathcal C}^\ast_{[\pmb\alpha+1]}-i{\mathcal H}^\ast_{[\pmb\alpha+1]\pmb'}.
\end{equation}

A brief inspection of the one forms in \eqref{3_5_Equations_I} shows that the linear span in $\mathfrak F_{2n+1}/\mathfrak F_{2n}$ of their real and imaginary parts can be generated by using only the real and imaginary parts of the first expression there. Similarly, for the the linear span of  the real and imaginary 
parts of the forms in \eqref{3_5_Equations_II}, we need only the real and imaginary parts of the forms which are in the first line, i.e.,
\begin{equation*}
2{\mathcal M}^\ast_{1[\pmb\alpha-2]\pmb'}
-i{\mathcal M}^\ast_{1\pmb'[\pmb\alpha-1]\pmb'}+\mathbb Re\big({\mathcal L}^\ast_{1\pmb\alpha}\big)\ \quad\text{and}\ \quad
2{\mathcal M}^\ast_{1[\pmb\alpha-2]\pmb'}
-i{\mathcal M}^\ast_{1\pmb'[\pmb\alpha-1]\pmb'}+i\mathbb Im\big({\mathcal L}^\ast_{1\pmb\alpha}\big)+{\mathcal M}^\ast_{1\pmb\alpha\pmb'}.
\end{equation*}
Observing also that the first two expressions in \eqref{3_5_Equations_III} correspond to one-forms that are real, we conclude that $\mathfrak F_{2n+1}/\mathfrak F_{2n}$ can be generated, over the real numbers, by using only
$
12n
$
one-forms, and thus
\begin{equation}\label{vnk1_ineq}
v_{2n+1}\le 12n.
\end{equation}

Furthermore, we obtain the relations 
\begin{equation}\label{eq_mod_2n+1}
\begin{cases}
{\mathcal V}^\ast_{11\pmb\alpha\pmb'}\equiv-\mathbb Im\big({\mathcal V}^\ast_{11\pmb'[\pmb\alpha+1]\pmb'}\big)\\
{\mathcal V}^\ast_{11\pmb\alpha}\ \equiv 
-\overline{{\mathcal V}^\ast_{11\pmb'\pmb\alpha\pmb'}}\\
 {\mathcal V}^\ast_{1\pmb'1\pmb'\pmb\alpha\pmb'}\equiv -\mathbb Im\big({\mathcal V}^\ast_{11\pmb'[\pmb\alpha+1]\pmb'}\big)
-2i{\mathcal V}^\ast_{11\pmb'[\pmb\alpha-1]\pmb'}\\
{\mathcal M}^\ast_{1\pmb\alpha}\equiv i\mathbb Re\big({\mathcal L}^\ast_{1\pmb\alpha}\big)\\
 {\mathcal M}^\ast_{1\pmb\alpha\pmb'}\equiv \overline{{\mathcal L}^\ast_{1\pmb\alpha}}\\
{\mathcal M}^\ast_{1\pmb'\pmb\alpha\pmb'}\equiv-2i\overline{{\mathcal L}^\ast_{1[\pmb\alpha-1]}}-i\mathbb Re\big({\mathcal L}^\ast_{1[\pmb\alpha+1]}\big)\\
{\mathcal L}^\ast_{1\pmb\alpha\pmb'}\equiv i\mathbb Re\big({\mathcal L}^\ast_{1[\pmb\alpha+1]}\big)\\
{\mathcal C}^\ast_{\pmb\alpha}\equiv i\mathbb Re\big({\mathcal C}^\ast_{[\pmb\alpha+1]\pmb'}\big)\\
{\mathcal H}^\ast_{\pmb\alpha}\equiv -i{\mathcal C}^\ast_{\pmb\alpha\pmb'}\\
{\mathcal H}^\ast_{\pmb\alpha\pmb'}\equiv -2{\mathcal C}^\ast_{[\pmb\alpha-1]\pmb'}-\mathbb Re\big({\mathcal C}^\ast_{[\pmb\alpha+1]\pmb'}\big)
\end{cases}
\quad
\mod{\mathfrak F_{2n+1}}
\end{equation}

By \eqref{eq_mod_2n+1}, the quotient space $\mathfrak F_{2n+2}/\mathfrak F_{2n+1}$ is generated by the real and imaginary parts of
\begin{equation}
{\mathcal V}^\ast_{11\pmb'\pmb\alpha\pmb'},\qquad {\mathcal L}^\ast_{1\pmb\alpha},\qquad
{\mathcal C}^\ast_{\pmb{\alpha\pmb'}},\qquad i{\mathcal R}^\ast+{\mathcal Q}^\ast,\qquad {\mathcal P}^\ast+\overline{\mathcal P}^\ast
\end{equation}
and therefore,
\begin{equation}\label{vnk2_ineq}
v_{2n+2}\le 6n+3.
\end{equation}

The quotient space $\mathfrak F_{2n+3}/\mathfrak F_{2n+2}$ is generated by the real and imaginary parts of the one-forms:
\begin{equation*}
\qquad {\mathcal P}^\ast-\overline{\mathcal P}^\ast;\qquad {\mathcal R}^\ast;\\
\end{equation*}
\begin{equation*}
{\mathcal V}^\ast_{\pmb{\alpha\beta\gamma'}}+\overline{{\mathcal V}^\ast_{\pmb{\alpha'\beta'\gamma'}}}
-{\mathcal S}^\ast_{\pmb{\alpha\beta\gamma}1}-{\mathcal S}^\ast_{\pmb{\alpha\beta\gamma'}1\pmb'}+i{\mathcal S}^\ast_{\pmb{\alpha\beta\gamma}[3]}-i{\mathcal S}^\ast_{\pmb{\alpha\beta\gamma'}[3]\pmb'}
\equiv2iY_{111[\pmb\alpha-3]}+2Y_{111[\pmb\alpha-1]}+2i\mathbb Im\big(Y_{111[\pmb\alpha+1]}\big);
\end{equation*}
\begin{multline*}
-i{\mathcal V}^\ast_{\pmb{\alpha\beta\gamma'}}+i\overline{{\mathcal V}^\ast_{\pmb{\alpha'\beta'\gamma'}}}-{\mathcal V}^\ast_{\pmb{\alpha\beta}[{\pmb{\gamma}+1}]}-\overline{{\mathcal V}^\ast_{\pmb{\alpha'\beta'}[{\pmb{\gamma}+1}]}}
+i{\mathcal S}^\ast_{\pmb{\alpha\beta\gamma'}1\pmb'}-i{\mathcal S}^\ast_{\pmb{\alpha\beta\gamma}1}
+{\mathcal S}^\ast_{\pmb{\alpha\beta}[{\pmb{\gamma}}+1]1\pmb'}-{\mathcal S}^\ast_{\pmb{\alpha\beta}[{\pmb{\gamma}+1}]\pmb'1}\\
-{\mathcal S}^\ast_{\pmb{\alpha\beta\gamma}[3]}
-{\mathcal S}^\ast_{\pmb{\alpha\beta\gamma'}[3]\pmb'}+i{\mathcal S}^\ast_{\pmb{\alpha\beta}[{\pmb{\gamma}+1}][3]\pmb'}
+i{\mathcal S}^\ast_{\pmb{\alpha\beta}[3][{\pmb{\gamma}+1}]\pmb'}\\
\equiv -2Y_{111[\pmb\alpha-3]}+2iY_{111[\pmb\alpha-1]}-2i\mathbb Im\big(Y_{111[\pmb\alpha-1]}\big)+2iY_{111[\pmb\alpha+1]};
\end{multline*}
\begin{equation*}
{\mathcal V}^\ast_{\pmb{\alpha\beta'\gamma'}}-\overline{{\mathcal V}^\ast_{\pmb{\alpha'\beta\gamma'}}}
-{\mathcal S}^\ast_{\pmb{\alpha\beta'\gamma}1}+\overline{{\mathcal S}^\ast_{\pmb{\alpha'\beta\gamma}1}}
+i{\mathcal S}^\ast_{\pmb{\alpha\beta'\gamma}[3]}+i\overline{{\mathcal S}^\ast_{\pmb{\alpha'\beta\gamma}[3]}}\ 
\equiv\ 2i\Big(\mathbb Re\big(Y_{111\pmb\alpha}\big)-\mathbb Im\big(Y_{111[\pmb\alpha-2]}\big)\Big);
\end{equation*}
\begin{multline*}
-i{\mathcal V}^\ast_{\pmb{\alpha\beta'\gamma'}}-i\overline{{\mathcal V}^\ast_{\pmb{\alpha'\beta\gamma'}}}
-{\mathcal V}^\ast_{\pmb{\alpha\beta'}[{\pmb{\gamma}+1}]}+\overline{{\mathcal V}^\ast_{\pmb{\alpha'\beta}[{\pmb{\gamma}+1}]}}
-i\overline{{\mathcal S}^\ast_{\pmb{\alpha'\beta\gamma}1}}-i{\mathcal S}^\ast_{\pmb{\alpha\beta'\gamma}1}
+{\mathcal S}^\ast_{\pmb{\alpha\beta'}[{\pmb{\gamma}}+1]1\pmb'}-{\mathcal S}^\ast_{\pmb{\alpha\beta'}[{\pmb{\gamma}+1}]\pmb'1}\\
-{\mathcal S}^\ast_{\pmb{\alpha\beta'\gamma}[3]}
+\overline{{\mathcal S}^\ast_{\pmb{\alpha'\beta\gamma}[3]}}+i{\mathcal S}^\ast_{\pmb{\alpha\beta'}[{\pmb{\gamma}+1}][3]\pmb'}
+i{\mathcal S}^\ast_{\pmb{\alpha\beta'}[3][{\pmb{\gamma}+1}]\pmb'}\\
\equiv -2i\Big(\mathbb Re\big(Y_{111[\pmb\alpha-2]}\big)+\mathbb Im\big(Y_{111\pmb\alpha}\big)+\mathbb Im\big(Y_{111[\pmb\alpha+2]}\big)\Big).
\end{multline*}

It is easy to observe that we can choose as generators the $2n+2$ one-forms 
\begin{equation*}
\qquad {\mathcal P}^\ast-\overline{\mathcal P}^\ast,\qquad {\mathcal R}^\ast,\qquad \mathbb Re\big(Y_{111\pmb\alpha}\big)-\mathbb Im\big(Y_{111[\pmb\alpha-2]}\big),
\qquad \mathbb Re\big(Y_{111[\pmb\alpha-2]}\big)+\mathbb Im\big(Y_{111\pmb\alpha}\big)+\mathbb Im\big(Y_{111[\pmb\alpha+2]}\big),
\end{equation*}
and thus 
\begin{equation}\label{vnk3_ineq}
v_{2n+3}\le 2n+2.
\end{equation}

Since the system of equations 
\begin{equation}\label{equation_Y_Y_Y}
 \mathbb Im\big(Y_{111[\pmb\alpha-4]}\big)+\mathbb Im\big(Y_{111\pmb\alpha}\big)+\mathbb Im\big(Y_{111[\pmb\alpha+2]}\big)\equiv 0 \mod{ \mathfrak F_{2n+3}}
\end{equation}
is non-degenerate (cf. the technical Lemma~\ref{tech_lemma} below), we obtain that $Y_{111\pmb\alpha}\in \mathfrak F_{2n+3}$ and thus $\mathfrak F_{2n+3}$ is just the free vector space generated by the real and imaginary parts of all the one forms 
$
 \mathcal S^\ast_{\alpha\beta\gamma\delta}$, $\mathcal V^\ast_{\alpha\beta\gamma}$, $\mathcal  L^\ast_{\alpha\beta}$, $ \mathcal  M^\ast_{\alpha\beta}$, $\mathcal C^\ast_\alpha$, $\mathcal H^\ast_\alpha$, $\mathcal P^\ast, \  \mathcal Q^\ast,\  \mathcal R^\ast.
$
Therefore,  $\dim{\mathfrak F_{2n+3}}=d_2$, where $d_2$ is given by \eqref{def_d_2}). We have also 
\begin{equation*}
\mathfrak F_{(2n+3)}=\mathfrak F_{(2n+4)}=\dots=\mathfrak F_{d_1}
\end{equation*} 
and thus $v_{2n+4}\ =\ v_{2n+5}\ =\ \dots\ =\  v_{d_1}\ =\ 0$, i.e., non-zero characters may appear only among $v_2,\dots,v_{2n+3}$.

Since $\mathfrak F_n\subset \mathfrak F_{n+1}\subset\dots\subset \mathfrak F_{2n}\subset \mathfrak F_{2n+1}\subset \mathfrak F_{2n+2}\subset \mathfrak F_{2n+1}$, we obtain, by using the inequalities \eqref{vnk_ineq}, \eqref{vnk1_ineq},\eqref{vnk2_ineq} and \eqref{vnk3_ineq}, that
\begin{multline*}
d_2=\dim{\mathfrak F_{2n+3}}=\dim{\mathfrak F_n}+\dim{\frac{\mathfrak F_{n+1}}{\mathfrak F_{n}}} +\dots+\dim{\frac{\mathfrak F_{2n+3}}{\mathfrak F_{2n+2}}}
\ =\  \dim{\mathfrak F_n}\ +\ v_{n+1}\ +\ \dots\ +\ v_{2n+3}\\
\pmb\le\  \frac{1}{24}n(n-1)(11n^2+61n+86) \\ 
+\ \sum_{\pmb\lambda=1}^n \frac{1}{2}(n+\pmb\lambda-1)(n-\pmb\lambda+4)(n-\pmb\lambda+5)\  +\ 12n\ +\ (6n+3)\ +\ (2n+2).  
\end{multline*}

Computing the sum of the terms on the RHS of the above inequality produces again the number $d_2$. This implies that each of the inequalities  \eqref{vnk_ineq}, \eqref{vnk1_ineq},\eqref{vnk2_ineq}, \eqref{vnk3_ineq} is actually an equality and thus, we have shown
\begin{equation}\label{all_charachters}
\begin{cases}
v_{\pmb\lambda}=\frac{1}{2}(\pmb\lambda-1)(\pmb\lambda-2n-4)(\pmb\lambda-2n-5)\\
v_{n+\pmb\lambda}\  =\ 
\frac{1}{2}(n+\pmb\lambda-1)(n-\pmb\lambda+4)(n-\pmb\lambda+5)\\
v_{2n+1}\ =\  12n\\
v_{2n+2}\ =\  6n+3\\
v_{2n+3}\ =\  2n+2\\
v_{2n+4}\ =\ v_{2n+5}\ =\ \dots\ =\  v_{d_1}\ =\ 0.
\end{cases}
\end{equation}

\subsection{A technical lemma} Here we give a proof to an algebraic lemma which is used in Section~\ref{sec_2n+3} to show that the system of equations \eqref{equation_Y_Y_Y} yields $$\mathbb Im\big(Y_{111\pmb\alpha}\big)\equiv 0\mod{\mathfrak F_{2n+3}}.$$

\begin{lemma}\label{tech_lemma}
Let $\mathbb Z_n=\{0,1,\dots,n-1\}$ be the least residue system modulo $n$. If $f: \mathbb Z_n\rightarrow \mathbb C$ is any function satisfying
\begin{equation}\label{equations_for_fk}
f(k)+f(k+4)+f(k+6)=0,\qquad \forall k\in\mathbb Z_n
\end{equation}
then, necessarily, $f=0$.
\end{lemma}
\begin{proof} We consider the values $f(1),\dots,f(n)$ as unknown variables $x_1,\dots,x_n$.
For each $k\in\mathbb N$, we let 
\begin{equation}
Q_{[k]}\overset{def}{\ =\ }x_{[k]}+x_{[k+4]}+x_{[k+6]},
\end{equation} 
where, by following the conventions adopted in \ref{extra_conv}, we use indices enclosed in square brackets to indicate that their values are considered modulo $n$.  Then, in order to proof the lemma, we need to show that the system of linear equations
\begin{equation}\label{system_linear eq_Q}
Q_1=Q_2=\dots= Q_n=0
\end{equation}
 is non-degenerate.

Let us define the sequence of numbers $a_1,\dots,a_k,\dots$ by the recurrence relation 
\begin{equation}\label{def_recurrence}
a_k+a_{k+1}+a_{k+3}=0,\qquad a_1=1,\ a_2=0,\ a_3=-1.
\end{equation}
Then a small combinatorial calculation shows that, for each $m\in\mathbb N$, we have the identity
\begin{equation}\label{general_equation_Qk}
\sum_{k=1}^{m}a_kQ_{[2k-1]}=x_1-a_{m+1}x_{[2m+1]}-a_{m+2}x_{[2m+3]}+a_{m}x_{[2m+5]}
\end{equation}
and, similarly, 
\begin{equation}\label{general_2_Qk}
\begin{split}
\sum_{k=1}^{m}a_kQ_{[2k+1]}=x_3-a_{m+1}x_{[2m+3]}-a_{m+2}x_{[2m+5]}+a_{m}x_{[2m+7]},\\
\sum_{k=1}^{m}a_kQ_{[2k+3]}=x_5-a_{m+1}x_{[2m+5]}-a_{m+2}x_{[2m+7]}+a_{m}x_{[2m+9]},
\end{split}
\end{equation}

Setting $m=n$ into \eqref{general_equation_Qk}, we obtain that \eqref{system_linear eq_Q} yields the equation
\begin{equation}\label{x_135_first}
(1-a_{n+1})x_1-a_{n+2}x_{3}-a_nx_5=0.
\end{equation}

Similarly, setting $m=n-1$ into the first equation of \eqref{general_2_Qk}, and $m=n-2$ into the second, we get, respectively,
\begin{equation}\label{x_135_second_and-third}
\begin{split}
-a_nx_1+(1-a_{n+1})x_{3}-a_{n-1}x_5&=0\\
-a_{n-1}x_1-a_{n}x_{3}+(1+a_{n-2})x_5&=0.
\end{split}
\end{equation}

Next we show that the determinant   
\begin{equation}\label{determinant_3x3}
\begin{gathered}
\begin{vmatrix}
1-a_{n+1}&-a_{n+2}&-a_n\\
-a_n&1-a_{n+1}&-a_{n-1}\\
-a_{n-1}&-a_{n}&1+a_{n-2}\\
\end{vmatrix}=a_n^3-2a_n a_{n-1}a_{n+1}-a_n a_{n-2}a_{n+2}+a_{n-1}^2 a_{n+2}\\+a_{n-2}a_{n+1}^2+2a_n a_{n-1}-a_n a_{n+2}
-2 a_{n-2} a_{n+1}+a_{n+1}^2+a_{n-2}-2a_{n+1}+1
\end{gathered}
\end{equation}
is never vanishing.  Indeed, consider the three different roots $z_1$, $z_2$, $z_3$ of the polynomial $z^3+z+1=0$ and take $c_1$, $c_2$, $c_3$ to be the unique complex numbers satisfying 
\begin{equation}
\begin{cases}
c_1z_1+c_2z_2+c_3z_3=1\\
c_1z_1^2+c_2z_2^2+c_3z_3^2=0\\
c_1z_1^3+c_2z_2^3+c_3z_3^3=-1.
\end{cases}
\end{equation}
Then, the solution of the recurrence relation\eqref{def_recurrence} has the form 
$
a_k=c_1z_1^k+c_2z_2^k+c_3z_3^k.
$ Substituting back into \eqref{determinant_3x3} and using the Vieta's formulae, we obtain that   
\begin{equation}
\begin{vmatrix}
1-a_{n+1}&-a_{n+2}&-a_n\\
-a_n&1-a_{n+1}&-a_{n-1}\\
-a_{n-1}&-a_{n}&1+a_{n-2}\\
\end{vmatrix}=(z_1^n-1)(z_2^n-1)(z_3^n-1),
\end{equation}
which is a never vanishing number, since non of the roots of the polynomial $z^3+z+1=0$ has a unite  norm. Therefore, the linear equations \eqref{x_135_first} and \eqref{x_135_second_and-third} have a unique solution $x_1=x_3=x_5=0$. Since the system \eqref{system_linear eq_Q} is invariant under cyclic permutations of the indices of its variables, this is enough to conclude that it is non-degenerate. 
\end{proof}

\subsection{Main theorem}
Now, we are in a position to check that the Cartan's test (cf. \eqref{Cartan_test_general}) is satisfied for our exterior differential system. Indeed, using \eqref{all_charachters}, we compute
\begin{multline}\label{CompCartanTest}
\sum_{\pmb\lambda=1}^n \Big(\pmb\lambda\, v_{\pmb\lambda}+(n+\pmb\lambda)v_{n+\pmb\lambda}\Big)  +\ (2n+1)v_{2n+1}\ +\ (2n+2)v_{2n+2}\ +\ (2n+3)v_{2n+3}\\
=\ \frac{2}{15}(2n+5)(2n+3)(n+3)(n+2)(n+1).
\end{multline}
The number on the RHS above is equal to the constant $D$ determined by \eqref{compute_D}. Therefore, the system is in involution.

\begin{thrm}\label{main} Assume we are given some arbitrary complex numbers  \begin{equation}\label{FixedComplexNumbers}
 \begin{split}
 \ &  \mathcal S^\circ_{\alpha\beta\gamma\delta},  \mathcal V^\circ_{\alpha\beta\gamma},\ \mathcal  L^\circ_{\alpha\beta},\ \mathcal  M^\circ_{\alpha\beta}, \ \mathcal C^\circ_\alpha, \  \mathcal H^\circ_\alpha,\  \mathcal P^\circ, \  \mathcal Q^\circ,\  \mathcal R^\circ\\
 \mathcal A^\circ_{\alpha\beta\gamma\delta\epsilon},\ & \mathcal B^\circ_{\alpha\beta\gamma\delta},\  \mathcal C^\circ_{\alpha\beta\gamma\delta} ,\ \mathcal  D^\circ_{\alpha\beta\gamma},\ \mathcal  E^\circ_{\alpha\beta\gamma},
 \ \mathcal  F^\circ_{\alpha\beta\gamma},\ \mathcal  G^\circ_{\alpha\beta},\ \mathcal  X^\circ_{\alpha\beta},\ \mathcal  Y^\circ_{\alpha\beta},\ \mathcal  Z^\circ_{\alpha\beta}, \\
  &\ (\mathcal N^\circ_1)_\alpha,\ (\mathcal N^\circ_2)_\alpha,\ (\mathcal N^\circ_3)_\alpha,\ (\mathcal N^\circ_4)_\alpha,\ (\mathcal N^\circ_5)_\alpha,\ \mathcal U^\circ_s,\  \mathcal W^\circ_s
\end{split}
\end{equation}
that depend totally symmetrically on the indices $1\le\alpha,\beta,\gamma,\delta\le 2n$  and satisfy the relations
\begin{equation*} 
 \begin{cases}
 (\mathfrak j{\mathcal S^\circ})_{\alpha\beta\gamma\delta}=\mathcal S^\circ_{\alpha\beta\gamma\delta}\\
 (\mathfrak j{\mathcal L^\circ})_{\alpha\beta}=\mathcal L^\circ_{\alpha\beta}\\
 \overline {\mathcal R^\circ}=\mathcal R^\circ.
 \end{cases}.
 \end{equation*}
Then, there exists a real analytic qc structure defined in a neighborhood $\Omega$ of $0\in R^{4n+3}$ such that for some point $u\in P_1$ with $\pi_o(\pi_1(u))=0$ (here, we keep the notation $\pi_1: P_1\rightarrow P_0$ and $\pi_o: P_o\rightarrow \Omega$ for the naturally associated to the qc structure of $\Omega$ principle bundles, as defined in section Section~\ref{Can_Cartan}),  the curvature functions \eqref{CurvatureFunctions} and their covariant derivatives \eqref{solution_space_coord} take at $u$ values given by the corresponding  complex numbers \eqref{FixedComplexNumbers}.     

Furthermore, the generality of the real analytic qc structures in $4n+3$ dimensions is given by $2n+2$ real analytic functions of $2n+3$ variables.
\end{thrm}
\begin{proof}
By the computation \eqref{CompCartanTest}, we have shown that the Cartan's
test is satisfied at the origin $o\in N$ (cf.  \eqref{ProductN}) for the
chosen integral element $E\subset T_oN$ which we have determined by the
equations \eqref{DefE}.  In order to prove the theorem, we need to extend
this computation to a slightly more general situation.  Namely, let us
consider the point 
%$$
$p\overset{def}{=}\Big(\text{id}, 0, (\mathcal
S^\circ_{\alpha\beta\gamma\delta}, \mathcal V^\circ_{\alpha\beta\gamma},\
\mathcal L^\circ_{\alpha\beta},\ \mathcal M^\circ_{\alpha\beta}, \ \mathcal
C^\circ_\alpha, \ \mathcal H^\circ_\alpha,\ \mathcal P^\circ, \ \mathcal
Q^\circ,\ \mathcal R^\circ)\Big)\in N$
%$$
and define the one-forms
{\allowdisplaybreaks
\begin{equation*}
 \begin{split}
\widehat{\mathcal S^\ast}_{\alpha\beta\gamma\delta}&\overset{def}{\ =\ }\mathcal S^\ast_{\alpha\beta\gamma\delta}-\Bigg\{ 
\mathcal A^\circ_{\alpha\beta\gamma\delta\epsilon}\,\theta^{\epsilon}-\pi^{\sigma}_{\bar\epsilon}(\mathfrak j\mathcal A^\circ)_{\alpha\beta\gamma\delta\sigma}\,\theta^{\bar\epsilon}
 +\Big(\mathcal B^\circ_{\alpha\beta\gamma\delta}+(\mathfrak j \mathcal B^\circ)_{\alpha\beta\gamma\delta}\Big)\eta_1 + i\mathcal C^\circ_{\alpha\beta\gamma\delta}\big(\eta_2+i\eta_3\big)\\
&\qquad\qquad\qquad\qquad\qquad\qquad\qquad\qquad\qquad\qquad\qquad\qquad\qquad - i(\mathfrak j\mathcal C^\circ)_{\alpha\beta\gamma\delta}\big(\eta_2-i\eta_3\big)\Bigg\}\\
\widehat{\mathcal V^\ast}_{\alpha\beta\gamma}&\overset{def}{\ =\ }{\mathcal V^\ast}_{\alpha\beta\gamma}-\Bigg\{\mathcal C^\circ_{\alpha\beta\gamma\epsilon}\,\theta^{\epsilon}+\pi^{\sigma}_{\bar\epsilon}\,\mathcal B^\circ_{\alpha\beta\gamma\sigma}\,\theta^{\bar\epsilon}
 +\mathcal D^\circ_{\alpha\beta\gamma}\eta_1 + \mathcal E^\circ_{\alpha\beta\gamma}\big(\eta_2+i\eta_3\big)+\mathcal F^\circ_{\alpha\beta\gamma}\big(\eta_2-i\eta_3\big)\Bigg\}\\
 \widehat{\mathcal L^\ast}_{\alpha\beta}&\overset{def}{\ =\ }{\mathcal L^\ast}_{\alpha\beta}-\Bigg\{-
(\mathfrak j\mathcal F^\circ)_{\alpha\beta\epsilon}\,\theta^{\epsilon}\!-\!\pi^{\sigma}_{\bar\epsilon}\, \mathcal F^\circ_{\alpha\beta\sigma}\,\theta^{\bar\epsilon}
 +i\Big((\mathfrak j\mathcal Z^\circ)_{\alpha\beta}-\mathcal Z^\circ_{\alpha\beta}\Big)\eta_1 +i \mathcal G^\circ_{\alpha\beta}\big(\eta_2+i\eta_3\big)-i(\mathfrak j\mathcal G^\circ)_{\alpha\beta}\big(\eta_2-i\eta_3\big)\Bigg\}\\
  \widehat{\mathcal M^\ast}_{\alpha\beta}&\overset{def}{\ =\ }{\mathcal M^\ast}_{\alpha\beta}-\Bigg\{-\mathcal E^\circ_{\alpha\beta\epsilon}\,\theta^{\epsilon}+\pi^{\sigma}_{\bar\epsilon}\Big((\mathfrak j\mathcal F^\circ)_{\alpha\beta\sigma}-i\mathcal D^\circ_{\alpha\beta\sigma}\Big)\,\theta^{\bar\epsilon}
 +\mathcal X^\circ_{\alpha\beta}\eta_1 + \mathcal Y^\circ_{\alpha\beta}\big(\eta_2+i\eta_3\big)+\mathcal Z^\circ_{\alpha\beta}\big(\eta_2-i\eta_3\big)\Bigg\}\\
  \widehat{\mathcal C^\ast}_{\alpha}&\overset{def}{\ =\ }{\mathcal C^\ast}_{\alpha}-\Bigg\{\mathcal G^\circ_{\alpha\epsilon}\,\theta^{\epsilon}-i\pi^{\sigma}_{\bar\epsilon}\mathcal Z^\circ_{\alpha\sigma}\,\theta^{\bar\epsilon}
 +(\mathcal N^\circ_1)_{\alpha}\eta_1 + (\mathcal N^\circ_2)_{\alpha}\big(\eta_2+i\eta_3\big)+(\mathcal N^\circ_3)_{\alpha}\big(\eta_2-i\eta_3\big)\Bigg\}\\
  \widehat{\mathcal H^\ast}_{\alpha}&\overset{def}{\ =\ }{\mathcal H^\ast}_{\alpha} -\Bigg\{-\mathcal Y^\circ_{\alpha\epsilon}\,\theta^{\epsilon}+i\pi^{\sigma}_{\bar\epsilon}\big(\mathcal G^\circ_{\alpha\sigma}-\mathcal X^\circ_{\alpha\sigma}\big)\theta^{\bar\epsilon}
 +(\mathcal N^\circ_4)_{\alpha}\eta_1 + (\mathcal N^\circ_5)_{\alpha}\big(\eta_2+i\eta_3\big)\\
 &\qquad\qquad\qquad\qquad\qquad\qquad\qquad\qquad\qquad\qquad\qquad+\Big((\mathcal N^\circ_1)_{\alpha}+i\pi^{\bar\sigma}_{\alpha}(\mathcal N^\circ_3)_{\bar\sigma}\Big)\big(\eta_2-i\eta_3\big)\Bigg\}\\
  \widehat{\mathcal R^\ast}&\overset{def}{\ =\ } \mathcal R^\ast-\Bigg\{4\pi^{\bar\sigma}_{\epsilon}(\mathcal N^\circ_3)_{\bar\sigma}\,\theta^{\epsilon}+4\pi^{\sigma}_{\bar\epsilon}(\mathcal N^\circ_3)_{\sigma}\,\theta^{\bar\epsilon}
 +i\big(\mathcal U^\circ_3-\overline{\mathcal U^\circ}_3\big)\eta_1 -i\big(\mathcal U^\circ_1+\mathcal W^\circ_3\big)\big(\eta_2+i\eta_3\big)\\
 &\qquad\qquad\qquad\qquad\qquad\qquad\qquad\qquad\qquad\qquad\qquad\qquad+i\big(\overline{\mathcal U^\circ}_1+\overline{\mathcal W^\circ}_3\big)\big(\eta_2-i\eta_3\big)\Bigg\}\\
  \widehat{\mathcal P^\ast}&\overset{def}{\ =\ }\mathcal P^\ast-\Bigg\{-4(\mathcal N^\circ_2)_{\epsilon}\,\theta^{\epsilon}-4\Big((\mathcal N^\circ_3)_{\bar\epsilon}+i\pi^{\sigma}_{\bar\epsilon}(\mathcal N^\circ_1)_{\sigma}\Big)\,\theta^{\bar\epsilon}
 +\mathcal U^\circ_1\eta_1 + \mathcal U^\circ_2\big(\eta_2+i\eta_3\big)+\mathcal U^\circ_3\big(\eta_2-i\eta_3\big)\Bigg\}\\
  \widehat{\mathcal Q^\ast}&\overset{def}{\ =\ }\mathcal Q^\ast-\Bigg\{4(\mathcal N^\circ_5)_{\epsilon}\,\theta^{\epsilon}+4i\pi^{\sigma}_{\bar\epsilon}\Big((\mathcal N^\circ_2)_{\sigma}+(\mathcal N^\circ_4)_{\sigma}\Big)\,\theta^{\bar\epsilon}
 +\mathcal W^\circ_1\eta_1 + \mathcal W^\circ_2\big(\eta_2+i\eta_3\big)+\mathcal W^\circ_3\big(\eta_2-i\eta_3\big)\Bigg\}.
 \end{split}
\end{equation*}
}

Since we have used here the formulae \eqref{FormulaeForCovDer}, we know (by Proposition~4.3 from \cite{MS}) that the three-forms $\Delta_{\alpha\beta}$, $\Delta_{\alpha}$ and $\Psi_s$ given by \eqref{Form_d2Gamma_ab}, \eqref{Form_d2phi_a}, \eqref{Form_d2psi_1} and \eqref{Form_d2psi_2} on $N$ will remain the same if replacing everywhere in the respective expressions the one-forms
\begin{equation*}
 \mathcal S^\ast_{\alpha\beta\gamma\delta},\  \mathcal V^\ast_{\alpha\beta\gamma},\ \mathcal  L^\ast_{\alpha\beta},\ \mathcal  M^\ast_{\alpha\beta}, \  \mathcal C^\ast_\alpha,\  \mathcal H^\ast_\alpha,\  \mathcal P^\ast, \ \mathcal Q^\ast,\   \mathcal R^\ast
\end{equation*} 	
by the one-forms
\begin{equation}\label{HatForms}
 \widehat{\mathcal S^\ast}_{\alpha\beta\gamma\delta},\  \widehat{\mathcal V^\ast}_{\alpha\beta\gamma},\ \widehat{\mathcal  L^\ast}_{\alpha\beta},\ \widehat{\mathcal  M^\ast}_{\alpha\beta}, \  \widehat{\mathcal C^\ast}_\alpha,  \ \widehat{\mathcal H^\ast}_\alpha,\  \widehat{\mathcal P^\ast}, \ \widehat{\mathcal Q^\ast},\   \widehat{\mathcal R^\ast}.
\end{equation} 
Therefore, if we repeat our computation from above for the new integral element $\widehat E\subset T_pN$, which is now determined by the vanishing of \eqref{HatForms}, 
%\begin{equation*}
 %\widehat{\mathcal S^\ast}_{\alpha\beta\gamma\delta}\ =\  \widehat{\mathcal V^\ast}_{\alpha\beta\gamma}\ =\ \widehat{\mathcal  L^\ast}_{\alpha\beta}\ =\ 
 %\widehat{\mathcal  M^\ast}_{\alpha\beta} \ =\ \widehat{\mathcal C^\ast}_\alpha \ =\  \widehat{\mathcal H^\ast}_\alpha\ =\  \widehat{\mathcal P^\ast} \  =\ 
 %%\widehat{\mathcal Q^\ast}\  =\ \widehat{\mathcal R^\ast}\ =\ 0,
%\end{equation*}
we will end up with the same result: The character sequence of $\mathcal I$  will be given again by the formulae \eqref{all_charachters} and the Cartan's test will remain satisfied. Thus $\widehat E$ must be again a Cartan-ordinary element of $\mathcal I$ and the Theorem follows by the Cartan's Third Theorem, as explained in Section~\ref{ex_dif_sys}.  
\end{proof}

%%%%%%%%%%%%%%%%%%%%

%\section{Appendix}

%%%%%%%%%%%%%%

\end{document}